\documentclass[11pt,a4paper,reqno]{amsart}
\usepackage[a4paper, left = 2cm,  hmargin={25mm,25mm},vmargin={25mm,25mm}]{geometry}
\usepackage[utf8]{inputenc}
\usepackage{geometry}
\usepackage[english]{babel}
\usepackage{graphicx}
\usepackage{color}
\usepackage{xcolor}
\usepackage{pdfpages}
\usepackage{amsbsy}
\usepackage{amssymb}
\usepackage{amsmath}
\usepackage{bm}
\usepackage{mathtools}
\usepackage{matlab-prettifier}
\usepackage{amsthm}
\usepackage{tikz-cd}
\usepackage{amsfonts}
\usepackage{array}
\usepackage{hyperref}
\usepackage[OT2,T1]{fontenc}
\usepackage{csquotes}
\usepackage{float}

\newcommand\m[1]{\begin{pmatrix}#1\end{pmatrix}} 

\everymath{\displaystyle}
\newtheorem{theorem}{Theorem}[section]

\theoremstyle{definition}
\newtheorem{lemma}[theorem]{Lemma}

\theoremstyle{definition}
\newtheorem{defn}[theorem]{Definition}

\theoremstyle{definition}
\newtheorem{proposition}[theorem]{Proposition}

\newtheorem{remark}{Remark}
\newtheorem*{notation}{Notation}

\newtheorem{corollary}[theorem]{Corollary}

\newcommand{\Mod}[1]{\ (\mathrm{mod}\ #1)}

\definecolor{lightgray}{gray}{0.5}
\usepackage[backend=biber, maxnames=10]{biblatex}
\begin{document}
\title[On a Rankin-Selberg integral of three Hermitian cusp forms] 
{On a Rankin-Selberg integral of three Hermitian cusp forms}
\author{Thanasis Bouganis and Rafail Psyroukis}
\address{Department of Mathematical Sciences\\ Durham University\\
South. Rd.\\ Durham, DH1 3LE, U.K..}
\email{athanasios.bouganis@durham.ac.uk, rafail.psyroukis@durham.ac.uk}
\maketitle 
\vspace{-1cm}
\begin{abstract}
    Let $K = \mathbb{Q}(i)$. We study the Petersson inner product of a Hermitian Eisenstein series of Siegel type on the unitary group $U_{5}(K)$, diagonally-restricted on $U_2(K)\times U_2(K)\times U_1(K)$, against two Hermitian cuspidal eigenforms $F, G$ of degree $2$ and an elliptic cuspidal eigenform $h$ (seen as a Hermitian modular form of degree 1), all having weight $k \equiv 0 \pmod 4$. We obtain, through this consideration, an integral representation of a certain Dirichlet series, together with an additional residue term. By taking $F$ to belong in the Maass space, we are able to show that the Dirichlet series possesses an Euler product. Moreover, its $p$-factor for an inert prime $p$ can be essentially identified with the
    twist by $h$ of a degree six Euler factor attached to $G$ by Gritsenko. 
    The question of whether the same holds for the primes that split remains unanswered here, even though we make considerable steps in that direction too. Our paper is inspired by a work of Heim, who considered a similar question in the case of Siegel modular forms.

\end{abstract}
\tableofcontents
\vspace{-1cm}
\section{Introduction} 

The spinor $L$-function attached to a cuspidal (holomorphic) Siegel eigenform $G$ of degree two has been an object of intense study in the literature. It was the seminal work of Andrianov in \cite{andrianov2} who first obtained an integral representation of such an $L$-function and from there derived a functional equation and established its analytic properties. However, the integral representation obtained by Andrianov does not allow one (or at least it is not known how) to obtain algebraicity properties of the critical values. The difficulty seems to be related to the fact that the integral representation involves Eisenstein series which are defined over symmetric spaces which do not have a structure of a Shimura variety. Later work, such as the one of Kohnen and Skoruppa in \cite{kohnen_skoruppa}, obtained other integral representations using Eisenstein series over the Siegel upper half space but of weight zero and hence not holomorphic (or nearly holomorphic). One should also mention here that if on the other hand $G$ is taken non-holomorphic, there is recent work of Loeffler, Pilloni, Skinner and Zerbes in \cite[Theorem A]{LPSZ}, where they construct a $p$-adic measure which interpolates twists of such $L$-values.
\newline

It is perhaps surprising that a seemingly more complicated object, the twist of the spinor $L$-function of $G$ (from now on always assumed holomorphic) by an elliptic cusp form $h$ (that is $\textup{Sp}_4 \times \textup{GL}_2$) does afford an integral representation, which allows one to not only study analytic properties but also algebraic (and $p$-adic). Actually, there are (at least) two different such integral representations. The one obtained by Furusawa in \cite{furusawa}, uses an Eisenstein series over a unitary group and its restriction to $\textup{Sp}_4$. There is a series of works based on this idea, most notably by Saha in \cite{saha}, generalising the work of Furusawa. \newline

The other integral representation of the twist was obtained by Heim in \cite{heim} in his effort to answer a question posed by Garrett on the possibility of extending the doubling method to more copies of the group (here the symplectic). Indeed, Heim managed to show that the twisted spinor $L$-function can be obtained by integrating the restriction of a Siegel type Eisenstein series of degree $5$ to $\mathbb{H}_2 \times \mathbb{H}_2 \times \mathbb{H}_1$ against the (degree two) cusp form $G$, another degree two cusp form $F$ in the Maass space and an elliptic cusp form $h$, all being Hecke eigenforms for their corresponding Hecke algebras. His method was based on the factorisation of the polynomials defining the $p$-factors of the relevant $L$-functions in some parabolic Hecke rings. This integral expression was later exploited systematically by Böcherer and Heim in \cite{bocherer_heim_2}, in order to establish various algebraicity properties and lift various restrictions on the weights of the Siegel and elliptic modular forms by the use of differential operators. \newline

Shortly after the work of Andrianov on the spinor $L$-function in \cite{andrianov2}, Gritsenko, in a series of papers, extended Andrianov's approach of the use of parabolic Hecke algebras to the study of a degree 6 $L$-function attached to a cuspidal Hermitian eigenform of degree two, where the underlying imaginary quadratic field is taken to be the field of Gaussian numbers $K:=\mathbb{Q}(i)$. Indeed in \cite{gritsenko_zeta}, Gritsenko first defined such an $L$-function and in the later work of \cite{gritsenko}, he obtained the analogue construction of Kohnen and Skoruppa using the factorization approach. Both integral representations allowed him to obtain a functional equation and study the analytic properties. However, as in the case of the symplectic group, neither of the above integral representations could be used to derive algebraicity properties, due to the Eisenstein series involved (only of real analytic nature). \newline

In this paper, we ask whether the phenomenon observed in the case of the twisted symplectic spinor $L$-function carries over to the unitary one. Namely, whether the degree 6 $L$-function considered by Gritsenko, twisted by an elliptic cusp form $h$, affords an integral representation which allows one to study algebraic properties of the twisted $L$-function. Here, the elliptic cusp form $h$ is seen as a Hermitian form of degree one (i.e. of $U_1$). Given the similarities between $\textup{Sp}_4$ and $U_2$, we investigate the possibility of extending the idea of Heim to the unitary setting. That is, we study the Petersson inner product of a Hermitian Eisenstein series of Siegel type on the unitary group $U_{5}(K)$, diagonally-restricted on $U_2(K)\times U_2(K)\times U_1(K)$, against two Hermitian cusp forms $F, G$ of degree $2$ and a Hermitian cusp form $h$ of degree $1$, all being Hecke eigenforms for their corresponding Hecke algebras.\newline

Let us now briefly mention the main results in our paper. Firstly, through the above consideration, we obtain an integral representation of a certain Dirichlet series, together with an additional residue term (see Theorem \ref{integral_representation_theorem}). The Dirichlet series is analogous to the one obtained by Heim in \cite[Theorem 2.6]{heim}. The additional residue term is in itself very interesting and does not show up in the work of Heim; however, it is not explored in this paper either. We hope to study it in future work. By now taking $F$ to belong in the Maass space, we proceed to factor this Dirichlet series for every rational prime $p$. In particular, in Theorem \ref{Main Theorem, inert case}, we show that its $p$-factor for an inert prime $p$ can be essentially identified with the degree $12$ $p$-factor of $Z_{G\otimes h}$, the degree $6$ $L$-function attached to $G$ by Gritsenko, twisted by the Satake parameters of $h$. The question of whether the same holds for the primes that split in $K$ remains unanswered here, even though we make considerable steps in that direction too. In particular, we have obtained all the essential ingredients, i.e. the factorization of polynomials in parabolic Hecke rings, the necessary rationality theorems, relations between Hecke operators as well as the main computations of the Dirichlet series. However, performing the last few calculations seems very complicated. Our progress is summarised in Theorem \ref{Main Theorem, split case}. Nevertheless, our computations in the split case give us a way to show that the Dirichlet series possesses an Euler product (Theorem \ref{Main Theorem, Euler Product}). The case of the only ramified prime, namely $2$, is not discussed here but we should say that our calculations for the inert case give essentially all ideas required to compute the Euler factor also in this case. \newline

The reader may have already recognised that the choice of $U_2$ is not as random. Indeed, thanks to the so-called accidental isogenies between orthogonal groups of small rank and other classical groups, the spinor $L$-function of  $\textup{Sp}_4$ can be identified with the standard $L$-function attached to a holomorphic modular form of $\textup{SO}(2,3)$. Similarly, the $L$-function studied by Gritsenko is closely related to the standard $L$-function of an orthogonal group of signature $(2,4)$. Therefore, the twist we are studying here is nothing else than the two dimensional twist of this standard $L$-function. Such twists have been studied before in \cite{orthogonal_book} but these methods can only be used for analytic results. If, on the other hand, one is interested in algebraicity results of special values, these methods cannot (or at least is not known how) used to obtain such results because of the use of orthogonal groups which do not correspond to Shimura varieties. On the other hand, the approach taken here (as a triple product) is known to give algebraicity results (see Proposition \ref{Algebraicity} and the discussion there), very much in the way that Heim and Böcherer obtained their algebraicity results in the case of the two dimensional twist of the spinor $L$-function attached to a degree two Siegel modular form. This possibility of obtaining algebraicity results is the main motivation for the present work.  

\begin{notation}
In the following, we will always use the following notation. We will denote by $K := \mathbb{Q}(i)$, the Gaussian field. Let also $\mathcal{O}_K := \mathbb{Z}[i]$ denote its ring of integers and
\begin{equation*}
    J_n = \m{0_n&-1_n\\1_n&0_n}.
\end{equation*}
We denote the space of $m\times n$ matrices with coefficients in a ring $R$ with $M_{m,n}(R)$. If $n=m$, we often use the notation $M_{n}(R)$. For any matrix $M$, we denote by $M^{t}, \textup{ }\det(M), \textup{ tr}(M)$ the transpose, determinant and trace of $M$ respectively. Let also $\textup{GL}_{n}(R) := M_{n}(R)^{\times}$. The zero and identity elements in $M_n(R)$ are denoted by $0_n$ and $1_n$ respectively. We will use the bracket notation $A[B] := \overline{B}^{t}AB$ for complex matrices $A,B$. By $\left[A_1, A_2, \cdots, A_n\right]$, we will denote the block diagonal matrix with the matrices $A_1, A_2, \cdots, A_n$ in the diagonal blocks. We will also use the symbol "$\textup{diag}$" if each $A_i$ is a scalar. For a complex number $z$, we denote by $e(z) := e^{2\pi i z}$ and by $N(z)$ its norm, i.e. $N(z) = z\overline{z}$. For a polynomial $U$ in $X$, with coefficients in some Hecke ring and $G$ a Hecke eigenform, we write $U_{G}$ for the polynomial obtained when we substitute the coefficients (Hecke operators) with the corresponding eigenvalues. Let also
\begin{equation*}
    \Gamma(z) := \int_{0}^{\infty} t^{z-1}e^{-t}\hbox{d}t
\end{equation*}
denote the Gamma function.
\end{notation}
\section{Preliminaries}
We start with some definitions. We mainly follow Gritsenko in \cite{gritsenko}.

\begin{defn}\label{unitary}
Let $R$ be either $K$, $\mathcal{O}_K$ or $\mathbb{C}$ and fix an embedding $R \hookrightarrow \mathbb{C}$. We write $U_n(R)$ for the $R$-points of the unitary group of degree $n \geq 1$. That is,
\begin{equation*}
    U_n(R) := \{g \in \textup{GL}_{2n}(R) \mid J_n[g] = J_n\}.
\end{equation*}
\end{defn}
Hence, for an element $\m{A&B\\C&D} \in U_n(R)$ with $n\times n$ matrices $A,B,C,D$, these satisfy the relations
\begin{equation*}
    \overline{A}^{t}C = \overline{C}^{t}A, \textup{ }\overline{D}^{t}B = \overline{B}^{t}D, \textup{ }A\overline{D}^{t} - \overline{B}^{t}C = 1_n.
\end{equation*}
\begin{defn}
The Hermitian upper half-plane of degree $n$ is defined by 
\begin{equation*}
    \mathbb{H}_n := \{Z = X+iY \in \textup{M}_n(\mathbb{C})| \overline{X}^{t} = X, \overline{Y}^{t} = Y>0\}.
\end{equation*}
\end{defn}
We fix an embedding $K \hookrightarrow \mathbb{C}$. Then, an element $g = \m{A&B\\C&D} \in U_{n}(K) \hookrightarrow U_n(\mathbb{C})$ of the unitary group acts on the above upper half plane via the action
\begin{equation*}
    Z \longmapsto g\langle Z\rangle := (AZ+B)(CZ+D)^{-1}.
\end{equation*}
The usual factor of automorphy is defined by $j(g, Z) := \textup{det}(CZ+D)$.\\

Let now $\Gamma_n$ denote the Hermitian modular group, that is $\Gamma_n := U_n(\mathcal{O}_K)$.
\begin{defn}\label{hermitian_modular_form}
A function $F : \mathbb{H}_n \longrightarrow \mathbb{C}$ is called a Hermitian modular form of integer weight $k \geq 0$ if
\begin{itemize}
    \item $F$ is holomorphic,
    \item $F$ satisfies 
    \begin{equation*}
        F\left(g\langle Z\rangle\right) = j(g, Z)^{k}F(Z),
    \end{equation*}
for all $g \in \Gamma_n$ and $Z \in \mathbb{H}_n$. 
\end{itemize}
If $n=1$, we further require that $F$ is holomorphic at infinity. 
\end{defn}

It is well known that the set of all such forms constitutes a finite dimensional space, which we denote by $M_{n}^{k}$. Because of our assumption if $n=1$ and of Köcher's principle for $n \geq 2$, each such $F$ admits a Fourier expansion
\begin{equation}\label{fourier-expansion}
    F(Z) = \sum_{N}a(N)e(\textup{tr}(NZ)),
\end{equation}
where $a(N) \in \mathbb{C}$ and $N$ runs through all the semi-integral non-negative Hermitian matrices
\begin{equation*}
    N \in \left\{(n_{ij})_{i,j=1}^{n}\geq 0 \mid n_{ii} \in \mathbb{Z}, n_{ij} = \overline{n_{ji}} \in \frac{1}{2}\mathcal{O}_K\right\}.
\end{equation*}
$F$ is called a \textbf{cusp form} if $a(N) \neq 0$ only for $N$ positive definite. We denote the space of cusp forms by $S_{n}^{k}$.
\begin{defn}\label{hermitian_inner_product}
The Petersson inner product for two Hermitian modular forms $F,G$, where at least one of them is a cusp form, is given by
\begin{equation*}
    \langle F, G\rangle := \int_{\Gamma_n\backslash \mathbb{H}_n}{F(Z)\overline{G(Z)}(\det Y)^{k}\hbox{d}^{*}Z},
\end{equation*}
where $\hbox{d}^{*}Z = (\det Y)^{-2n}\hbox{d}X\hbox{d}Y$ with $Z=X+iY$.
\end{defn}
For $R$ as in Definition \ref{unitary}, we consider the following parabolic subgroups of $U_{n}(R)$:
\begin{equation*}
    P_{n,r}(R) = \left\{\begin{pmatrix}* & * \\ 0_{n-r, n+r}&*\end{pmatrix} \in U_{n}(R)\right\},
\end{equation*}
\begin{equation*}
    C_{n,r}(R) = \left\{\begin{pmatrix}* & * \\ 0_{n+r, n-r}&*\end{pmatrix} \in U_{n}(R)\right\}.
\end{equation*}
In particular, when $R = \mathcal{O}_K$, we will just write $P_{n,r}, C_{n,r}$. \\\\
We will be particularly interested in the parabolic subgroup $P_{n,n-1}(K) \leq U_n(K)$, mainly because we can treat the so-called Fourier-Jacobi forms as a special type of modular forms with respect to this subgroup. Let us make this more explicit. We first start with the definition of the slash operator.
\begin{defn}
    Let $n \geq 1$ and $k$ be any integer. Then, for any function $F$ on $\mathbb{H}_{n+1}$ and a matrix $g = \m{A&B\\C&D} \in U_{n+1}(K)$, we define
    \begin{equation*}
        (F\mid_k g)(Z) := \textup{det}(CZ+D)^{-k}F(g\langle Z\rangle).
    \end{equation*}
\end{defn}
We now define $\Gamma_{n,1} := P_{n+1,n}(\mathcal{O}_K)$, the group of integral points of the parabolic $P_{n+1,n}(K)$. We then have the following definition:
\begin{defn}\label{parabolic_modular_forms}
Let $n \geq 1$. A holomorphic function $F$ on $\mathbb{H}_{n+1}$ is a modular form of weight $k$ with respect to the parabolic subgroup $\Gamma_{n,1}$ if the following conditions hold:
\begin{itemize}
    \item $F \mid_k M = F$ for all $M \in \Gamma_{n,1}$,
    \item The function $F(Z)$ is bounded in the domain $\textup{Im}(Z)\geq c > 0$, for all $c>0$.
\end{itemize}
We note here that we can ommit the second condition if $n \geq 2$. This again follows by Köcher's principle. The space of all such forms will be denoted by $M_{n,1}^{k}$. Again, each such $F$ has a Fourier expansion as in equation \eqref{fourier-expansion} and we call $F$ a \textbf{cusp form} if $a(N) \neq 0$ only for positive definite matrices $N$.
\end{defn}
We can now give the definition of Fourier-Jacobi forms, as in \cite[page 2887]{gritsenko}.
\begin{defn}\label{jacobi}
    A complex-valued, holomorphic function $\phi$ on $\mathbb{H}_n \times \mathbb{C}^{n} \times \mathbb{C}^{n}$ is said to be a Fourier-Jacobi form of genus $n$, weight $k$ and index $m$ if the function 
    \begin{equation*}
    \tilde{\phi}\left(\m{\tau&z_1\\z_2^t&\omega}\right):= \phi(\tau, z_1, z_2)e(m\omega),
    \end{equation*}
    where $\omega \in \mathbb{H}_1$ is chosen so that $\m{\tau&z_1\\z_2^t&\omega} \in \mathbb{H}_{n+1}$, is a modular form with respect to the group $\Gamma_{n,1}$. The space of such forms is denoted by $J_{k,m}^{n}$ and we will call $\tilde{\phi}$ a $P$-form, as in \cite[Section 3.4]{heim}.
\end{defn}
Let now $F \in S_{n}^{k}$. If we partition
\begin{equation*}
    Z = \m{\tau&z_1 \\ z_2^{t} &\omega}, 
\end{equation*}
with $\tau \in \mathbb{H}_{n-1}, \omega \in \mathbb{H}_1$ and $z_1,z_2 \in \mathbb{C}^{n-1}$, we can consider the Fourier expansion of $F$ with respect to the variable $\omega$ to be
\begin{equation*}
    F(Z) = \sum_{m=1}^{\infty} \phi_{m}(\tau, z_1, z_2)e(m\omega).
\end{equation*}
The functions $\phi_{m} : \mathbb{H}_{n-1} \times \mathbb{C}^{n-1}\times \mathbb{C}^{n-1} \longrightarrow \mathbb{C}$ are then Fourier-Jacobi forms in the sense of Definition \ref{jacobi} and are called the Fourier-Jacobi coefficients of $F$. We will now restrict ourselves to the case $n=2$.
\begin{defn}\label{inner_product_jacobi}
The Petersson inner product of two Fourier-Jacobi forms $\phi, \psi \in J_{k,m}^{2}$ is defined as
\begin{equation*}
    \langle \phi, \psi \rangle := \int_{\mathcal{F}^{J}}\phi(\tau, z_1,z_2)\overline{\psi(\tau,z_1,z_2)}v^{k}e^{-\pi m |z_1-\overline{z_2}|^2/v}\hbox{d}\mu,
\end{equation*}
where $\hbox{d}\mu = v^{-4}\hbox{d}u\hbox{d}v\hbox{d}x_1\hbox{d}y_1\hbox{d}x_2\hbox{d}y_2$ with $\tau = u+iv$, $z_j=x_j+iy_j$ for $j=1,2$ and $\mathcal{F}^{J}$ is a fundamental domain for the action of $P_{2,1}$ on $\mathbb{H}_1 \times \mathbb{C}^{2}$. 
\end{defn}
The reader should note that we are using the same symbol to denote the inner product for two Fourier-Jacobi forms as the one we use to denote the inner product for two Hermitian modular forms (see Definition \ref{hermitian_inner_product}). However, we will always use a greek letter ($\phi$ or $\psi$) to denote a Fourier-Jacobi form
and a latin letter to denote a Hermitian modular form. This should help eliminate any possibility of confusion.\\
In the following, for $Z \in \mathbb{H}_2$ as above, we write
\begin{equation}\label{imaginary,real}
    \textup{Re}(Z) = \m{x_{\tau} & x_{z_1}\\x_{z_2}&x_{\omega}}, \textup{ Im}(Z) = \m{y_{\tau} & y_{z_1}\\y_{z_2}&y_{\omega}},
\end{equation}
for its real and imaginary part respectively.
\begin{defn}
    Let $\phi_{m}, \psi_{m} \in J_{k,m}^{2}$ and denote by $\tilde{\phi}_m, \tilde{\psi}_m$ the $P$-forms obtained as in Definition \ref{jacobi}. We then define
\begin{equation*}
    \langle \tilde{\phi}_m, \tilde{\psi}_m\rangle_{\mathcal{A}} := \int_{\mathcal{Q}_{1,1}} \widetilde{\phi}_m(Z)\widetilde{\psi}_m(Z)(\det{Y})^{k}\hbox{d}^{*}Z,
\end{equation*}
where $\hbox{d}^{*}Z = (\det{Y})^{-4}\hbox{d}X\hbox{d}Y$ is the invariant element for the action of the unitary group on $\mathbb{H}_2$ and
\begin{equation*}
    \mathcal{Q}_{1,1} := \left\{Z \in \mathbb{H}_{2} \mid (\tau, z_1, z_2) \in \mathcal{F}^{J} \textup{ and } |x_{\omega}| \leq 1/2\right\}.
\end{equation*}
\end{defn}
There is a relation between the two inner products above, given in the following Lemma.
\begin{lemma}\label{inner_product_p_forms}
Let $\phi_{m}, \psi_{m} \in J_{k,m}^{2}$ and denote by $\tilde{\phi}_m, \tilde{\psi}_m$ the corresponding $P$-forms. Then
\begin{equation*}
    \langle\phi_m, \psi_m\rangle = \beta_km^{k-3}\langle\tilde{\phi}_m, \tilde{\psi}_m \rangle_{\mathcal{A}},
\end{equation*}    
for some (specified) constant $\beta_k$.
\end{lemma}
\begin{proof}
We have
\begin{equation*}
    \langle \widetilde{\phi}_m, \widetilde{\psi}_m\rangle_{\mathcal{A}} = \int_{Q_{1,1}}\phi(\tau,z_1,z_2)e^{2\pi i m\omega}\overline{\psi(\tau, z_1,z_2)}e^{-2\pi i m \overline{\omega}}(\det{Y})^{k-4}\hbox{d}X\hbox{d}Y.
\end{equation*}
Let now $\widetilde{y}_{\omega} := y_{\omega} - |z_1-\overline{z_2}|^2/4y_{\tau}$. Then $\det{Y} = y_{\tau}\Tilde{y}_{\omega}$. Hence, the above integral can be written as 
\begin{equation*}
    \int_{\tilde{y}_{\omega}>0}\int_{\mathcal{F}^{J}}\int_{x_{\omega} \Mod{1}}\phi_{m}(\tau,z_1,z_2)e^{-4\pi m(\tilde{y}_{\omega}+|z_1-\overline{z_2}|^2/4y_{\tau})}\overline{\psi_{m}(\tau,z_1,z_2)}(y_{\tau}\tilde{y}_{\omega})^{k-4}\hbox{d}\tau \hbox{d}z_1\hbox{d}z_2\hbox{d}\tilde{y}_{\omega}\hbox{d}x_{\omega}
\end{equation*}
\begin{equation*}
    =\langle \phi_m, \psi_m\rangle\int_{\tilde{y}_{\omega}>0}e^{-4\pi m \tilde{y}_{\omega}}\tilde{y}_{\omega}^{k-4}\hbox{d}\tilde{y}_{\omega} = (4\pi m)^{3-k}\Gamma(k-3)\langle\phi_m, \psi_m\rangle,
\end{equation*}
so the result follows with $\beta_k = (4\pi)^{k-3}\Gamma(k-3)^{-1}$.    
\end{proof}
\section{Hecke Rings and $L$-functions}\label{hecke algebras}
In this section, we give an account of a general Hecke theory we will need. We follow Gritsenko in \cite{gritsenko}. Let $n\geq 1$. We start by defining the groups of similitude:
\begin{equation*}
    S^{n} := \{g \in M_{2n}(K) \mid J_n[g] = \mu(g)J_n, \textup{ for some } \mu(g) > 0\},
\end{equation*}
\begin{equation*}
    S^{n}_{p} := \{g \in S^{n} \cap M_{2n}(\mathcal{O}_K[p^{-1}]) \mid \mu(g) = p^{\delta}, \delta \in \mathbb{Z}\},
\end{equation*}
for each rational prime $p$. It is then well-known that the pairs $(\Gamma_n, S^{n}), (\Gamma_{n}, S^{n}_p)$ are Hecke pairs and we can define the corresponding Hecke rings, which we will denote by $H^n$ and $H_{p}^{n}$ respectively (see \cite[p. 2869, 2870]{gritsenko3} for definitions of Hecke pairs and Hecke rings). From \cite[Corollary $2.2$]{gritsenko}, we can decompose the global Hecke ring into the tensor product of $p$-rings as follows:
\begin{equation*}
    H(\Gamma_n, S^{n}) = \bigotimes_{p}H(\Gamma_n, S^{n}_p).
\end{equation*}
We start with a very general Lemma regarding the embeddings of Hecke algebras.
\begin{lemma}\label{lemma:embeddings}
Let $(\Gamma_0, S_0)$ and $(\Gamma, S)$ be two Hecke pairs. We assume that
\begin{equation*}
    \Gamma_0 \subset \Gamma, \,\,\,\Gamma S_0 = S,\, \,\,\Gamma \cap S_0S_0^{-1} \subset \Gamma_0.
\end{equation*}
Then, given an arbitrary element $X \in H(\Gamma, S)$, according to the second condition, we can write it as
\begin{equation*}
    X = \sum_{i} a_i (\Gamma g_i),
\end{equation*}
with $g_i \in S_0$. Then, if we set
\begin{equation*}
    \epsilon(X) := \sum_{i} a_i (\Gamma_{0}g_i),
\end{equation*}
then $\epsilon$ does not depend on the selection of the elements $g_i \in S_0$ and is an embedding (as a ring homomorphism) of the Hecke algebra $H(\Gamma, S)$ to $H(\Gamma_0, S_0)$.
\end{lemma}
\begin{proof}
See \cite[page 2890]{gritsenko}.
\end{proof}
Now, each $p$-ring is isomorphic to the Hecke ring over the corresponding local field, and the structure of these rings depends on the decomposition of the prime $p$ in $\mathcal{O}_K$ (see \cite[p. 2889]{gritsenko}). In order to work locally, we give the following definitions:
\begin{equation*}
    K_p := K \otimes \mathbb{Q}_p,\textup{ } \mathcal{O}_p := \mathcal{O}_K \otimes \mathbb{Z}_p,\textup{ } \Phi_p := (2i)^{-1}\m{0_n&-1_n\\1_n&0_n},
\end{equation*}
which denote the algebra over $\mathbb{Q}_p$, the maximal lattice and a Hermitian form on the vector space $K_p$ respectively. We also define the unitary group $G_p^{n}$ and a maximal compact subgroup $U_p^{n}$ by
\begin{equation*}
    G_p^{n} := \{g \in \textup{GL}_{2n}(K_p) \mid g^{*}\Phi_p g = \mu(g)\Phi_p, \textup{ for some } \mu(g) \in \mathbb{Q}_p^{*}\},
\end{equation*}
\begin{equation*}
    U_{p}^{n} := \{g \in G_{p}^{n} \cap M_{2n}(\mathcal{O}_p) \mid \mu(g) \in \mathbb{Z}_p^{*}\},
\end{equation*}
where $g^{*} := (g_{ji})^{\sigma}$, with $\sigma$ is the canonical involution of the algebra $K_p$, determined by the behaviour of the prime $p$ in $K$ (split, inert or ramified). We now have the following Proposition:
\begin{proposition}
    For every prime $p$, the local Hecke ring $H(U_{p}^{n}, G_{p}^{n})$ is isomorphic to the $p$-ring $H(\Gamma_n, S_{p}^{n})$.
\end{proposition}
\begin{proof}
    See \cite[Proposition $2.3$]{gritsenko}.
\end{proof}
Let us now recall the definition of the so-called spherical or Satake mapping. We again follow \cite{gritsenko}. We need to distinguish between the cases $p$ is inert or $p=2$ and $p$ splits. In the first case, we know that given $g \in G_p^{n}$, we have the double coset decomposition
\begin{equation*}
    U_{p}^{n}gU_p^{n} = \sum_{i} U_p^{n}M^{m_i}N_i,
\end{equation*}
where $N_i$ is a unipotent matrix, $m_i = (m_{i_1}, \cdots, m_{i_n}; m_{i_0})$ and
\begin{equation*}
    M^{m_i} = \m{p^{m_{i_0}}(\overline{D}^{t})^{-1}&0\\0&D}, \textup{ }D = \textup{diag}(\pi^{m_{i_1}}, \cdots, \pi^{m_{i_n}}),
\end{equation*}
with $\pi=p$ if $p$ is inert or $\pi = (1+i)$ if $p=2$.
We then define
\begin{equation*}
    \Phi : H(U_p^{n}, G_p^{n}) \longrightarrow \mathbb{Q}^{W_{n}}[x_0^{\pm 1}, \cdots, x_n^{\pm 1}], 
\end{equation*}
via
\begin{equation}\label{satake inert}
    \Phi(U_p^{n}gU_{p}^{n}) = \sum_{i} x_0^{m_{i_0}}\prod_{j=1}^{n} (x_jq^{-j})^{m_{i_j}},
\end{equation}
where the ring $ \mathbb{Q}^{W_{n}}[x_0^{\pm 1}, \cdots, x_n^{\pm 1}]$ denotes the ring of polynomials invariant with respect to the permutation of the variables $x_0, \cdots, x_n$ under the transformations $w^{(i)}, i=1,\cdots, n$, defined by
\begin{equation*}
    x_0 \longmapsto p^{-1}x_0x_i^e, \textup{ } x_i \longmapsto p^{2/e}x_i^{-1}, \textup{ } x_j \longmapsto x_j \textup{ } (j \neq 0, i),
\end{equation*}
with $q$ denoting the number of elements in the residue field $\mathbb{Q}(i) \otimes \mathbb{Q}_p$ and $e$ is the ramification index of the prime $p$.\\\\
For the case of decomposable $p$, the definition of the spherical mapping is different. In particular, from \cite[Proposition 2.4]{gritsenko}, there is an isomorphism
\begin{equation*}
\rho: H(U_p^{n}, G_p^{n}) \longrightarrow H(\textup{GL}_{2n}(\mathbb{Z}_p), \textup{GL}_{2n}(\mathbb{Q}_p))[x^{\pm 1}].
\end{equation*}
 We can then define the Satake mapping $\Omega$ for $H(\textup{GL}_{2n}(\mathbb{Z}_p), \textup{GL}_{2n}(\mathbb{Q}_p))$ in an analogous way as for the case $p$ inert or $p=2$, as in \cite[p. 2873]{gritsenko3}. For the reader's convenience, let us describe it here: Given an element $X \in H(\textup{GL}_{2n}(\mathbb{Z}_p), \textup{GL}_{2n}(\mathbb{Q}_p))$, we know that we can write it as
\begin{equation*}
    X = \sum_{i}a_i \textup{GL}_{2n}(\mathbb{Z}_p)g_i,
\end{equation*}
where $g_i = \m{p^{d_{i1}}&*&* &* \\0 &p^{d_{i2}}&*& *\\ 0& 0&\ddots&*\\ 0&0&0 &p^{d_{i2n}}}$ and $a_i \in \mathbb{C}$. Then, the mapping $\Omega$ given by
\begin{equation*}
    \Omega(X) := \sum_{i}\prod_{j=1}^{2n} (x_jp^{-j})^{d_{ij}},
\end{equation*}
defines an isomorphism between $H(\textup{GL}_{2n}(\mathbb{Z}_p), \textup{GL}_{2n}(\mathbb{Q}_p))$ and $\mathbb{Q}^{\textup{sym}}[x_1^{\pm 1}, \cdots, x_{2n}^{\pm 1}]$ of symmetric polynomials. We then define the Satake mapping $\Phi$ in this case as the composition
\begin{equation}\label{satake split}
    \Phi := \Omega \circ \rho.
\end{equation}
Let us now define the parabolic Hecke algebras we will need.
Let $S^{n,1}, \textup{ }S_{p}^{n,1}, \textup{ }\Gamma_{n,1}$ denote the intersection of the groups $S^{n+1}, \textup{ }S_{p}^{n+1}, \textup{ }\Gamma_{n+1}$ with the parabolic subgroup $P_{n+1,n}$ respectively. Again, the pairs $(\Gamma_{n,1}, S^{n,1})$ and $(\Gamma_{n,1}, S_{p}^{n,1})$ are Hecke pairs (cf. \cite[Section 3]{gritsenko}) and we can then define the Hecke rings
\begin{equation*}
    H^{n,1} := H(\Gamma_{n,1}, S^{n,1}), \textup{ }H_{p}^{n,1} := H(\Gamma_{n,1}, S_{p}^{n,1}).
\end{equation*}
Since $\Gamma_{n+1}S_{p}^{n,1} = S_{p}^{n+1}$ and after writing an element $X \in H_{p}^{n+1}$ as
\begin{equation*}
    X = \sum_{i} a_i \Gamma_{n+1} g_i,
\end{equation*}
with $g_i \in S_{p}^{n+1}$, we can define an embedding
\begin{equation*}
    X \longmapsto \epsilon(X) = \sum_{i} a_i\Gamma_{n,1}g_i,
\end{equation*}
using Lemma \ref{lemma:embeddings}. Moreover, we can embed $H(\Gamma_{n}, S^{n}) \xhookrightarrow{} H(\Gamma_{n,1}, S^{n,1})$ in two ways, as follows:\\

If $X = \Gamma_{n}g\Gamma_{n}$ with $g = \left[A, D\right] \in S^{n}$, we define
\begin{equation}\label{+- embeddings}
    j_{-}(X) := \Gamma_{n,1}\left[A,\mu(g), D, 1\right]\Gamma_{n,1}, \textup{ }j_{+}(X) := \Gamma_{n,1}\left[A,1, D, \mu(g)\right]\Gamma_{n,1}.
\end{equation}
These are related by an anti-homomorphism $* : H_{p}^{n,1} \longrightarrow H_p^{n,1}$, given by 
\begin{equation}\label{antihomomorphism}
    \sum_{i}a_i\Gamma_{n,1}M_i\Gamma_{n,1} \longmapsto \sum_{i}a_i\Gamma_{n,1}\mu(M_i)M_i^{-1}\Gamma_{n,1},
\end{equation}
as in \cite[Lemma 3.1]{gritsenko}.\\

We now again restrict our discussion to the degree $2$ case. We note that $H^{1,1}$ is not commutative and also does not split into the tensor product of the $H_p^{1,1}$ rings.
The structure of the parabolic Hecke rings $H_p^{1,1}$ again depends on the decomposition of the prime $p$ in $\mathcal{O}_K$.
\\\\
If $p$ is inert or $p=2$, then the structure of the Hecke ring is constructed in a similar way as the corresponding ring for the symplectic group of genus $2$, see \cite{heim}, \cite{gritsenko1} for example. 
\\\\
In the case of a decomposable $p$, however, the situation is quite different. This follows from the fact that
\begin{equation*}
H(U_p^{2}, G_p^{2}) \cong H(\textup{GL}_{4}(\mathbb{Z}_p), \textup{GL}_{4}(\mathbb{Q}_p))[x^{\pm 1}]. 
\end{equation*}
This gives that the corresponding $p$-ring of the parabolic Hecke algebra is isomorphic to the ring of polynomials of one variable with coefficients from the Hecke ring of the parabolic subgroup
\begin{equation*}
P_{1,2,1}(\mathbb{Z}_p) = \left\{\m{g_1 & *& *\\0&g& *\\0&0&g_2} \in \textup{GL}_{4}(\mathbb{Z}_p) \mid g_1,g_2 \in \mathbb{Z}_{p}^{\times}, \textup{ }g \in \textup{GL}_{2}(\mathbb{Z}_p)\right\}.
\end{equation*}
Properties of this ring have been investigated in \cite{gritsenko3} and that's the ring where our calculations involving Hecke operators are going to occur.\\\\
Finally, let us describe the action of elements of $H^{1,1}$ on modular forms.
Let $F$ denote any modular form of weight $k$ with respect to the parabolic subgroup $\Gamma_{1,1}$, as in Definition \ref{parabolic_modular_forms}. Let also $g=\m{A&B\\C&D} \in S^{1,1}$ with corresponding $\mu(g)>0$. We then define
\begin{equation*}
    (F\mid_{k} g)(Z) := \mu(g)^{2k-4}\textup{det}(CZ+D)^{-k}F(g\langle Z\rangle).
\end{equation*}
Let now  $X = \Gamma_{1,1}\m{*&0&*&*\\**&a&*&*\\**&0&*&*\\0&0&0&b}\Gamma_{1,1} = \sum_{i} \Gamma_{1,1}g_i \in H^{1,1}$, for some $g_i \in S^{1,1}$. Then, we define
\begin{equation*}
    (F\mid_{k} X)(Z) := \sum_{i}(F\mid_{k} g_i)(Z).
\end{equation*}
Gritsenko gave the following very convenient definition of the signature.
\begin{defn}\label{signature}
    The signature of $X$ is defined as $s(X) := b/a$.
\end{defn}
Using the signature $s:=s(X)$ of $X$ we can now define its action on Fourier-Jacobi forms.
\begin{proposition}\label{fourier_jacobi_action}
    Let $\phi \in J_{k,m}^{2}$ denote a Fourier-Jacobi form of weight $k$ and index $m$. Then, for $Z = \m{\tau&z_1\\z_2&\omega} \in \mathbb{H}_2$, we define the action of $X$ on $\phi$ via
    \begin{equation*}
        \left(\phi \mid_k X\right)(\tau, z_1, z_2) := \left(\widetilde{\phi} \mid_k X \right)(Z)e\left(-\frac{m}{s}\omega\right) ,
    \end{equation*}
with $\widetilde{\phi}(Z):= \phi(\tau, z_1, z_2)e(m\omega)$. Then $\phi \mid_k X$ belongs to $J_{k,m/s}^{2}$ if $m/s$ is an integer and is $0$ otherwise.
\end{proposition}
\begin{proof}
    See \cite[Lemma $4.1$]{gritsenko}.
\end{proof}
\textbf{Note}: Throughout the paper, we will sometimes just write $\mid$ instead of $\mid_{k}$ for the weight $k$-action on Fourier-Jacobi forms, as the weight is always fixed.\\

Now, if $F$ is a Hermitian cusp form, we can write
\begin{equation*}
    F\left(\m{\tau&z_1\\z_2&\omega}\right) = \sum_{m=1}^{\infty}\phi_{m}(\tau,z_1,z_2)e(m\omega).
\end{equation*}
For $X \in H^{1,1}$ as above, we have that $F \mid_{k} X$ is a modular form with respect to $\Gamma_{1,1}$ and so we can write
\begin{equation*}
    (F \mid_{k} X)\left(\m{\tau&z_1\\z_2&\omega}\right) = \sum_{m=1}^{\infty}\psi_m(\tau,z_1,z_2)e(m\omega).
\end{equation*}
Therefore, there is an action of Hecke operators from $H^{1,1}$ on the Fourier-Jacobi forms coming from a Hermitian modular form $F$ via
\begin{equation*}
    \phi_m^{(F)} \mid\mid X := \psi_m^{(F\mid_{k} X)}.
\end{equation*}
We note here that this action is extended to $P$-forms in the obvious way.\\

Given now the definitions of the Hecke algebras above, we assume that $G \in S_2^{k}$ is a Hecke eigenform for $H^{2}$, i.e. it is an eigenfunction for all Hecke operators in $H^{2}$. We remind the reader here that for a polynomial $U[X] \in H^{2}[X]$ and $G$ a Hecke eigenform, we write $U_{G}$ for the polynomial obtained when we substitute the operators with their corresponding eigenvalues.
\begin{defn}\label{standard l-function}
    The standard $L$-function attached to $G$ (see also \cite[Paragraph 20.6]{arithmeticity}) is defined as 
    \begin{equation*}
        Z_{G}^{(2)}(s) := \prod_{p \textup{ inert or }p=2} Z_{p,G}^{(2)}(p^{-2s})^{-1}\prod_{p = \pi\overline{\pi}} Z_{\pi,G}^{(2)}(p^{-s})^{-1}Z_{\overline{\pi},G}^{(2)}(p^{-s})^{-1},
    \end{equation*}
where for each inert prime $p$ or $p=2$, $Z_{p}^{(2)}(t) := \Phi^{-1}\left(z_{p}^{(2)}(t)\right)$ and for $2 \neq p=\pi\overline{\pi}$, $Z_{\pi}^{(2)}(t) := \Phi^{-1}\left(z_{\pi}^{(2)}(t)\right)$ and $Z_{\overline{\pi}}^{(2)}(t) := \Phi^{-1}\left(z_{\overline{\pi}}^{(2)}(t)\right)$, where
\begin{equation*}
z_{p}^{(2)}(t) :=
\begin{cases}
    \prod_{i=1}^{2}(1-p^{2}x_{i,p}t)(1-p^{4}x_{i,p}^{-1}t) & \textup{ if $p$ inert}\\
    \prod_{i=1}^{2}(1-px_it)(1-p^{2}x_i^{-1}t) & \textup{ if $p=2$}
\end{cases},
\end{equation*}
\begin{equation*}
    z_{\pi}^{(2)}(t) := \prod_{i=1}^{4}(1-p^{-1}x_{i,p}t), \textup{ }z_{\overline{\pi}}^{(2)}(t) := \prod_{i=1}^{4} (1 - p^4x_{i,p}^{-1}t),
\end{equation*}
and $\Phi$ is the Satake mapping of equations \eqref{satake inert} and \eqref{satake split}. As the polynomials $Z_p^{(2)}(t), Z_{\pi}^{(2)}(t), Z_{\overline{\pi}}^{(2)}(t)$ are all in $H^{2}[t]$, the corresponding $Z_{p,G}^{(2)}(t), Z_{\pi, G}^{(2)}(t), Z_{\overline{\pi}, G}^{(2)}(t)$ are obtained as explained just before this Definition.
\end{defn}
\begin{defn}\label{gritsenko l-function}
    The $L$-function attached to $G$ by Gritsenko in \cite[p. 2545]{gritsenko_zeta} (for the case of $p$ inert and $p=2$) and in the proof of \cite[Lemma 2.1]{gritsenko_zeta} (for the case of split prime $p$) is defined as
    \begin{equation*}
        Q_{G}^{(2)}(s) := \prod_{p \textup{ inert}}(1+p^{k-2-s})^{-2}Q_{p,G}^{(2)}(p^{-s})^{-1}\prod_{p \textup{ splits or }p=2}Q_{p,G}^{(2)}(p^{-s})^{-1},
    \end{equation*}
where $Q_{p}^{(2)}(t) := \Phi^{-1}\left(q_{p}^{(2)}(t)\right)$ with
\begin{equation*}
q_{p}^{(2)}(t) := 
    \begin{cases}
    (1-x_{0,p}t)\prod_{r=1}^{2}\prod_{1\leq i_1<i_2\leq 2}(1-p^{-r}x_{i_1,p}\cdots x_{i_r,p}x_{0,p}t) & \textup{ if $p$ is inert}\\
    (1-x_{0,p}t)\prod_{r=1}^{2}\prod_{1\leq i_1<i_2\leq 2}(1-p^{-r}(x_{i_1,p}\cdots x_{i_r,p})^2x_{0,p}t) & \textup{ if $p=2$}\\
    \prod_{1\leq i<j\leq 4}(1-p^{-3}x_{i,p}x_{j,p}xt) & \textup{ if $p$ splits}
    \end{cases},
\end{equation*}
and $\Phi$ the Satake mapping of equations \eqref{satake inert} and \eqref{satake split}. As $Q_{p}^{(2)}(t) \in H^{2}[t]$, the corresponding $Q_{p,G}^{(2)}(t)$ is obtained as explained just before Definition \ref{standard l-function}.
\end{defn}
Let us now define the so-called Maass space for the case of Hermitian cusp forms. We mainly follow \cite{gritsenko_maass} and for the Definition we will use \cite[Lemma $2.4$]{gritsenko_maass}.
\begin{defn}\label{maass_defn} 
The Maass space is the space
\begin{equation*}
    \left\{F\left(\begin{pmatrix}\tau&z_1\\z_2&\omega\end{pmatrix}\right) = \sum_{m=1}^{\infty}\left(\phi(\tau, z_1,z_2)\mid _{k}T_{-}(m)\right)e^{2\pi i m\omega}m^{3-k} \textup{ } | \ \textup{ }\phi \in J_{k,1}^{2}\right\},
\end{equation*}
where $T_{-}(m) := j_{-}(T(m)) \in H^{1,1}$, with
\begin{equation*}
    T(m) := \sum_{\substack{g \in S^{1} \cap \textup{M}_2(\mathbb{\mathbb{Z}})\\\mu(g)=m}}\Gamma_1g\Gamma_1,
\end{equation*}
and $j_{-}$ is the embedding of equation \eqref{+- embeddings}. In particular, this is the standard Hecke element of $\textup{SL}_2(\mathbb{Z})$, viewed as an element of $H^1$.
\end{defn}
The main property of the Maass space is the following:
\begin{proposition}\label{maass_lift}
Let $F \in S_2^{k}$ belong in the Maass space defined and assume $F$ is an eigenfunction for the Hecke algebra $H^{2}$. Then, there exists $f \in S_{k-1}\left(\Gamma_0(4), \left(\frac{-4}{*}\right)\right)$, which is also an eigenfunction for its corresponding Hecke algebra, such that
\begin{equation*}
    Q_{F}^{(2)}(s) = \zeta(s-k+1)L\left(s-k+2, \left(\frac{-4}{*}\right)\right)\zeta(s-k+3)R_{f}(s),
\end{equation*}
where $R_{f}(s)$ denotes the symmetric square function of $f$ defined as follows:
Let $f(\tau) = \sum_{n \geq 1}a(n)e(n\tau)$ be the Fourier expansion of $f$ and assume we write
\begin{equation*}
    1-a(p)t + \left(\frac{-4}{p}\right)p^{k-2}t^2 = (1 - \alpha_p t)\left(1 - \beta_p\left(\frac{-4}{p}\right)t\right).
\end{equation*}
We then define
\begin{equation*}
    R_{f}(s) := (1-a(2)^2 2^{-s})^{-1}(1-\overline{a(2)^2}2^{-s})^{-1}\prod_{p \neq 2}\left[(1-\alpha_p^2p^{-s})\left(1 - \left(\frac{-4}{p}\right)\alpha_p\beta_p p^{-s}\right)(1- \beta_{p}^{2}p^{-s})\right]^{-1}.
\end{equation*}
We call $F$ the Maass lift of $f$.
\end{proposition}
\begin{proof}
    See \cite[Theorem, p. 69]{gritsenko_maass} or the Appendix in \cite{gritsenko}.
\end{proof}
We end this section by giving a Lemma regarding the correspondence of elliptic cusp forms and Hermitian modular forms of degree $1$, both as analytic objects as well as Hecke eigenforms.
\begin{lemma}
A Hermitian modular form of degree $1$ and weight $k$ with $k\equiv 0 \pmod 4$ can be considered as a classical modular form of the same weight (i.e. for the group $\hbox{SL}_{2}(\mathbb{Z})$) and vice versa. Also, a classical modular form which is a normalised eigenform for the Hecke algebra $H(\textup{GL}_2(\mathbb{Z}), \textup{GL}_2(\mathbb{Q}))$ is also a normalised eigenform for $H(\Gamma_1, S^1)$, when it is considered as a Hermitian modular form and vice versa.
\end{lemma}
\begin{proof}
We have that $\Gamma_1 = \hbox{SL}_{2}(\mathbb{Z}) \cdot \{\alpha \cdot 1_{2} \mid \alpha \in \mathcal{O}_{K}^{\times}\}$ and the corresponding upper half planes are the same. So, holomorphicity is equivalent (including infinity). Now, for the invariance condition, the one direction is trivial, as $\hbox{SL}_{2}(\mathbb{Z}) \subseteq \Gamma_{1}$. For the other one, let $\gamma \in \Gamma_{1}$ and write $\gamma = \alpha \delta$ with $\delta = \begin{pmatrix}a&b\\c&d\end{pmatrix}\in \hbox{SL}_{2}(\mathbb{Z})$ and $\alpha \in \mathcal{O}_K^{\times}$. Then
\begin{equation*}
    (F|_{k}\gamma)(Z)  = (\alpha c z + \alpha d)^{-k}F(Z) = (F|_{k}\delta)(Z),
\end{equation*}
as $k \equiv 0 \pmod 4$ and $\mathcal{O}_K^{\times} = \{\pm 1, \pm i\}$.\\

Assume now that we start with a normalised (i.e $a(1) =1$ in the Fourier expansion) Hermitian cusp form $h$ of degree $1$, which we further take to be an eigenform for the corresponding Hecke algebra. The canonical embedding of $\textup{GL}_2(\mathbb{Q})$ into $S^1$, the group of similitude of degree $1$, induced from the embedding $\mathbb{Q} \hookrightarrow K$, allows us to see $h$ as a classical normalised Hecke eigenform. \\

But the converse is also true, that is, if we start with $h$ a classical normalised Hecke eigenform, then it is also a normalised Hermitian eigenform of degree $1$. Indeed, since the Hecke operators of the Hermitian Hecke algebra are normal, we know that the space of Hermitian cusp forms is diagonalizable with a basis of normalised eigenforms. But then this basis has to coincide with the basis derived by diagonalizing the action of the classical Hecke algebra, thanks to the multiplicity one theorem for the classical Hecke algebra.
\end{proof}
\begin{remark}
    From now on, we will use the terms elliptic modular form (or classical modular form) and Hermitian modular form of degree $1$ interchangeably.
\end{remark}
\section{Integral Representation and Dirichlet Series}
We start with the following definition of the Eisenstein series we will need.
\begin{defn}\label{eisenstein_defn}
Let $0\leq r \leq n$ and $F \in S_r^k$ with $k \equiv 0 \pmod 4$. The Klingen-type Eisenstein series with respect to the parabolic subgroup $C_{n,r}$ attached to $F$ is given by:
\begin{equation*}
    E_{n,r}^{k}(Z,F;s) = \sum_{\gamma \in C_{n,r}\backslash \Gamma_n}F(\gamma \langle Z\rangle_{*})j(\gamma, Z)^{-k}\left(\frac{\det \operatorname{Im}\gamma \langle Z\rangle}{\det{\operatorname{Im}\gamma \langle Z\rangle_{*}}}\right)^{s},
\end{equation*}
where $*$ denotes the lower right $r\times r$ part of the matrix. 
\end{defn}
This series converges absolutely and uniformly on compact subsets of $\mathbb{C}$ for $k + 2\textup{Re}(s) > 2(n+r)$. In order to prove this, we use the fact that the function $(\det{\operatorname{Im}\gamma \langle Z\rangle_{*}})^{k/2}F(\gamma \langle Z\rangle_{*})$ is bounded on $\mathbb{H}_n$, say by a constant $C$ (see \cite[Lemma III.2.4]{Krieg_Book}) and that $\det \operatorname{Im}\gamma \langle Z\rangle = |j(\gamma, Z)|^{-2} \det \textup{Im} Z$. Hence, the series is bounded by
\begin{equation*}
    C (\det \textup{Im} Z )^{s} \sum_{\gamma \in C_{n,r}\backslash \Gamma_{n}} (\det{\operatorname{Im}\gamma \langle Z\rangle_{*}})^{-\frac{1}{2}(k+2s)}|j(\gamma, Z)|^{-(k+2s)},
\end{equation*}
and the last series converges absolutely and uniformly on compact subsets of $\mathbb{C}$, whenever $k+2\textup{Re}(s)>2(n+r)$, from \cite[Theorem V.2.8]{Krieg_Book}.\\

Let $n \geq 2$. Given $F,G \in S_n^k$ and $h \in S_1^k$, which are all eigenforms for their corresponding Hecke algebras, the expression of interest is
\begin{equation*}
\Phi(F, G, h; s) := \left\langle\left\langle\left\langle E_{2n+1,0}^{k}\left(\begin{pmatrix}z_1&&\\&z_2&\\&&z_3\end{pmatrix}; s\right), F(z_3)\right\rangle, G(z_2)\right\rangle, h(z_1)\right\rangle,
\end{equation*}
where here the inner product $\langle\textup{ },\textup{ } \rangle$ denotes the Hermitian inner product of Definition \ref{hermitian_inner_product}.
We have the following algebraic result:
\begin{proposition}\label{Algebraicity} Assume that $F, G$ and $h$ have algebraic Fourier coefficients. For $k > 4n+2$ we have
\[
\frac{\Phi(F,G,h;0)}{\langle F,F\rangle \langle G, G\rangle \langle h,h\rangle} \in \overline{\mathbb{Q}}.
\]  
\end{proposition}
\begin{proof}This can be shown exactly as \cite[Theorem 1.9]{heim}. In the proof there, a result of Böcherer is used on the algebraic decomposition of the space of modular forms as an orthogonal product of the space of cusps forms and of the Eisenstein series, i.e. 
\[
M_{n}^k\left(\overline{\mathbb{Q}}\right) = S_n^k\left(\overline{\mathbb{Q}}\right) \oplus \textup{Eis}_n^k\left(\overline{\mathbb{Q}}\right).
\]
Such a result is also available for unitary groups in \cite[Theorem $27.14$]{arithmeticity}. 
\end{proof}

Actually, one can give an even stronger statement of the proposition above, namely establish even a reciprocity law on the action of the absolute Galois group. The statement is similar to \cite[Theorem $1.9$]{heim} of Heim. The main point here is that we can establish an algebraicity result for special values of $L$-functions if we can relate the expression above to an Euler product expression. This is the main motivation of the present paper. \newline

By using the well-known doubling method for unitary groups, as for example is studied in \cite[Equation 24.29 (a)]{arithmeticity}, we know that the first inner product is related to a Klingen-type Eisenstein series as defined above. That is, for $F$ with totally real Fourier coefficients, 
\begin{equation*}
\left\langle E_{2n+1,0}^{k}\left(\begin{pmatrix}z_1&&\\&z_2&\\&&z_3\end{pmatrix}; s\right), F(z_3)\right\rangle=
     \nu(s)\frac{Z_{F}^{(n)}(s+k/2)}{\prod_{i=0}^{2n-1}L(2s+k-i, \chi^i)}E_{n+1,n}^{k}\left(\begin{pmatrix}z_1&0\\0&z_2\end{pmatrix}, F;s\right),
\end{equation*}
where $Z_{F}^{(n)}$ is the standard $L$-function attached to $F$ (a generalisation of Definition \ref{standard l-function} to degree $n$ modular forms), $\chi$ is the non-trivial quadratic character attached to the extension $K/\mathbb{Q}$ and $\nu(s)$ is an expression involving Gamma factors (the explicit expression is given in \cite{arithmeticity}). So, our focus shifts to computing
\begin{equation*}
    \left\langle\left\langle E_{n+1,n}^{k}\left(\begin{pmatrix}W&0\\0&Z\end{pmatrix}, F;s\right), G(Z)\right\rangle, h(W)\right\rangle \textup{ }(W\in \mathbb{H}_1,\textup{ } Z \in \mathbb{H}_n).
\end{equation*}
We will now focus on the case $n=2$. Given Definition \ref{eisenstein_defn} of the Eisenstein series, we will start by finding representatives for $C_{3,2} \backslash \Gamma_{3}$. We begin by first finding representatives for $C_{3,2}(K) \backslash U_{3}(K)$. Now, for any $m,n \geq 1$, there is an embedding $U_{m}(K) \times U_n(K) \xhookrightarrow{}U_{m+n}(K)$ given by
\begin{equation*}
    \m{A_1&B_1\\C_1 & D_1} \times \m{A_2&B_2\\C_2&D_2} \xhookrightarrow{} \m{A_1&0&B_1&0\\0&A_2&0&B_2\\C_1&0&D_1&0\\0&C_2&0&D_2}.
\end{equation*}
Also, for each $r \in \mathbb{Q}^{\times}$, we consider the following subgroups of $U_1(K)$ and $U_2(K)$, respectively:
\begin{equation}\label{h subgroups}
    H_{1,r}(K) = \left\{\m{a & b\\-br^2 & a} \in U_1(K)\right\}, \,\, H_{2,r}(K) = \left\{\m{a_1 & a_2 & b_1 & b_2 \\ irb_3 & a_4& b_3& b_4 \\-r^2b_1& ira_2 &a_1&irb_2\\ird_3&c_4&d_3&d_4} \in U_2(K)\right\}.
\end{equation}
\begin{proposition}\label{reps}
The right coset space $C_{3,2}(K) \backslash U_{3}(K)$ has representatives
\begin{equation*}
    S_1 = C_{1,0}(K) \backslash U_1(K) \times 1_{4},
\end{equation*}
\begin{equation*}
    S_2 = \pi_{2} \cdot (1_2 \times C_{2,1}(K)\backslash U_2(K)),
\end{equation*}
\begin{equation*}
    S_3 = \xi \cdot \left(C_{1,0}(K)\backslash U_1(K) \times ((T\times 1_{2})\cdot C_{2,1}(K)\backslash U_2(K))\right),
\end{equation*}
\begin{equation*}
    W_r = \xi_r \cdot (D \cdot H_{1,r}(K)\backslash U_1(K) \times H_{2,r}(K) \backslash U_{2}(K)), \ r \in \mathbb{Q}^{\times} / N_{K/\mathbb{Q}}(K^{\times}),
\end{equation*}
where 
\begin{multline*}
\pi := \begin{pmatrix}0 & 1\\1& 0 \\&& 1\\&&&0 &1\\&&&1&0\\&&&&&1_{}\end{pmatrix},\textup{ } \xi := \begin{pmatrix}1 & 0\\1& 1 \\&& 1\\&&&1 &-1\\&&&0&1\\&&&&&1\end{pmatrix},\textup{ } T := \left\{\begin{pmatrix}a & 0\\0& \overline{a}^{-1}\end{pmatrix} \textup{ }|\textup{ }a\in K^{\times}\right\},
\end{multline*}
\begin{equation*}
    \xi_r := \m{1 & 0 & 0 & 0\\1 & 1 & 0 & 0\\ ir & ir & 1&-1 \\ -ir & 0 & 0 & 1} \times 1_{2}, \,\, D:=\left\{\m{d & 0 \\0 & d} \mid d \in K, \, N(d)=1\right\}.
\end{equation*}
\end{proposition}
\begin{proof}
Let $e_1,\ e_2,\ e_{3},\ f_1,\ f_2,\ f_{3}$ denote the standard basis for $K^{6}$, viewed as row vectors. The map $g \longmapsto f_1 g$ induces a bijection between $C_{3,2}(K)\backslash U_{3}(K)$ and the set $X$ of one-dimensional isotropic subspaces in $K^{6}$ (this is a standard fact, consequence of Witt's Theorem, as in \cite[Lemma 2.1]{Shimura_Euler_Product} for example). Let $V$ be such a subspace. We decompose $K^{6} = K^2 \oplus K^{4}$ according to the embedding $U_1(K) \times U_2(K) \xhookrightarrow{}U_{3}(K)$ (i.e., $K^2 = Ke_1 \oplus Kf_1$). There are three possibilities: $V$ is contained in $K^2$, $V$ is contained in $K^{4}$ or $V$ is not contained in either. \\

In the first two cases, $V$ is an isotropic subspace of $K^2$ or $K^4$, respectively, and hence we obtain the same set of representatives $S_1, S_2$ as in \cite[Proposition 2.1]{heim}.\\

For the last case, assume $V$ is spanned by the isotropic vector $v$. We decompose $v = v_1 \oplus v_2$ and we have two possibilities: $v_1$ is isotropic or $v_1$ is not isotropic. In the first case, in analogy with \cite[Proposition 2.1]{heim}, we obtain the set $S_3$.\\

Assume now $v_1$ is not isotropic. Then $v_2$ will not be isotropic either (since $v$ is isotropic). Let us write $\overline{v_1}^{t}J_1v_1 = 2ir = -\overline{v_2}^{t}J_2v_2$, with $r \in \mathbb{Q}^{\times}$. By Witt's Theorem (\cite[Theorem 1.2]{Shimura_Euler_Product}), we have that an isotropic vector $w = w_1 \oplus w_2 \in K^{6}$ will be in the same orbit as in $v$ under the action of $U_1(K) \times U_2(K)$ if and only if $\overline{w_1}^{t}J_1w_1 = 2ir = -\overline{w_2}^{t}J_2w_2$. This shows that the isotropic vectors $v=v_1 \oplus v_2$, with the same norm on the first component (hence the second too), form a single orbit under the action of $U_1(K) \times U_2(K)$. Since we are free to scale $v$ by some $\lambda \in K^{\times}$, (because we work with the subspace spanned by $v$), we must consider $r \in \mathbb{Q}^{\times}/N_{K/\mathbb{Q}}(K^{\times})$. Here, $N_{K/\mathbb{Q}}(a+ib):= a^2+b^2$ for $a, b \in \mathbb{Q}$.\\

Now, for each such $r$, we consider $v_r = \m{ir & ir & 0 & 1 & -1 & 0}$ as a representative of its orbit. We then observe that the matrix $\xi_r$ defined in the statement of the Proposition satisfies $f_1 \xi_r = v_r$. Hence, we deduce that the double quotient $C_{3,2}(K)\backslash U_{3}(K)/ (U_1(K) \times U_2(K))$ has the following irredundant representatives:
\begin{equation*}
    1_{6}, \pi, \xi, \{\xi_r, r \in \mathbb{Q}^{\times}/N_{K/\mathbb{Q}}(K^{\times})\}.
\end{equation*}
We now have that for $g_1 \in U_1(K), \ g_2 \in U_2(K)$ and $r$ as above, $C_{3,2}(K)\xi_r (g_1 \times g_2) = C_{3,2}(K)\xi_r$ if and only if $\xi_r(g_1 \times g_2)\xi_r^{-1} \in C_{3,2}(K)$. This then implies that $g_1 \in H_{1,r}(K),\ g_2 \in H_{2,r}(K)$ and $a+irb = a_1-irb_1$, where we write $g_1, g_2$ as in \eqref{h subgroups}.\\

For $i=1,2$, if $g_i \in U_i(K)$, we write $g_1=h_1e$ and $g_2=h_2f$, where $h_i \in H_{i,r}(K)$ and $e,f$ belong to a set of proper representatives for $H_{i,r}(K)\backslash U_i(K)$, respectively. By writing $h_i$ as in \eqref{h subgroups}, we let $d := (a_1-irb_1)(a+irb)^{-1}$. This is well-defined as $N(a+irb) = N(a_1-irb_1)= 1$ by unitarity, so in particular $a+irb, \ a_1+irc_1$ are non-zero. Moreover, $N(d)=1$ and if $D := \textup{diag}(d, d) \in U_1(K)$, we have that $Dh_1 \times h_2 \in \xi_r^{-1}C_{3,2}(K)\xi_r$. Hence,
\begin{multline*}
    C_{3,2}(K)\xi_r (g_1 \times g_2) = C_{3,2}(K)\xi_r (h_1 e \times h_2 f) = C_{3,2}(K) \xi_r(D h_1 \times h_2)(D^{-1} \times 1) (e \times f) \\= C_{3,2}(K)\xi_r(D^{-1}e \times f).
\end{multline*}
This gives us the set of representatives $W_r$. Hence, the Proposition follows.
\end{proof}
We now want to pull these representatives back to representatives for $C_{3,2}\backslash \Gamma_{3}$. 
\begin{corollary}\label{cor2}
The right coset space $C_{3,2}\backslash \Gamma_{3}$ has representatives
\begin{equation*}
    T_1 = C_{1,0}\backslash \Gamma_1 \times 1_{4},
\end{equation*}
\begin{equation*}
    T_2 = \pi \cdot (1_2 \times C_{2,1}\backslash \Gamma_2),
\end{equation*}
\begin{equation*}
    T_3 = \bigsqcup_{p,q} (\xi^{p,q} \times 1_2)\cdot (C_{1,0}\backslash \Gamma_1 \times C_{2,1}\backslash \Gamma_2),
\end{equation*}
\begin{equation*}
    V_r, \,\, r \in \mathbb{Q}^{\times}/N_{K/\mathbb{Q}}(K^{\times})
\end{equation*}
Here, $p,q \in \mathbb{Z}[i]\backslash\{0\}$ with $\gcd(p,q)=1$, $q=u+iv, \textup{ } u>0,\textup{ } v\geq 0$ and $\xi^{p,q} := \begin{pmatrix}*&*&0&0\\q&p&0&0\\0&0&\overline{p}&-\overline{q}\\0&0&*&*\end{pmatrix}$, with $\xi^{p,q} \times 1_{2} \in \Gamma_{3}$. The sets $V_r$ correspond to the representatives obtained by pulling $W_r$ of Proposition \ref{reps} back to $\mathcal{O}_K$.
\end{corollary}
\begin{proof}
We first observe that there is a one-to-one correspondence between $C_{n,r}(K)\backslash U_n(K)$ and $C_{n,r}\backslash \Gamma_{n}$, for any $n, r$ with $0\leq r\leq n$. Indeed, since $K$ has class number one, that is $\mathbb{Z}[i]$ is a principal ideal domain, we have from \cite[Proposition 7.2 (2), p. 48]{Shimura_Euler_Product}, that $U_n(K) = C_{n,r}(K) \Gamma_n$, and hence
\begin{equation*}
    C_{n,r}(K)\backslash U_n(K) \cong (\Gamma_n \cap C_{n,r}(K)) \backslash \Gamma_n = C_{n,r} \backslash \Gamma_n.
\end{equation*}
Now, each $V_r$ is obtained by pulling $W_r$ of Proposition \ref{reps} back to $\mathcal{O}_K$, thanks to this correspondence. Moreover, $T_1$ and $T_2$ are obtained by the sets $S_1, S_2$ of Proposition \ref{reps}. In order to obtain the set $T_3$, it suffices to pull $S_3$ back to $\mathcal{O}_K$. We therefore need to find a matrix $M_a$ in $C_{3,2}(K)$ such that 
\begin{equation}\label{xi reps}
    M_a\cdot \xi \cdot \left(1_2 \times \begin{pmatrix}a&0\\0&\overline{a}^{-1}\end{pmatrix}\times 1_{2}\right) \in \Gamma_{3},
\end{equation}
with $\xi$ the matrix of Proposition \ref{reps}. We parametrize $K$ as 
\begin{equation*}
    \left\{\frac{p}{q} \mid p, q \in \mathbb{Z}[i], \ \gcd(p,q)=1,\textup{ } q = u+iv,\textup{ } u>0,\textup{ } v\geq 0 \right\}.
\end{equation*}
All these elements are different as $p,q$ are in $\mathbb{Z}[i]$ with the above conditions and their union is $K$. For $a = p/q$ as above, we define 
\begin{equation*}
M_{p,q} := \begin{pmatrix}p^{-1}&y&0&0\\0&q&0&0\\0&0&\overline{p}&0\\0&0&l&\overline{q}^{-1}\end{pmatrix}\times 1_{2},
\end{equation*}
with $l, q$ chosen so that $l\overline{q} \equiv 1 \pmod{ \overline{p}}$ and $y = -q\overline{l}/p$. Then $M_{p,q} \in C_{3,2}(K)$ and we can then see that the product of equation \eqref{xi reps} belongs to $\Gamma_{3}$ and has the claimed form.
\end{proof}
Our aim now is to show an analogue of \cite[Theorem $2.3$]{heim}. For any $n \geq 1$, if $M \in \Gamma_{n}, \textup{ }Z \in \mathbb{H}_n$, we define the following quantities:
\begin{equation*}
    \chi^{k,s}(M, Z) := j(M,Z)^{-k}|j(M,Z)|^{-2s}, \textup{ }\delta(Z) := \det\left( \operatorname{Im}Z\right).
\end{equation*}
Then, we have $\delta\left(M\langle Z\rangle\right) = |j(M,Z)|^{-2}\delta(Z)$, which follows from \cite[Theorem II.1.7, (c)]{Krieg_Book}.
\begin{proposition}\label{prop: eisenstein}
Let $k \equiv 0 \pmod 4$ and $k+2\textup{Re}(s) > 10$. Let also $F \in S_{2}^{k}$, $z_1 \in \mathbb{H}_1$ and $z_2 \in \mathbb{H}_2$. We then have
\begin{equation*}
    E_{3,2}^{k}\left([z_1, z_2], F; s\right) = E_{1,0}^{k}(z_1;s)F(z_2) + E_{2,1}^{k}(z_2, F_{z_1};s) \textup{ }+
\end{equation*}
\begin{multline*}
+\textup{ }\delta(z_1)^{s}\delta(z_2)^{s}\sum_{p,q}\sum_{\substack{\gamma_1 \in C_{1,0}\backslash \Gamma_1\\ \gamma_2 \in C_{2,1}\backslash \Gamma_{2}}}\chi^{k,s}(\gamma_1, z_1)\chi^{k,s}(\gamma_2, z_2)\times \\\times F\left(\begin{pmatrix}N(q)\gamma_1 \langle z_1\rangle && 0 \\ 0 && 0\end{pmatrix}+(\gamma_2\langle z_2\rangle)\left[\begin{pmatrix}\overline{p}&&0\\0&&1_{n-1}\end{pmatrix}\right]\right)\times
\end{multline*}
\begin{equation*}
\times \textup{ }\delta\left(\begin{pmatrix}N(q)\gamma_1 \langle z_1\rangle && 0 \\ 0 && 0\end{pmatrix}+(\gamma_2\langle z_2\rangle)\left[\begin{pmatrix}\overline{p}&&0\\0&&1_{n-1}\end{pmatrix}\right]\right)^{-s} + \mathcal{E}^k\left([z_1, z_2], F; s\right),
\end{equation*}
where
\begin{equation*}
    \mathcal{E}^k\left([z_1, z_2], F; s\right) := \sum_{r \in \mathbb{Q}^{\times}/N_{K/\mathbb{Q}}(K^{\times})}\sum_{\gamma \in V_r}F(\gamma \langle [z_1,z_2]\rangle_{*})j(\gamma, [z_1,z_2])^{-k}\left(\frac{\det \operatorname{Im}\gamma \langle [z_1,z_2]\rangle}{\det{\operatorname{Im}\gamma \langle [z_1,z_2]\rangle_{*}}}\right)^{s}.
\end{equation*}
\begin{flushleft}
Here, for $\tau \in \mathbb{H}_{1}$, $F_{z_1}(\tau) := F\left([z_1, \tau]\right) \in S_{1}^{k}$ and $p, q$ are summed as in Corollary \ref{cor2}. We also remind the reader here that $N(q)$ denotes the norm of $q$ and $[a, b]$ the block diagonal matrix with diagonal blocks $a,b$ as in Notation.
\end{flushleft}
\end{proposition}
\begin{proof}
As we have shown after Definition \ref{eisenstein_defn}, for $k+2\textup{Re}(s) > 10$, the Eisenstein series $E_{3,2}^{k}\left(Z, F; s\right)$ is absolutely and uniformly convergent on compact subsets of $\mathbb{C}$. We split the Eisenstein series according to the representatives of Corollary \ref{cor2}. We can write
\begin{multline*}
    E_{3,2}^{k}\left([z_1, z_2], F; s\right) = \delta(z_1)^{s}\delta(z_2)^{s} \sum_{i=1}^{3}\sum_{M \in T_i}\chi^{k,s}(M, [z_1, z_2])F(M\langle [z_1,z_2]\rangle_{*})\delta\left(M\langle[z_1,z_2]\rangle_{*}\right)^{-s} +\\+\mathcal{E}^k\left([z_1, z_2], F; s\right).
\end{multline*}
For the representatives of $T_1, T_2$ in Corollary \ref{cor2}, the proof is exactly the same as in \cite[Theorem $2.3$]{heim}.\\\\
For a representative $M$ of $T_3$, write $M = (\xi^{p,q}\times 1_{2})(\gamma_1 \times \gamma_2)$ with $\gamma_1 \in C_{1,0}\backslash \Gamma_1$ and $\gamma_2 \in C_{2,1}\backslash \Gamma_2$. We write $\gamma_2\langle z_2\rangle = \begin{pmatrix}x_1&x_2\\x_3&x_4\end{pmatrix}$ and we then have
\begin{align*}
    (M\langle[z_1,z_2]\rangle)_{*} &= ((\xi^{p,q} \times 1_{2})[\gamma_1 z_1, \gamma_2 z_2])_{*} = \left(\begin{pmatrix}*&*&0\\q&p&0\\0&0&1\end{pmatrix}\begin{pmatrix}\gamma_1\langle z_1\rangle&0&0\\0&x_1&x_2\\0&x_3&x_4\end{pmatrix}\begin{pmatrix}*&\overline{q}&0\\0&\overline{p}&0\\0&0&1\end{pmatrix}\right)_{*} \\
    &=\begin{pmatrix}N(q)\gamma_1\langle z_1\rangle + N(p)x_1 && px_2 \\ \overline{p}x_3 && x_4\end{pmatrix}=\begin{pmatrix}N(q)\gamma_1 \langle z_1\rangle && 0 \\ 0 && 0\end{pmatrix}+(\gamma_2\langle z_2\rangle)\left[\begin{pmatrix}\overline{p}&&0\\0&&1\end{pmatrix}\right].
\end{align*}
Also,
\begin{equation*}
    j((\xi^{p,q}\times 1_2)(\gamma_1\times \gamma_2), [z_1,z_2]) = j(\xi^{p,q} \times 1_{2}, [\gamma_1z_1, \gamma_2z_2])j(\gamma_1\times \gamma_2, [z_1,z_2]),
\end{equation*}
and $j(\xi^{p,q} \times 1_{2}, [\gamma_1z_1, \gamma_2z_2]) = \det(D)$, where we denote $\xi^{p,q}:=\begin{pmatrix}A&0\\0&D\end{pmatrix}$. \\

By unitarity, we have $D \overline{A}^{t} = 1_2$ so $\det (D)\cdot \overline{\det (A)}=1_{2}$, so $N(\det(D))=1$ and $\det(D)$ is in $\mathbb{Z}[i]$, which shows that $\det(D)\in \{\pm 1, \pm i\}$. As $k \equiv 0 \pmod 4$, we get 
\begin{equation*}
\chi^{k,s}((\xi \times 1_{2})(\gamma_1\times \gamma_2),[z_1,z_2]) = \chi^{k,s}(\gamma_1, z_1)\chi^{k,s}(\gamma_2, z_2),
\end{equation*}
and so the Proposition follows.
\end{proof}
We now write $Z = \begin{pmatrix}\tau&z_1\\z_2&\omega\end{pmatrix} \in \mathbb{H}_2, \textup{ }W \in \mathbb{H}_1$ and consider the Fourier-Jacobi expansions of $F, G$ and the Fourier expansion of $h$ as follows:
\begin{equation}\label{fourier-expansions}
    F(Z) = \sum_{m=1}^{\infty}\phi_{m}(\tau, z_1,z_2)e^{2\pi i m \omega}, \textup{ } G(Z) = \sum_{m=1}^{\infty}\psi_{m}(\tau, z_1,z_2)e^{2\pi i m \omega}, \textup{ } h(W) = \sum_{n=1}^{\infty}a_ne^{2\pi i n W}.
\end{equation}
We state and prove the main Theorem of this section, which will give us the Dirichlet series with which we are going to work with.
\begin{theorem}\label{integral_representation_theorem}
Let $k \equiv 0 \pmod 4$. Let $F,G \in S_2^{k}$ and $h \in S_1^{k}$. Then, for $\textup{Re}(s) > 5-k/2$, we have
\begin{multline*}
    \left\langle\left\langle E_{3,2}^{k}\left(\begin{pmatrix}W&0\\0&Z\end{pmatrix}, F;s\right), G(Z)\right\rangle, h(W)\right\rangle = (4\pi)^{-(2k+s-4)}\frac{\Gamma(2k+s-4)\Gamma(k+s-3)\Gamma(k+s-1)}{\Gamma(2k+2s-4)} \times \\ \times D_{F,G,h}(s) + R_{F,G,h}(s),
\end{multline*}
where $R_{F,G,h}(s) := \langle \langle \mathcal{E}^{k}\left([W, Z], F; s\right), G(Z)\rangle, h(W)\rangle$, and we define the Dirichlet series
\begin{equation*}
D_{F, G, h}(s) := \sum_{p,q}\sum_{m=1}^{\infty}\langle \phi_m\mid U_{p}, \psi_{mN(p)}\rangle\overline{a_{mN(q)}}N(p)^{-(k+s-3)}N(q)^{-(k+s-1)}m^{-(2k+s-4)},
\end{equation*}
with $U_p$ the operator defined by
\begin{align*}
    U_p : J_{k,m}^{2} &\longrightarrow J_{k, mN(p)}^{2} \\ \phi_{m}(\tau, z_1, z_2) &\longmapsto \phi_m(\tau, \overline{p}z_1, pz_2),
\end{align*}
and $p,q$ summed as in Corollary \ref{cor2}.
\end{theorem}
\begin{proof}
From Proposition \ref{prop: eisenstein}, we can rewrite the Eisenstein series as a sum involving four summands. Clearly, $R_{F,G,h}(s)$ corresponds to the summand $\mathcal{E}^{k}([Z,W], F;s)$. We will now deal with the third one. This can be written as (the summations are as in Corollary \ref{cor2}):
\begin{flalign*}
    I_3 &:= \int_{\Gamma_1\backslash \mathbb{H}_1}\int_{\Gamma_2\backslash\mathbb{H}_2}\delta(W)^{k+s}\delta(Z)^{k+s}\sum_{p,q}\sum_{\gamma_1, \gamma_2}\chi^{k,s}(\gamma_1, W)\chi^{k,s}(\gamma_2,Z) \times\\
    &\times F\left(\begin{pmatrix}N(q)\gamma_1 \langle W\rangle & 0 \\ 0 & 0\end{pmatrix}+(\gamma_2\langle Z\rangle)\left[\begin{pmatrix}\overline{p}&0\\0&1\end{pmatrix}\right]\right)\delta\left(\begin{pmatrix}N(q)\gamma_1 \langle W\rangle & 0 \\ 0 & 0\end{pmatrix}+(\gamma_2\langle Z\rangle)\left[\begin{pmatrix}\overline{p}&0\\0&1\end{pmatrix}\right]\right)^{-s}\times \\
    &\times \overline{G(Z)}\overline{h(W)}\hbox{d}^{*}W\hbox{d}^{*}Z.
\end{flalign*}
Now, using the automorphy condition for $G$ and $h$, we have 
\begin{align*}
    G(Z) &= (G\mid_{k}\gamma_2)(Z) = j(\gamma_2, Z)^{-k}G(\gamma_2 \langle Z\rangle ), \\
    h(W) &= (h\mid_{k}\gamma_1)(W) = j(\gamma_1, W)^{-k}h(\gamma_1 \langle W\rangle ).
\end{align*}
Also $\delta(\gamma_2 \langle Z\rangle ) = |j(\gamma_2, Z)|^{-2}\delta(Z)$, \textup{ }$\delta(\gamma_1\langle W\rangle ) = |j(\gamma_1, W)|^{-2}\delta(W)$, so
\begin{align*}
    \overline{G(Z)}\delta(Z)^{k+s}\chi^{k,s}(\gamma_2, Z) &= \overline{j(\gamma_2, Z)}^{-k}\overline{G(\gamma_2 \langle Z\rangle )}\delta(\gamma_2\langle Z\rangle)^{k+s}|j(\gamma_2, Z)|^{2(k+s)}j(\gamma_2, Z)^{-k}|j(\gamma_2, Z)|^{-2s}\\&=\overline{G(\gamma_2 \langle Z\rangle )}\delta(\gamma_2 \langle Z\rangle )^{k+s},
\end{align*}
and similarly for $h$. Hence, by the usual "unfolding" trick, we obtain:
\begin{equation*}
    I_3 = \int_{C_{1,0}\backslash\mathbb{H}_1}\int_{C_{2,1}\backslash \mathbb{H}_2} \sum_{p,q} F\left(\begin{pmatrix}N(q)W & 0 \\ 0 & 0\end{pmatrix}+Z\left[\begin{pmatrix}\overline{p}&0\\0&1\end{pmatrix}\right]\right)\delta\left(\begin{pmatrix}N(q)W & 0 \\ 0 & 0\end{pmatrix}+Z\left[\begin{pmatrix}\overline{p}&0\\0&1\end{pmatrix}\right]\right)^{-s}\times
\end{equation*}
\begin{equation*}
\times\overline{G(Z)}\overline{h(W)}\delta(Z)^{k+s}\delta(W)^{k+s}\hbox{d}^{*}Z\hbox{d}^{*}W.
\end{equation*}
We now consider the matrix $M := \m{I& \\ &I}$, where $I := \m{&1\\1&}$. We then check that $M \in \Gamma_2$ and from Definition \ref{hermitian_modular_form}, we have 
\begin{equation*}
    F\left(Z\left[\m{0&1\\1&0}\right]\right) = F(Z) \textup{ }\forall Z \in \mathbb{H}_2,
\end{equation*}
as $k$ is even and $M\langle Z\rangle = Z[I]$. The same also holds for $G$. In particular, this shows
\begin{equation*}
    F\left(\begin{pmatrix}N(q)W & 0 \\ 0 & 0\end{pmatrix}+Z\left[\begin{pmatrix}\overline{p}&0\\0&1\end{pmatrix}\right]\right) = F\left(\begin{pmatrix}0 & 0 \\ 0 & N(q)W\end{pmatrix}+Z\left[\m{0&1\\1&0}\right]\left[\begin{pmatrix}1&0\\0&\overline{p}\end{pmatrix}\right]\right).
\end{equation*}
What is more, we can directly compute
\begin{equation*}
    \delta\left(\begin{pmatrix}N(q)W & 0 \\ 0 & 0\end{pmatrix}+Z\left[\m{0&1\\1&0}\right]\left[\begin{pmatrix}\overline{p}&0\\0&1\end{pmatrix}\right]\right) = 
    \delta\left(\begin{pmatrix}0 & 0 \\ 0 & N(q)W\end{pmatrix}+Z\left[\begin{pmatrix}1&0\\0&\overline{p}\end{pmatrix}\right]\right).
\end{equation*}
By now setting $Z \longmapsto M \langle Z\rangle$ and using the fact that $M^{-1}C_{2,1}M = P_{2,1}$, we can rewrite $I_3$ as ($C_{1,0}=P_{1,0}$ as well):
\begin{equation*}
        I_3 = \int_{P_{1,0}\backslash\mathbb{H}_1}\int_{P_{2,1}\backslash \mathbb{H}_2} \sum_{p,q} F\left(\begin{pmatrix}0 & 0 \\ 0 & N(q)W\end{pmatrix}+Z\left[\begin{pmatrix}1&0\\0&\overline{p}\end{pmatrix}\right]\right)\delta\left(\begin{pmatrix}0& 0 \\ 0 & N(q)W\end{pmatrix}+Z\left[\begin{pmatrix}1&0\\0&\overline{p}\end{pmatrix}\right]\right)^{-s}\times
\end{equation*}
\begin{equation*}
\times\overline{G(Z)}\overline{h(W)}\delta(Z)^{k+s}\delta(W)^{k+s}\hbox{d}^{*}Z\hbox{d}^{*}W.
\end{equation*}
Now, a fundamental domain for the action of $P_{1,0}$ on $\mathbb{H}_1$ is
\begin{equation*}
    \mathcal{F} = \{z\in \mathbb{H}_1 \mid z = x+iy, \textup{ }0 \leq x \leq 1\},
\end{equation*}
while a fundamental domain for the action of $P_{2,1}$ on $\mathbb{H}_2$ is (cf. \cite[p. 2907]{gritsenko})
\begin{equation*}
    \left\{\begin{pmatrix}\tau&z_1\\z_2&\omega\end{pmatrix}\mid (\tau, z_1, z_2) \in \mathcal{F}^{J},\textup{ } y_{\omega} > |z_1-\overline{z_2}|^2/4y_{\tau},\textup{ } 0\leq x_{\omega}\leq 1\right\},
\end{equation*}
using the notation of equation \eqref{imaginary,real} for the real and imaginary parts. Hence, we have
\begin{flalign*}
    I_3 &= \sum_{p,q}\int_{\mathcal{F}^{J}}\hbox{d}\tau \hbox{d}z_{1}\hbox{d}z_{2}\int_{y_{\omega}>|z_1-\overline{z_2}|^2/4y_{\tau}}\hbox{d}y_{\omega}\int_{0}^{1}\hbox{d}x_{\omega}\int_{0}^{1}\hbox{d}x_{W}\int_{0}^{\infty}\hbox{d}y_{W}\delta(Z)^{k+s-4}\times \\
    &\times\sum_{m=1}^{\infty}\phi_m(\tau, \overline{p}z_1, pz_2)e^{2\pi im(N(q)W+N(p)\omega)}\sum_{n=1}^{\infty}\overline{a_n}e^{-2\pi i n \overline{W}}\sum_{k=1}^{\infty}\overline{\psi_{k}(\tau, z_1, z_2)}e^{-2\pi i k \overline{\omega}}\times
\end{flalign*}
\begin{equation*}
\times\delta\left(\begin{pmatrix}0& 0 \\ 0 & N(q)W\end{pmatrix}+Z\left[\begin{pmatrix}1&0\\0&\overline{p}\end{pmatrix}\right]\right)^{-s}.
\end{equation*}
We now first perform the integration over $x_{\omega}$ and $x_{W}$. For $x_{\omega}$, we have
\begin{equation*}
    \int_{0}^{1}e^{2\pi i mN(p)x_{\omega} - 2\pi i k x_{\omega}}\hbox{d}x_{\omega},
\end{equation*}
which is zero, unless $k = mN(p)$, in which case the integral is $1$. Similarly for $x_{W}$, the integral we need to compute is
\begin{equation*}
    \int_{0}^{1}e^{2\pi i mN(q)x_{W}-2\pi i n x_{W}}\hbox{d}x_{W},
\end{equation*}
which is again zero unless $n = mN(q)$, in which case it is $1$. These are the only terms we need to integrate as the real parts of $\omega$ and $W$ do not appear as arguments of $\delta$ by definition. We now substitute $t = y_{\omega} - |z_1-\overline{z_2}|^2/4y_{\tau}$ and compute
\begin{equation*}
    \delta(Z) = \det\left(\frac{1}{2i}\left(Z - \overline{Z}^{t}\right)\right) = y_{\tau}t,
\end{equation*}
and
\begin{equation*}
    \delta\left(\begin{pmatrix}0& 0 \\ 0 & N(q)W\end{pmatrix}+Z\left[\begin{pmatrix}1&0\\0&\overline{p}\end{pmatrix}\right]\right) = \delta\left(\begin{pmatrix}\tau & \overline{p}z_1 \\ pz_2 & N(q)W+N(p)\tau'\end{pmatrix}\right) = y_{\tau}(N(q)y_W+N(p)t).
\end{equation*}
So, the integral $I_3$ becomes
\begin{equation*}
\sum_{p,q}\int_{\mathcal{F}^{J}}\sum_{m=1}^{\infty}\phi_m(\tau, \overline{p}z_1, pz_2)\overline{\psi_{mN(p)}(\tau, z_1,z_2)}\overline{a_{mN(q)}}y_{\tau}^{k-4}e^{-\pi m (|z_1-\overline{z_2}|^2/y_{\tau})}\hbox{d}\tau \hbox{d}z_1 \hbox{d}z_2 \times
\end{equation*}
\begin{equation*}
\times\int_{0}^{\infty}\hbox{d}t\int_{0}^{\infty}\hbox{d}y_{W}t^{k+s-4}(N(q)y_{W}+N(p)t)^{-s}e^{-4\pi m (N(q)y_{W}+N(p)t)}y_{W}^{s+k-2}=
\end{equation*}
\begin{equation*}
    = \sum_{p,q}\sum_{m=1}^{\infty}\langle \phi_m|U_{p}, \psi_{mN(p)}\rangle\overline{a_{mN(q)}}\int_{0}^{\infty}\int_{0}^{\infty}\frac{\hbox{d}t}{t}\frac{\hbox{d}y_{W}}{y_{W}}t^{k+s-3}y_{W}^{s+k-1}(N(q)y_{W}+N(p)t)^{-s}e^{-4\pi m (N(q)y_{W}+N(p)t)}=
\end{equation*}
\begin{multline}\label{I_3}
    = (4\pi)^{-(2k+s-4)}\frac{\Gamma(2k+s-4)\Gamma(k+s-3)\Gamma(k+s-1)}{\Gamma(2k+2s-4)}\times\\
    \times\sum_{p,q}\sum_{m=1}^{\infty}\langle \phi_m|U_{p}, \psi_{mN(p)}\rangle\overline{a_{mN(q)}}N(p)^{-(k+s-3)}N(q)^{-(k+s-1)}m^{-(2k+s-4)},
\end{multline} 
by using the fact that (cf. \cite[Theorem 2.6]{heim})
\begin{equation*}
    \int_{0}^{\infty}\int_{0}^{\infty}\frac{\hbox{d}x}{x}\frac{\hbox{d}y}{y}x^{\alpha}y^{\beta}\left(\frac{xy}{x+y}\right)^{\gamma}e^{-(x+y)} = \frac{\Gamma(\alpha+\beta+\gamma)\Gamma(\alpha+\gamma)\Gamma(\beta+\gamma)}{\Gamma(\alpha+\beta+2\gamma)},
\end{equation*}
and substituting $x = 4\pi m N(p) t$, $y = 4\pi m N(q)y_{W}$, $\alpha = k-3$, $\beta = k-1$, $\gamma = s$. This formula follows after setting $u = x+y$ and then $t = x/u$ and using the Euler integral of the first kind
\begin{equation*}
    \int_{t=0}^{1}t^{z_1-1}(1-t)^{z_2-1}\hbox{d}t = \frac{\Gamma(z_1)\Gamma(z_2)}{\Gamma(z_1+z_2)}, \textup{ } \forall z_1,z_2 \in \mathbb{C}.
\end{equation*}
Let us now consider the second summand in the decomposition of Proposition \ref{prop: eisenstein}. We want to compute
\begin{equation*}
    I_2 := \int_{\Gamma_1 \backslash \mathbb{H}_1}\int_{\Gamma_2\backslash\mathbb{H}_2}\sum_{\gamma \in C_{2,1}\backslash\Gamma_2}j(\gamma, Z)^{-k}F_{W}\left(\gamma\langle Z\rangle_{*}\right)\left(\frac{\delta(\gamma\langle Z\rangle)}{\delta(\gamma\langle Z \rangle_{*})}\right)^{s}\overline{G(Z)}\overline{h(W)}\delta(Z)^{k}\delta(W)^{k}\hbox{d}^{*}Z\hbox{d}^{*}W.
\end{equation*}
Using again the automorphy condition for $G$, we obtain, by unfolding the integral, that
\begin{equation*}
    I_2 = \int_{\Gamma_1\backslash\mathbb{H}_1}\int_{C_{2,1}\backslash\mathbb{H}_2}F_{W}(Z_{*})\overline{G(Z)}\overline{h(W)}\delta(Z)^{k+s}\delta(Z_{*})^{-s}\hbox{d}^{*}Z\hbox{d}^{*}W.
\end{equation*}
Using the same reasoning for the interchange of the parabolic subgroups $C_{2,1}, \textup{ }P_{2,1}$ as in the case of $I_3$, we rewrite the inner integral as
\begin{equation*}
    I_2 = \int_{P_{2,1}\backslash\mathbb{H}_2}F\left(\begin{pmatrix}\tau & 0\\ 0 & W\end{pmatrix}\right)\overline{G(Z)}\delta(Z)^{k+s}y_{\tau}^{-s}\hbox{d}^{*}Z,
\end{equation*}
and by using the Fourier-Jacobi expansions and the fundamental domains mentioned before, we have
\begin{equation*}
    I_2 = \int_{\mathcal{F}^{J}}\hbox{d}\tau \hbox{d}z_1 \hbox{d}z_2 \int_{y_{\omega}>{|z_1-\overline{z_2}|^2/4y_{\tau}}}\int_{0}^{1}\hbox{d}x_{\omega}\sum_{m=1}^{\infty}\phi_m(\tau,0,0)e^{2\pi i mW}\sum_{n=1}^{\infty}\overline{\psi_m(\tau, z_1,z_2)}e^{-2\pi i m \overline{\omega}}\delta(Z)^{k+s-4}y_{\tau}^{-s}.
\end{equation*}
We can then see that this is zero by calculating the integral over $x_{\omega}$, i.e.
\begin{equation*}
    \int_{0}^{1}e^{-2\pi imx_{\omega}}dx_{\omega} = 0.
\end{equation*}
Therefore, $I_2=0$. Finally, we will show that the integral involving the first summand of Proposition \ref{prop: eisenstein} is zero. We have, after a direct computation
\begin{equation*}
    \langle\langle E_{1,0}(W;s)F(Z), G(Z)\rangle, h(W)\rangle = \langle F(Z), G(Z)\rangle\langle E_1(W;s), h(W)\rangle,
\end{equation*}
and we will show the second inner product is zero. But
\begin{equation*}
    \langle E_{1,0}(W;s), h(W)\rangle = \int_{\Gamma_1 \backslash \mathbb{H}_1} \sum_{\gamma \in C_{1,0}\backslash\Gamma_1}j(\gamma, W)^{-k}\delta(\gamma W)^{s}\overline{h(W)}\delta(W)^{k}\hbox{d}^{*}W.
\end{equation*}
By the usual unfolding trick, the above integral equals
\begin{equation*}
    \int_{C_{1,0}\backslash \mathbb{H}_1}\delta(W)^{k+s}\overline{h(W)}\hbox{d}^{*}W = \int_{x=0}^{1}\int_{y=0}^{\infty}\sum_{n=1}^{\infty}\overline{a_n}e^{-2\pi in(x-iy)}y^{k-2}\hbox{d}x\hbox{d}y = 0,
\end{equation*}
by looking at the integral
\begin{equation*}
    \int_{0}^{1}e^{-2\pi inx}\hbox{d}x = 0,
\end{equation*}
for all $n\geq 1$. Hence, $I_3$ is the only integral that has a non-zero contribution and the result now follows from equation \eqref{I_3}.
\end{proof}
\section{Inert primes}\label{inert primes}
This is the case which is closer to the situation considered by Heim in \cite{heim}. Our aim is to relate the Dirichlet series $D_{F,G,h}(s)$ of Theorem \ref{integral_representation_theorem} with $L$-functions, when $F,G,h$ are Hecke eigenforms.
\subsection{Hecke operators and weak rationality theorems}\label{hecke inert}
Throughout this section, $p$ is assumed to be a rational prime which remains prime in $\mathcal{O}_K$. Let us make a list of Hecke operators in $H(\Gamma_2, S_p^2)$ and $H(\Gamma_{1,1}, S_p^{1,1})$ and relations between them.
\begin{itemize}
    \item $T_{p} := \Gamma_2\textup{diag}(1,1,p,p)\Gamma_2$.
    \item $T_{1,p} := \Gamma_2\textup{diag}(1,p,p^2,p)\Gamma_2$.
    \item $\Delta_{p} := \Gamma_{2}\textup{diag}(p,p,p,p)\Gamma_{2} = \Gamma_{2}\textup{diag}(p,p,p,p)$.
    \item $T^{J}(p) := \Gamma_{1,1}\textup{diag}(1,p,p^2,p)\Gamma_{1,1}$.
    \item $\nabla_{p} := \sum_{a \in \mathbb{Z}/p\mathbb{Z}}\Gamma_{1,1}\begin{pmatrix}p&0&0&0\\0&p&0&a\\0&0&p&0\\0&0&0&p\end{pmatrix} = \sum_{a \in \mathbb{Z}/p\mathbb{Z}}\nabla{a}$.
    \item $\Delta_{p^{\delta}} := \Gamma_{1,1}\textup{diag}(p^{\delta},p^{\delta},p^{\delta},p^{\delta})\Gamma_{1,1} = \Gamma_{1,1}\textup{diag}(p^{\delta},p^{\delta},p^{\delta},p^{\delta}), \textup{ }\delta \geq 1$.
\end{itemize}
Also, for any operator $X(p)$, we write $X^{r}(p)$ to denote $\Delta_{p}^{-1}X(p)$. Note here that we use the same notation for $\Delta_p$ as an element of $H_p^{2}$ and as an element of $H_p^{1,1}$. Finally, we define
\begin{equation}\label{t(p)}
    T_{\pm}(p^{\delta}):=j_{\pm}(T(p^{\delta})), \textup{ }\Lambda_{\pm}(p^{\delta}) := j_{\pm}\left(\Gamma_{1}\textup{diag}(p^{\delta}, p^{\delta})\Gamma_1\right), \textup{ }\delta \geq 1,
\end{equation}
where $j_{\pm}$ are the embeddings of equation \eqref{+- embeddings} and $T(p^{\delta})$ as in Definition \ref{maass_defn}. Therefore
\begin{itemize}
    \item $\Lambda_{-}(p) = \Gamma_{1,1}\textup{diag}(p,p^2,p,1)\Gamma_{1,1} = \Gamma_{1,1}\textup{diag}(p,p^2,p,1)$.
    \item $\Lambda_{+}(p) = \Gamma_{1,1}\textup{diag}(p,1,p,p^2)\Gamma_{1,1}$.
    \item $T_{-}(p) = \Gamma_{1,1}\textup{diag}(1,p,p,1)\Gamma_{1,1}$.
    \item $T_{+}(p) = \Gamma_{1,1}\textup{diag}(1,1,p,p)\Gamma_{1,1}$.
\end{itemize}
As $\Lambda_{-}(p)$ has only one right coset and for any $X \in H_p^{1,1}$, we have $(j_{-}(X))^{*}=j_{+}(X)$ and vice-versa, where $*$ is the anti-homomorphism of equation \eqref{antihomomorphism}, we have
\begin{equation}\label{lambda_inert}
    \Lambda_{\pm}(p^{\delta}) = \Lambda_{\pm}(p^{\delta-1})\Lambda_{\pm}(p), \textup{ }\forall \delta \geq 1.
\end{equation}
These also imply $\Lambda_{\pm}(p^{\delta}) = \Lambda_{\pm}(p)^{\delta}$ for all $\delta \geq 1$.
\begin{proposition}\label{satake}
Let $\Phi$ denote the Satake mapping of Section \ref{hecke algebras} for the inert prime $p$. We have
\begin{itemize}
    \item $\Phi(T_{p}) = x_0 + p^{-1}x_0x_1+p^{-1}x_0x_2 + p^{-2}x_0x_1x_2 = x_0(1+p^{-1}x_1)(1+p^{-1}x_2)$.
    \item $\Phi(T_{1,p}) = p^{-2}x_0^2x_1 + p^{-2}x_0^2x_2 + p^{-4}x_0^2x_1^2x_2 + p^{-4}x_0^2x_1x_2^2 + p^{-6}(p^2+1)(p-1)x_0^2x_1x_2$.
    \item $\Phi(\Delta_p) = p^{-6}x_0^2x_1x_2$.
\end{itemize}
\end{proposition}
\begin{proof}
From Definition \ref{gritsenko l-function}, we have
\begin{equation*}
    q_p^{(2)}(t) = (1-x_0t)(1-p^{-1}x_0x_1t)(1-p^{-1}x_0x_2t)(1-p^{-2}x_0x_1x_2t).
\end{equation*}
Also, from \cite[Lemma $3.6$]{gritsenko}, we have
\begin{equation*}
    \Phi^{-1}(q_p^{(2)}(t)) = 1-T_pt + (pT_{1,p}+p(p^3+p^2-p+1)\Delta_p)t^2-p^{4}\Delta_pT_pt^3+p^{8}\Delta_p^{2}t^4.
\end{equation*}
By comparing these two expressions, the Proposition follows.
\end{proof}
Let now $D_p^{(2)}(X) := Z_{p}^{(2)}(p^{-3}X)$, with $Z_p^{(2)}$ as in Definition \ref{standard l-function}.
\begin{proposition}\label{rankin}
We have 
\begin{equation*}
    D_p^{(2)}(X) = 1-B_1X+B_2X^2-B_1X^3+X^4,
\end{equation*}
where 
\begin{align*}
    &B_1 = p^{-3}\Delta_p^{-1}(T_{1,p} - (p^2+1)(p-1)\Delta_p),\\
    &B_2 = p^{-4}\Delta_p^{-1}(T_{p}^2 - 2pT_{1,p}-2p(p^2-p+1)\Delta_p).
\end{align*}
\end{proposition}
\begin{proof}
This follows by direct verification, after applying the Satake isomorphism and using Proposition \ref{satake}. We remind the reader here that $Z_{p}^{(2)}(X) = \Phi^{-1}(z_{p}^{(2)}(X))$, where
\begin{equation*}
    z_{p}^{(2)}(X) = \prod_{i=1}^{2}(1-p^{4}x_{i,p}^{-1}X)(1-p^{2}x_{i,p}X).
\end{equation*}
This also gives the $\Phi$-image of $D_{p}^{(2)}$.
\end{proof}
We now have the following Proposition regarding the factorisation of $D_{p}^{(2)}$.
\begin{proposition}\label{prop: rankin_prime}
We have the following factorisation in $H_{p}^{1,1}[X]$:
\begin{equation*}
    D_{p}^{(2)}(X) = (1 - p^{-3}\Delta_{p}^{-1}\Lambda_{-}(p)X)S^{(2)}(X)(1-p^{-3}\Delta_{p}^{-1}\Lambda_{+}(p)X),
\end{equation*}
where
\begin{equation*}
    S^{(2)}(X) = S_0 - S_1X + S_2X^2 - S_3X^3,
\end{equation*}
with
\begin{itemize}
    \item $S_0 = 1$.
    \item $S_1 = p^{-3}(T^{J,r}(p) + \nabla_{p}^{r} - p(p^2-p+1))$.
    \item $S_2 = p^{-4}\Delta_{p}^{-1}T_{+}(p)T_{-}(p) - p^{-3}T^{J,r}(p) - 2p^{-3}\nabla_{p}^{r} - p^{-2}(p-2)$.
    \item $S_3 = p^{-3}(\nabla_{p}^{r} - p)$.
\end{itemize}
\end{proposition}
\begin{proof}
This can be verified directly, by using the following identities/relations which can be found in the proof \cite[Proposition 3.2]{gritsenko}, or be proved directly.
\begin{itemize}
    \item $\epsilon(T_{1,p}) = T^{J}(p) + \Lambda_{-}(p)+\Lambda_{+}(p) + \nabla_{p} - \Delta_p$.
    \item $\epsilon(T_{p}) = T_{-}(p) + T_{+}(p)$.
    \item $T_{-}(p)T_{+}(p) = pT^{J}(p) + (p^3+p^4)\Delta_p$.
    \item $\Lambda_{-}(p)T_{+}(p) = p^3\Delta_pT_{-}(p)$.
    \item $T_{-}(p)\Lambda_{+}(p) = p^3\Delta_pT_{+}(p)$.
    \item $\Lambda_{-}(p)\Lambda_{+}(p) = p^{6}\Delta_{p}^2$.
    \item $\Lambda_{-}^{r}(p)\nabla_{p}^{r} = p\Lambda_{-}^{r}(p)$.
    \item $\nabla_{p}^{r}\Lambda_{+}^{r}(p) = p\Lambda_{+}^{r}(p)$.
\end{itemize}
Here, $\epsilon$ denotes the embedding of $H(\Gamma_2, S^{2})$ to $H(\Gamma_{1,1}, S^{1,1})$, as described in Lemma \ref{lemma:embeddings}.
\end{proof}
Now, if $F \in S_{2}^{k}$ has a Fourier-Jacobi expansion as in equation \eqref{fourier-expansions} and $Q_{p}^{(2)}$ denotes the $p$-factor of Gritsenko's $L$-function, as in Definition \ref{gritsenko l-function}, we have the following weak rationality theorems.
\begin{proposition}\label{prop: spinor primes}
Let $F \in S_{2}^{k}$ be a Hecke eigenform for $H(\Gamma_2, S^{2})$ and $m \geq 1$. Then
\begin{align*}
    Q_{p,F}^{(2)}(X)\sum_{\delta \geq 0}\phi_{mp^{\delta}} \textup{ }|\textup{ }T_{+}(p^{\delta})X^{\delta} = \left(\phi_m - \phi_{m/p}\textup{ }|\textup{ }T_{-}(p)X + p\phi_{m/p^2} \textup{ } |\textup{ } \Lambda_{-}(p)X^2\right)\mid(1+p(\nabla_{p}-p\Delta_{p})X^2),
\end{align*}
where $\phi_{m}\mid (1+p(\nabla_{p}-p\Delta_{p})X^2) =$ $\begin{cases}
    \phi_{m} & \textup{if } p\mid m\\
    (1-p^{2k-6}X^2)\phi_m & \textup{otherwise}
\end{cases}$.
\end{proposition}
\begin{proof}
We follow the same proof as in \cite[Corollary 1]{gritsenko_1995}. Then, the result follows from \cite[Proposition $3.2$]{gritsenko}. We will just show the computations for the last claim of our Proposition. We have
\begin{equation*}
    (F \textup{ }|\textup{ }\Delta_{p})(Z) = (p^2)^{2k-4}(p^2)^{-k}F(Z) = p^{2k-8}F(Z),
\end{equation*}
and so $\phi_m \mid \Delta_p = p^{2k-8}\phi_{m}$.
Also,
\begin{equation*}
    (F\textup{ }|\textup{ }\nabla{a})(Z) = (p^2)^{2k-4}(p^2)^{-k}F\left(\begin{pmatrix}\tau&z_1\\z_2&\tau'+a/p\end{pmatrix}\right) = p^{2k-8}\sum_{m=1}^{\infty}\phi_m(\tau,z_1,z_2)e^{2 \pi im\tau'}e^{2\pi i m a/p},
\end{equation*}
so $\phi_m \mid \nabla_p = \left(p^{2k-8}\sum_{a=0}^{p-1}e^{2\pi i ma/p}\right)\phi_m = \begin{cases}0 &\textup{if } (m,p)=1\\ p^{2k-7}\phi_m& \textup{otherwise}\end{cases},$\\
from which the result follows.
\end{proof}
\begin{proposition}\label{rat2}
Let $F\in S_{2}^{k}$ be a Hecke eigenform for $H(\Gamma_{2}, S^2)$ and $m \geq 1$. Then
\begin{equation*}
    D_{p,F}^{(2)}(X)\sum_{\delta \geq 0}\phi_{mp^{2\delta}} \mid (\Delta_{p^{\delta}}^{-1}\Lambda_{+}(p^{\delta}))(p^{-3}X)^{\delta} = \phi_{m}\mid S^{(2)}(X) - \phi_{m/p^{2}}\mid  (\Delta_{p}^{-1}\Lambda_{-}(p)S^{(2)}(X))p^{-3}X.
\end{equation*}
\end{proposition}
\begin{proof}
    This follows from Proposition \ref{prop: rankin_prime}, using the same techniques as in \cite[Corollary 1]{gritsenko_1995}.
\end{proof}
In a similar fashion to Heim in \cite[page 227]{heim} now, we have that the action of the operators $T_{+}(p), \Lambda_{+}(p), \nabla_{p}^{r}(p)$ on Fourier-Jacobi forms of index coprime to $p$ is identical to zero. This leads to the definition of the following polynomials:
\begin{align*}
    S^{(2)}(X)^{\textup{factor}} &:= 1 -(p^{-3}T^{J,r}(p)-p^{-2}+p^{-1})X + p^{-2}X^2,\\
    S^{(2)}(X)^{\textup{prim}} &:= 1 - (p^{-3}T^{J,r}(p)-p^{-2}(p^2-p+1))X + (-p^{-3}T^{J,r}(p)-p^{-2}(p-2))X^2 + p^{-2}X^3\\&= S^{(2)}(X)^{\textup{factor}}(1+X).
\end{align*}

Hence, $\phi \textup{ }|\textup{ } S^{(2)}(X) = \phi \textup{ }|\textup{ } S^{(2)}(X)^{\textup{prim}}$ if $\phi \in J_{k,m}^2$ with $\gcd(m,p)=1$. We now have the following Lemma:
\begin{lemma}\label{index_p}
    Let $\phi \in J_{k,p}^2$. Then $\phi \textup{ }|\textup{ } S^{(2)}(X)T_{+}(p) = \phi \textup{ }|\textup{ } T_{+}(p)S^{(2)}(X)^{\textup{factor}}$.
\end{lemma}
\begin{proof}
    The proof follows by the following results:
    \begin{itemize}
        \item $\phi \mid \nabla_{p}^{r} = p \phi$.
        \item $\phi \mid T_{+}(p)\nabla_{p}^{r} = 0$, because $\phi \mid T_{+}(p)$ will have index $1$.
        \item $\phi \mid \left[T^{J,r}(p), \textup{ }T_{+}(p)\right] = \phi \mid (p^3T_{+}(p)-\nabla_{p}^{r}T_{+}(p) = (p^3-p)\phi \mid T_{+}(p)$, by the first point.
    \end{itemize}
Here $\left[T^{J,r}(p), \textup{ }T_{+}(p)\right] := T^{J,r}(p)T_{+}(p) - T_{+}(p)T^{J,r}(p)$ denotes the commutator. We will now give the proof of the third point. By the proof of Proposition \ref{prop: rankin_prime}, we have
\begin{itemize}
    \item $\epsilon(T_{1,p}) = T^{J}(p) + \Lambda_{-}(p)+\Lambda_{+}(p) + \nabla_{p} - \Delta_p$,
    \item $\epsilon(T_{p}) = T_{-}(p) + T_{+}(p)$.
\end{itemize}
Now, $H_p^{2}$ is a commutative Hecke algebra and as $\epsilon$ is a ring homomorphism, we have 
\begin{equation*}
    \epsilon(T_{1,p})\epsilon(T_{p}) = \epsilon(T_{p})\epsilon(T_{1,p}).
\end{equation*}
By then considering the elements whose product has signature $p$ (see Definition \ref{signature} and \cite[Section 3.3]{heim}), we obtain
\begin{equation*}
    T^{J}(p)T_{+}(p) + \Lambda_{+}(p)T_{-}(p) + (\nabla_p-\Delta_p)T_{+}(p) = T_{+}(p)T^{J}(p) + T_{-}(p)\Lambda_{+}(p) + T_{+}(p)(\nabla_p-\Delta_p),
\end{equation*}
from which the result follows, as $\phi \mid \Lambda_{+}(p) = 0$ for $\phi \in J_{k,p}^2$ and $T_{-}(p)\Lambda_{+}(p) = p^3\Delta_pT_{+}(p)$.
\end{proof}
\subsection{Calculation of the Dirichlet Series}\label{dirichlet series inert}
Let now $F,G,h$ have Fourier expansions as in equation \eqref{fourier-expansions}. In what follows, we will assume that $F,G,h$ are all Hecke eigenforms for their corresponding Hecke rings ($h$ is assumed to be normalised) and $F$ is in the Maass space, as we have defined in Definition \ref{maass_defn}. Also, we will always assume that $F,G, h$ all have totally real Fourier coefficients. This is a technical assumption which could be lifted. We can rewrite $D_{F,G,h}$ of Theorem \ref{integral_representation_theorem} as (we have $\phi_m = m^{3-k}\phi_1 \textup{ }|\textup{ } T_{-}(m)$ for all $m\geq 1$ from Definition \ref{maass_defn})
\begin{align}\label{dirichlet_series_defn}
    D_{F,G,h}(s) \nonumber &=  4\sum_{l,\epsilon,m}\langle m^{3-k}\phi_1 \textup{ }|\textup{ }T_{-}(m)U_l, \psi_{mN(l)}\rangle a_{mN(\epsilon)}N(l)^{-(k+s-3)}N(\epsilon)^{-(k+s-1)}m^{-(2k+s-4)}\\
    &=4\beta_{k}\sum_{l,\epsilon, m}\langle\widetilde{\phi}_{1} \textup{ }|\textup{ }T_{-}(m)U_{l}, \widetilde{\psi}_{mN(l)}\rangle_{\mathcal{A}}a_{mN(\epsilon)}N(l)^{-s}N(\epsilon)^{-(k+s-1)}m^{-(2k+s-4)},
\end{align}
with $l,\epsilon \in \mathbb{Z}[i]$ coprime with their real parts positive and imaginary parts non-negative and $m \in \mathbb{N}$. Here, $\beta_k$ is the constant of Lemma \ref{inner_product_p_forms}.
Now, if $\phi \in J_{k,m}^2$, we have from Proposition \ref{fourier_jacobi_action} and the fact that $\Lambda_{-}(p)$ has a single right coset representative, that
\begin{equation*}
    \phi \textup{ }|\textup{ }\Lambda_{-}(p) = p^{3k-8}\tilde{\phi}\left(\begin{pmatrix}\tau&pz_1\\pz_2&p^2\tau'\end{pmatrix}\right)e^{-2\pi mp^2\tau'} = p^{3k-8}\phi(\tau, pz_1, pz_2) = p^{3k-8}\phi \textup{ }|\textup{ }U_{p},
\end{equation*}
with $U_p$ the operator defined in Theorem \ref{integral_representation_theorem}. We now define the $p$-part of the Dirichlet series
\begin{align}\label{p-part dirichlet}
    D_{F,G,h}^{(p)}(s) \nonumber &:= \sum_{l, \epsilon, m \geq 0}\langle \widetilde{\phi}_1\textup{ }|\textup{ }T_{-}(p^{m})U_{p^{l}}, \widetilde{\psi}_{p^{m+2l}}\rangle_{\mathcal{A}} a_{p^{m+2\epsilon}}p^{-2sl}p^{-2(k+s-1)\epsilon}p^{-(2k+s-4)m}\\
    &=\sum_{l, \epsilon, m \geq 0}\langle \widetilde{\phi}_1\textup{ }|\textup{ }T_{-}(p^{m})\Lambda_{-}(p^l), \widetilde{\psi}_{p^{m+2l}}\rangle_{\mathcal{A}}a_{p^{m+2\epsilon}}p^{-(3k+2s-8)l}p^{-2(k+s-1)\epsilon}p^{-(2k+s-4)m},
\end{align}
together with the condition that $\min(l, \epsilon) = 0$. The last line is obtained using the relation between $U_p$ and $\Lambda_{-}(p)$ (and hence of $\Lambda_{-}(p^{l})$ and $U_{p^{l}}$).\\\\
Now, with respect to the inner product of Fourier-Jacobi forms, we have by \cite[Proposition $5.1$]{gritsenko} that 
$\Lambda_{-}^{\textup{adj}}(p^l) = p^{(2k-6)l}\Lambda_{+}(p^l)$ and $T_{-}^{\textup{adj}}(p^l) = p^{(k-3)l}T_{+}(p^l)$ for any $l \geq 1$. This then gives that the adjoint of $\Lambda_{-}(p^l)$ is $\Lambda_{+}(p^l)$ for the inner product of $P$-forms and similarly the $P$-form adjoint for $T_{-}(p^l)$ is $T_{+}(p^l)$.\\\\
Let now $X := p^{-(k+s-1)}$ and $N := p^{k-1}$. Consider the Satake parameters $\alpha_1, \alpha_2$ of the modular form $h$ such that $\alpha_{1}+\alpha_{2} = a_{p}$ and $\alpha_1\alpha_2 = p^{k-1}$. Let also $X_i := \alpha_{i}p^{-(2k+s-4)}, \textup{ }i=1,2$. We write 
\begin{equation}\label{dirichlet_inert_dec}
D_{F,G,h}^{(p)}(s) = D_{(\epsilon)}(s) + D_{(l)}(s) - D_{(\epsilon, l)}(s),
\end{equation}
where the corresponding index means that this variable (or both) is $0$. Using the fact that
\begin{equation*}
a_{p^{m}} = \frac{\alpha_1^{m+1}-\alpha_2^{m+1}}{\alpha_1-\alpha_2},
\end{equation*}
and properties for the adjoint operators we mentioned above, we obtain:
\begin{align}
\begin{split}\label{first_part}
    D_{(\epsilon)}(s)(\alpha_1-\alpha_2) &= \alpha_1\sum_{l,m=0}^{\infty}\langle \tilde{\phi}_1, \tilde{\psi}_{p^{m+2l}}\mid T_{+}(p^{m})\Lambda_{+}(p^l)\rangle_{\mathcal{A}}p^{-(3k+2s-8)l}(\alpha_1p^{-(2k+s-4)})^{m}\\
    &- \alpha_2\sum_{l,m=0}^{\infty}\langle \tilde{\phi}_1, \tilde{\psi}_{p^{m+2l}}\mid T_{+}(p^{m})\Lambda_{+}(p^l)\rangle_{\mathcal{A}}p^{-(3k+2s-8)l}(\alpha_2p^{-(2k+s-4)})^{m}.
\end{split}
\end{align}
\begin{align}
\begin{split}\label{second_part}
    D_{(l)}(s)(\alpha_1-\alpha_2) &= \alpha_1\sum_{\epsilon,m=0}^{\infty}\langle \tilde{\phi}_1, \tilde{\psi}_{p^{m}}\mid T_{+}(p^{m})\rangle_{\mathcal{A}}(\alpha_1p^{-(k+s-1)})^{2\epsilon}(\alpha_1p^{-(2k+s-4)})^{m}\\
    &- \alpha_2\sum_{\epsilon,m=0}^{\infty}\langle \tilde{\phi}_1, \tilde{\psi}_{p^{m}}\mid T_{+}(p^{m})\rangle_{\mathcal{A}}(\alpha_2p^{-(k+s-1)})^{2\epsilon}(\alpha_2p^{-(2k+s-4)})^{m}.
\end{split}
\end{align}
\begin{align}
\begin{split}\label{third_part}
    D_{(\epsilon, l)}(s)(\alpha_1-\alpha_2) &= \alpha_1\sum_{m=0}^{\infty}\langle \tilde{\phi}_1, \tilde{\psi}_{p^{m}}\mid T_{+}(p^{m})\rangle_{\mathcal{A}}(\alpha_1p^{-(2k+s-4)})^{m}\\
    &- \alpha_2\sum_{m=0}^{\infty}\langle \tilde{\phi}_1, \tilde{\psi}_{p^{m}}\mid T_{+}(p^{m})\rangle_{\mathcal{A}}(\alpha_1p^{-(2k+s-4)})^{m}.
\end{split}
\end{align}
\begin{remark}\label{identity_theorem}
In the following, we want to show a relation of $D_{F,G,h}^{(p)}(s)$ with some other holomorphic function on an open subset of $\mathbb{C}$, namely for $\textup{Re}(s)$ large enough. By the Identity Theorem, it suffices to show equality when $s$ is large enough and real. This is true because that part of the real line has accumulation points (in fact every point is an accumulation point). Therefore, we will show the equalities below for $s \in \mathbb{R}$ large enough.
\end{remark}
\begin{proposition}\label{easy_part_dirichlet_inert}
We have 
\begin{equation*}
D_{(l)}(s) - D_{(\epsilon, l)}(s) = \frac{\langle\tilde{\phi}_1, \tilde{\psi}_1\rangle_{\mathcal{A}}}{\alpha_1-\alpha_2}\left(\frac{\alpha_1^3X^2}{Q_{p,F}^{(2)}(X_1)}-\frac{\alpha_2^{3}X^2}{Q_{p,F}^{(2)}(X_2)}\right).
\end{equation*}
\end{proposition}
\begin{proof}
This follows from equations \eqref{second_part}, \eqref{third_part} and Proposition \ref{prop: spinor primes} with $m=1$. We have (because we have a Hermitian inner product, we need to conjugate in the second argument)
\begin{align*}
    \sum_{m=0}^{\infty}\langle\tilde{\phi}_1, \tilde{\psi}_{p^{m}}\mid T_{+}(p^{m})\rangle_{\mathcal{A}}X_{1}^{m} = \sum_{m=0}^{\infty}\langle\tilde{\phi}_1, \tilde{\psi}_{p^{m}}\mid T_{+}(p^{m})X_{2}^{m}\rangle_{\mathcal{A}} &= \langle\tilde{\phi}_1, Q_{p,F}^{(2)}(X_2)^{-1}(1-p^{2k-6}X_2^{2})\tilde{\psi}_1\rangle_{\mathcal{A}} \\
    &=(1-p^{2k-6}X_1^2)Q_{p,F}^{(2)}(X_1)^{-1}\langle \tilde{\phi}_1, \tilde{\psi}_1\rangle_{\mathcal{A}},
\end{align*}
because $X_1, X_2$ are complex conjugates. Also,
\begin{equation*}
    \sum_{\epsilon=0}^{\infty}(\alpha_1p^{-(k+s-1)})^{2\epsilon} = \frac{1}{1 - \alpha_1^2p^{-2(k+s-1)}}.
\end{equation*}
So, the first part of the difference we are interested in is
\begin{equation*}
     \alpha_1\frac{\langle\tilde{\phi}_1, \tilde{\psi}_1\rangle_{\mathcal{A}}}{\alpha_1-\alpha_2}Q_{p,F}^{(2)}(X_1)^{-1}(1-p^{2k-6}X_1^2)\left(\frac{1}{1-\alpha_1^2p^{-2(k+s-1)}}-1\right) = \frac{\langle\tilde{\phi}_1, \tilde{\psi}_1\rangle_{\mathcal{A}}}{\alpha_1-\alpha_2}\frac{\alpha_1^3X^2}{Q_{p,F}^{(2)}(X_1)},
\end{equation*}
and similarly for the second part.
\end{proof}
\begin{proposition}\label{prop1}
Let $Y := p^2NX^2$ and $l \geq 0$. Then, for $i=1,2$, we have
\begin{align*}
    \sum_{m=0}^{\infty}\tilde{\psi}_{p^{m+2l}}\mid T_{+}(p^{m})\Lambda_{+}(p^{l})X_i^{m}(X^2N^{-1}p^{5})^{l} &= Q_{p,F}^{(2)}(X_i)^{-1}(\tilde{\psi}_{p^{2l}}-\tilde{\psi}_{p^{2l-1}}\mid T_{-}(p)X_i + p\tilde{\psi}_{p^{2l-2}}\mid \Lambda_{-}(p)X_i^2)\\&\mid \left((1+p(\nabla_p - p\Delta_p)X_i^2)\Delta_{p^{l}}^{-1}\Lambda_{+}(p^{l})(p^{-3}Y)^{l}\right).
\end{align*}
\begin{proof}
The proof follows from Proposition \ref{prop: spinor primes}, together with the fact that $F\mid  \Delta_{p^{l}} = (p^{2k-8})^{l}F$.
\end{proof}
\end{proposition}
In order to compute $D_{\epsilon}(s)$ of equation $\eqref{first_part}$, we will compute each of the summands above. We note that we need to interchange the $X_i$'s when we take them in/out of the inner products. 
\begin{proposition}\label{prop2}
We have
\begin{equation*}
    \alpha_1Q_{p,G}^{(2)}(X_1)^{-1}\sum_{l=0}^{\infty}\langle\tilde{\phi}_1, \tilde{\psi}_{p^{2l}}\mid_{k}(1+p(\nabla_p-p\Delta_p)X_2^2)\Delta_{p^{l}}^{-1}\Lambda_{+}(p^{l})(p^{-3}Y)^{l}\rangle_{\mathcal{A}} -
\end{equation*}
\begin{equation*}
    -\alpha_2 Q_{p,G}^{(2)}(X_2)^{-1}\sum_{l=0}^{\infty}\langle\tilde{\phi}_1, \tilde{\psi}_{p^{2l}}\mid_{k}(1+p(\nabla_p-p\Delta_p)X_1^2)\Delta_{p^{l}}^{-1}\Lambda_{+}(p^{l})(p^{-3}Y)^{l}\rangle_{\mathcal{A}}=
\end{equation*}
\begin{equation*}
    =\left(\alpha_1Q_{p,G}^{(2)}(X_1)^{-1}-\alpha_{2}Q_{p,G}^{(2)}(X_2)^{-1}\right)\frac{\langle\tilde{\phi}_1, \tilde{\psi}_{1}\mid _{k}S^{(2)}(Y)\rangle_{\mathcal{A}}}{D_{p,G}^{(2)}(Y)}+\left(\frac{\alpha_2^{3}X^2}{Q_{p,G}^{(2)}(X_2)} - \frac{\alpha_1^{3}X^2}{Q_{p,G}^{(2)}(X_1)}\right)\langle\tilde{\phi}_1, \tilde{\psi}_1\rangle_{\mathcal{A}}.
\end{equation*}.
\end{proposition}
\begin{proof}
We first observe
\begin{equation*}
    \tilde{\psi}_{p^{2l}}\mid _{k}(1 + p(\nabla_p - p\Delta_{p})X_2^2)
= \begin{cases}(1-p^{2k-6}X_2^2)\tilde{\psi}_1&\textup{if } l=0\\ \tilde{\psi}_{p^{2l}}& \textup{if }l\geq 1\end{cases},
\end{equation*}
by using the result of Proposition \ref{prop: spinor primes}. Hence, by Proposition \ref{rat2} we obtain
\begin{align*}
    \sum_{l=0}^{\infty}\tilde{\psi}_{p^{2l}}\mid_{k}(1+p(\nabla_p-p\Delta_p)X_2^2)\Delta_{p^{l}}^{-1}\Lambda_{+}(p^{l})(p^{-3}Y)^{l} &= (1-p^{2k-6}X_2^2)\tilde{\psi}_1 + \sum_{l=1}^{\infty}\tilde\psi_{p^{2l}}\mid_{k}\Delta_{p^{l}}^{-1}\Lambda_{+}(p^{l})(p^{-3}Y)^{l}\\
    &= \sum_{l=0}^{\infty}\tilde{\psi}_{p^{2l}}\mid_{k}\Delta_{p^{l}}^{-1}\Lambda_{+}(p^{l})(p^{-3}Y)^{l} - p^{2k-6}X_2^2\tilde{\psi}_1 \\
    &= \tilde{\psi}_{1}\mid_{k}S^{(2)}(Y)D_{p,G}^{(2)}(Y)^{-1} - \alpha_2^2X^2\tilde{\psi}_1.
\end{align*}
After taking the inner product with $\tilde{\phi}_1$ and keeping in mind the conjugation happening, we get the expression for the first term. Similarly for the other term and from this the result follows.
\end{proof}
\begin{proposition}\label{prop3}
We have
\begin{equation*}
    \alpha_1Q_{p,G}^{(2)}(X_1)^{-1}\sum_{l=0}^{\infty}p\langle\tilde{\phi}_1, \tilde{\psi}_{p^{2l-2}}\mid_{k}\Lambda_{-}(p)X_2^2(1+p(\nabla_p-p\Delta_p)X_2^2)\Delta_{p^{l}}^{-1}\Lambda_{+}(p^{l})(p^{-3}Y)^{l}\rangle_{\mathcal{A}} -
\end{equation*}
\begin{equation*}
    -\alpha_2 Q_{p,G}^{(2)}(X_2)^{-1}\sum_{l=0}^{\infty}p\langle\tilde{\phi}_1, \tilde{\psi}_{p^{2l-2}}\mid_{k}\Lambda_{-}(p)X_1^2(1+p(\nabla_p-p\Delta_p)X_1^2)\Delta_{p^{l}}^{-1}\Lambda_{+}(p^{l})(p^{-3}Y)^{l}\rangle_{\mathcal{A}}=
\end{equation*}
\begin{equation*}
    =\left(\frac{\alpha_1^3Np^4X^4}{Q_{p,G}^{(2)}(X_1)} - \frac{\alpha_2^3Np^4X^4}{Q_{p,G}^{(2)}(X_2)}\right)\frac{\langle\tilde{\phi}_1,\tilde{\psi}_{1}\mid_{k}S^{(2)}(Y)\rangle_{\mathcal{A}}}{D_{p,G}^{(2)}(Y)}.
\end{equation*}
\end{proposition}
\begin{proof}
We use the identities $\Lambda_{-}(p)(\nabla_p-p\Delta_p) = 0$ and $\Lambda_{-}(p)\Lambda_{+}(p) = p^{6}(\Delta_{p})^2$, from the proof of Proposition \ref{prop: rankin_prime}. We have, for $i=1, 2$:
\begin{equation*}
    \sum_{l=0}^{\infty}p\tilde{\psi}_{p^{2l-2}}\mid_{k}\Lambda_{-}(p)X_i^2(1+p(\nabla_p-p\Delta_p)X_i^2)\Delta_{p^{l}}^{-1}\Lambda_{+}(p^{l})(p^{-3}Y)^{l}=
\end{equation*}
\begin{equation*}
    =p\sum_{l=1}^{\infty}\tilde{\psi}_{p^{2l-2}}\mid_{k}\Lambda_{-}(p)X_i^2\Delta_{p^{l}}^{-1}\Lambda_{+}(p^{l})(p^{-3}Y)^{l}
    =p\sum_{l=1}^{\infty}\tilde{\psi}_{p^{2l-2}}\mid_{k}p^{6}(\Delta_{p})^2\Delta_{p^{l}}^{-1}\Lambda_{+}(p^{l-1})(p^{-3}Y)^{l}X_i^2=
\end{equation*}
\begin{equation*}
    =p^4\sum_{l=1}^{\infty}\tilde{\psi}_{p^{2l-2}}\mid\Delta_{p^{l-1}}^{-1}\Lambda_{+}(p^{l-1})(p^{-3}Y)^{l-1}\Delta_pX_i^2Y
    =p^{2k-4}D_{p,G}^{(2)}(Y)^{-1}\tilde{\psi}_{1}\mid S^{(2)}(Y)X_i^2Y,
\end{equation*}
from which the result then follows. The last equality follows from Proposition \ref{rat2}. We also used equation \eqref{lambda_inert} as well as the facts that $\Delta_p^{l} = \Delta_{p^{l}}$ and that $\Delta_{p}$ and $\Lambda_{+}(p)$ commute, because of the fact that $\Delta_p$ has a single right coset representative, which is $p 1_{4}$.
\end{proof}
\begin{proposition}\label{prop4}
Let $\lambda_p$ denote the eigenvalue given by $\tilde{\psi}_{p} \textup{ }|\textup{ } T_{+}(p) = \lambda_p\tilde{\psi}_1$. We then have
\begin{equation*}
    \alpha_1Q_{p,G}^{(2)}(X_1)^{-1}\sum_{l=0}^{\infty}\langle\tilde{\phi}_1, \tilde{\psi}_{p^{2l-1}}\mid_{k}T_{-}(p)X_2(1+p(\nabla_p-p\Delta_p)X_2^2)\Delta_{p^{l}}^{-1}\Lambda_{+}(p^{l})(p^{-3}Y)^{l}\rangle_{\mathcal{A}} -
\end{equation*}
\begin{equation*}
    -\alpha_2 Q_{p,G}^{(2)}(X_2)^{-1}\sum_{l=0}^{\infty}\langle\tilde{\phi}_1, \tilde{\psi}_{p^{2l-1}}\mid_{k}T_{-}(p)X_1(1+p(\nabla_p-p\Delta_p)X_1^2)\Delta_{p^{l}}^{-1}\Lambda_{+}(p^{l})(p^{-3}Y)^{l}\rangle_{\mathcal{A}}=
\end{equation*}
\begin{equation*}
    = \left(\frac{\alpha_1^2p^4X^3\lambda_{p}}{Q_{p,G}^{(2)}(X_1)(1+Y)} - \frac{\alpha_2^2p^4X^3\lambda_{p}}{Q_{p,G}^{(2)}(X_2)(1+Y)}\right)\frac{\langle\tilde{\phi}_1,\tilde{\psi}_1\rangle S_{F}^{(2)}(Y)^{\textup{prim}}}{(\alpha_1-\alpha_2)D_{p,G}^{(2)}(Y)}.
\end{equation*}
\end{proposition}
\begin{proof}
The proof is exactly the same as the proof in \cite[Proposition $4.5$]{heim} (using also the fact that $\phi \mid_{k} S^{(2)}(X)T_{+}(p) = \phi \mid_{k} T_{+}(p)S^{(2)}(X)^{\textup{factor}}$, if $\phi \in J_{k,p}^2$ as in Lemma \ref{index_p}.)
\end{proof}

Now, Proposition \ref{easy_part_dirichlet_inert} gives us a way to compute $D_{(l)}(s)-D_{(\epsilon,l)}(s)$. Propositions \ref{prop1}, \textup{ }\ref{prop2},\textup{ }\ref{prop3},\textup{ }\ref{prop4} can be used to compute $D_{(l)}(s)$. Hence, from equation \eqref{dirichlet_inert_dec}, we obtain
\begin{equation*}
    D_{F,G,h}^{(p)}(s) = \frac{\langle\tilde{\phi}_1,\tilde{\psi}_1\rangle S_{F}^{(2)}(Y)^{\textup{prim}}}{(\alpha_1-\alpha_2)D_{p,G}^{(2)}(Y)}\times
\end{equation*}
\begin{equation*}
    \times\left(\frac{\alpha_1}{Q_{p,G}^{(2)}(X_1)}-\frac{\alpha_2}{Q_{p,G}^{(2)}(X_2)} + \frac{\alpha_1^3Np^4X^4}{Q_{p,G}^{(2)}(X_1)}-\frac{\alpha_2^3Np^4X^4}{Q_{p,G}^{(2)}(X_2)} - \frac{\alpha_1^2p^4X^3\lambda_p}{Q_{p,G}^{(2)}(X_1)(1+Y)} + \frac{\alpha_2^2p^4X^3\lambda_p}{Q_{p,G}^{(2)}(X_2)(1+Y)}\right).
\end{equation*}
Let us now look at the expression in the big bracket. The numerator equals
\begin{equation*}
    ((\alpha_1+\alpha_1^3Np^4X^4)(1+Y) - \alpha_1^2p^4X^3\lambda_p)Q_{p,G}^{(2)}(X_2) - ((\alpha_2+\alpha_2^3Np^4X^4)(1+Y) - \alpha_2^2p^4X^3\lambda_p)Q_{p,G}^{(2)}(X_1).
\end{equation*}
Here 
\begin{equation*}
    Q_{p,G}^{(2)}(t) = 1 - \lambda_pt + (p\lambda_{T_{1,p}}+p(p^3+p^2-p+1)p^{2k-8})t^2 -p^4p^{2k-8}\lambda_pt^3+p^{4k-8}t^4,
\end{equation*}
where $\lambda_{T_{1,p}}$ is the eigenvalue corresponding to the operator $T_{1,p}$. Let then
\begin{equation*}
    A_2 := p\lambda_{T_{1,p}}+p(p^3+p^2-p+1)p^{2k-8}.    
\end{equation*}
By then performing the very lengthy calculation, and grouping in powers of $Y$, we obtain that the above numerator equals
\begin{equation*}
    (\alpha_1-\alpha_2)(1-Y)(1 - Y(A_2p^2N^{-2}-2)+Y^2(p^2N^{-2}\lambda_p^2-2A_2p^2N^{-2}+2) - Y^3(A_2p^2N^{-2}-2)+Y^4)
\end{equation*}
\begin{equation*}
    = (\alpha_1-\alpha_2)(1-Y)D_{p,G}^{(2)}(Y),
\end{equation*}
using Proposition \ref{rankin}. Hence, we obtain
\begin{equation}\label{final_expression_dirichlet_inert}
    D_{F,G,h}^{(p)}(s) = \frac{\langle\tilde{\phi}_1,\tilde{\psi}_1\rangle_{\mathcal{A}} S_{F}^{(2)}(Y)^{\textup{factor}}(1-Y)}{Q_{p,G}^{(2)}(X_1)Q_{p,G}^{(2)}(X_2)}.
\end{equation}
Let us now explore the connection of $S_{F}^{(2)}(Y)^{\textup{factor}}$ with known $L$-functions.
\begin{proposition}\label{s_factor}
We have
\begin{equation*}
    S_{F}^{(2)}(Y)^{\textup{factor}} = L_{p}(f, k+s-2)L_{p}\left(f, k+s-2, \left(\frac{-4}{p}\right)\right),
\end{equation*}
    where $f \in S_{k-1}\left(\Gamma_{0}(4), \left(\frac{-4}{\cdot}\right)\right)$ is the modular form whose Maass lift is $F$, as in Proposition \ref{maass_lift}.
\end{proposition}
\begin{proof}
Let $T(a,b):= \Gamma_0(4)\textup{diag}(a,b)\Gamma_0(4)$ for $a,b\geq 1$ such that $a\mid b$. Then, write
\begin{equation*}
    f \mid_{k-1} T(p) = a(p)f,
\end{equation*}
for the standard operator $T(p) := \Gamma_{0}(4)\textup{diag}(1,p)\Gamma_{0}(4)$, with $a(p) \in \mathbb{C}$. This is the same operator as $T(1,p)$. Here, the $\mid_{k-1}$ action is the usual $\textup{GL}_2$-action. By standard relations between Hecke operators, we then have
\begin{equation*}
    T(p^2) = T(p)^2 - \left(\frac{-4}{p}\right)p^{k-2},
\end{equation*}
where $T(p^2) := T(1,p^2) + \left(\frac{-4}{p}\right)p^{k-3}$. This then implies
\begin{equation*}
    f \mid_{k-1} T(1,p^2) = (a(p)^2 + p^{k-2}+p^{k-3})f.
\end{equation*}
Using now \cite[Lemma 3.3]{gritsenko_maass} we obtain that 
\begin{equation*}
    \tilde{\phi}_1 \mid_k T^{J}(p) = p^{k-4}(a(p)^2 + p^{k-2}+p^{k-3})\tilde{\phi}_1.
\end{equation*}
Hence
\begin{equation*}
    S_{F}^{(2)}(Y)^{\textup{factor}} = 1 - p^{1-k}(a(p)^2+2p^{k-2})Y + p^{-2}Y^2.
\end{equation*}
We now define the Satake parameters of $f$ as follows:
\begin{equation*}
    1 - a(p)p^{-s}+\left(\frac{-4}{p}\right)p^{k-2}p^{-2s} = (1 - \alpha_p p^{-s})\left(1 - \beta_p \left(\frac{-4}{p}\right)p^{-s}\right).
\end{equation*}
But, $Y = p^2NX^2$, $N = p^{k-1}$ and $X = p^{-(k+s-1)}$, so we obtain
\begin{align*}
    S_{F}^{(2)}(Y)^{\textup{factor}} &= 1 - p^{4-2k}(\alpha_p^2 + \beta_p^2)p^{-2s} + p^{-2k+4-2s} \\
    &= (1-p^{4-2k-2s}\alpha_p^{2})(1 - p^{4-2k-2s}\beta_p^2)\\
    &= (1 - \alpha_pp^{2-k-s})\left(1 - \beta_p\left(\frac{-4}{p}\right)p^{2-k-s}\right)\left(1 - \alpha_p\left(\frac{-4}{p}\right)p^{2-k-s}\right)(1 - \beta_p p^{2-k-s})\\
    &=L_p(f, k+s-2)L_p\left(f, k+s-2, \left(\frac{-4}{p}\right)\right).
\end{align*}
\end{proof}
Hence, from equation \eqref{final_expression_dirichlet_inert} and Proposition \ref{s_factor}, we obtain the following Theorem:
\begin{theorem}\label{Main Theorem, inert case}
Let $F,G \in S_{2}^{k}$ and $h \in S_{1}^{k}$ be Hecke eigenforms, all having totally real Fourier coefficients, $h$ normalised, and $F$ belonging in the Maass space, with corresponding $f \in S_{k-1}\left(\Gamma_{0}(4), \left(\frac{-4}{\cdot}\right)\right)$. Let also $\phi_1, \psi_1$ denote the first Fourier-Jacobi coefficients of $F,G$, $X_i=\alpha_ip^{-(2k+s-4)}, \textup{ }i=1,2$, where $\alpha_i$ denote the Satake parameters of $h$ and $Y=p^{-k-2s+2}$. We then have, for $\textup{Re}(s)$ large enough
\begin{equation*}
    D_{F,G,h}^{(p)}(s) = \frac{\langle\tilde{\phi}_1,\tilde{\psi}_1\rangle_{\mathcal{A}} L_{p}(f, k+s-2)L_{p}\left(f, k+s-2, \left(\frac{-4}{p}\right)\right)(1-Y)}{Q_{p,G}^{(2)}(X_1)Q_{p,G}^{(2)}(X_2)}.
\end{equation*}
Here, $D_{F,G,h}^{(p)}$ denotes the $p$-part of the Dirichlet series, as in equation \eqref{p-part dirichlet} and $Q_{p,G}^{(2)}$ denotes the $p$-factor of Gritsenko's $L$-function, as in Definition \ref{gritsenko l-function}.
\end{theorem}
\section{Split Primes}\label{split primes}
We will now consider the case where the odd rational prime $p$ splits. That is, we have that $p = \pi \overline{\pi}$ for some prime element $\pi \in \mathcal{O}_K$. Our aim in this Section is to prove weak rationality theorems analogous to Propositions \ref{prop: spinor primes} and \ref{rat2}. In order to do that, we will first have to factorise the polynomials which serve as the $p$-factors of the standard and Gritsenko's $L$-function in the parabolic Hecke ring $H_p^{1,1}$, as defined in Definitions \ref{standard l-function} and \ref{gritsenko l-function}. The factorisation of the latter polynomial has been done by Gritsenko in \cite[Proposition 3.2]{gritsenko}. Our aim, therefore, is to factorise the standard Hecke polynomial. However, as we mentioned in Section \ref{hecke algebras}, $H^{1,1}_{p}$ is isomorphic to the ring of polynomials of one variable with coefficients from the Hecke ring of the parabolic subgroup
\begin{equation*}
P_{1,2,1}(\mathbb{Z}_p) = \left\{\m{g_1 & *& *\\0&g& *\\0&0&g_2} \in \textup{GL}_{4}(\mathbb{Z}_p) \mid g_1,g_2 \in \mathbb{Z}_{p}^{\times}, \textup{ }g \in \textup{GL}_{2}(\mathbb{Z}_p)\right\}.
\end{equation*}
Hence, we will first investigate Hecke rings of the general linear group and then use this isomorphism to translate the relations back to $H_p^{1,1}$.
\subsection{Hecke rings in $\textup{GL}_4$ and factorisation}\label{gl4}
For this section, we denote by $\Gamma_{n} := \textup{GL}_{n}(\mathbb{Z}_p)$ and let 
\begin{equation*}
    P_{m_1,\cdots,m_l}^{(n)} := \left\{\m{A_1&*&\cdots&*\\0&A_2&\cdots&*\\\vdots&\vdots&\vdots&\vdots\\0&0&\cdots&A_{l}}, \textup{ }A_i \in \textup{GL}_{m_i}(\mathbb{Q}_p),\textup{ } m_1+\cdots+m_l = n\right\}
\end{equation*}
be a parabolic subgroup of $\textup{GL}_{n}(\mathbb{Q}_p)$. Denote by $\Gamma_{P} = \Gamma_{P}^{(n)} = \Gamma_{m_1,\cdots,m_l}^{(n)}$, the group of $\mathbb{Z}_p$-points of $P_{m_1,\cdots,m_l}^{(n)}$. Let also $H_{n} := H(\Gamma_n, \textup{GL}_{n}(\mathbb{Q}_p))$ be the full Hecke ring in this case and $H_{m_1,\cdots,m_l} := H(\Gamma_{m_1,\cdots,m_l}^{(n)},  P_{m_1,\cdots,m_l}^{(n)})$ denote the corresponding parabolic Hecke ring.
\\\\
Let us now explicitly describe the isomorphism $H_{p}^2 \cong H(\textup{GL}_4(\mathbb{Q}_p),\textup{GL}_4(\mathbb{Z}_p)[x^{\pm}]$, which will yield $H_{p}^{1,1} \cong H_{1,2,1}[x^{\pm}]$, as Gritsenko does in \cite[Proposition $2.4$]{gritsenko}.\\

We fix an identification $K_p := K \otimes \mathbb{Q}_p \cong \mathbb{Q}_{p} \times \mathbb{Q}_p$ and denote by $(\mu, -\mu)$ the image of the element $(2i)^{-1}$. Let also $e:=(1,0), \textup{ }e^{\sigma} := (0,1) \in K_{p}$. We perform a change of variables with the matrix $C$ (i.e. $g \longmapsto C^{-1}gC$), where 
\[
C :=  \begin{pmatrix} eI_2 & - \mu e^{\sigma} I_2 \\ \mu e^{\sigma} 1_2 & e 1_2 \end{pmatrix} = \left( \begin{pmatrix} I_2 & 0_2 \\ 0_2 &  1_2 \end{pmatrix}, \begin{pmatrix} 0_2 & - \mu  I_2 \\ \mu  1_2 & 0_2\end{pmatrix} \right) \in \textup{GL}_{4}(\mathbb{Q}_p) \times \textup{GL}_{4}(\mathbb{Q}_p),
\]
where $I_2 := \m{0&1\\1&0}$. We remark that for $n=2$:
\[
\begin{pmatrix} I_2 & 0 \\ 0 &  1_2 \end{pmatrix} \textup{diag}(a_1,a_2,a_3,a_4) \begin{pmatrix} I_2 & 0 \\ 0 &  1_2 \end{pmatrix}^{-1} = \textup{diag}(a_2, a_1, a_3, a_4).
\]
We then have that $G_p^{(2)}$ (see Section \ref{hecke algebras}) is identified by 
\begin{equation*}
    \widetilde{G}_p^{(2)} = \{(X,Y) \in \textup{GL}_4(\mathbb{Q}_p) \times \textup{GL}_4(\mathbb{Q}_p) \mid Y^{t} \cdot X = c1_4, \textup{ }c \in \mathbb{Q}_p^{\times}\},
\end{equation*}
and $U_p^{(2)}$ by
\begin{equation*}
    \widetilde{U}_p^{(2)} = \{(\gamma, \alpha(\gamma^{-1})^t) \mid  \gamma \in \textup{GL}_4(\mathbb{Z}_p), \textup{ }\alpha \in \mathbb{Z}_p^{\times}\}.
\end{equation*}
The identification $G_p^{(2)} \cong \textup{GL}_{4}(\mathbb{Q}_p) \times \mathbb{Q}_p^{\times}$ is obtained by projection to the first component induced by $K_p = \mathbb{Q}_p \times \mathbb{Q}_p$.\\\\
The double coset of $(M, c(M^{-1})^{t}) \in \widetilde{G}_p^{(2)}$ with respect to $\widetilde{U}_p^{(2)}$ is determined by the double coset of $M$ with respect to $\textup{GL}_4(\mathbb{Z}_p)$ and by the order $\delta$ of the ideal $c\mathbb{Z}_p$. We will denote such a coset by $(M, \delta)_{\widetilde{U}_p^{(2)}}$. 
Note that here there is a choice, namely whether we have $\pi \longmapsto (pu,v)$ with $u,v \in \mathbb{Z}_p^{\times}$ or $\overline{\pi} \longmapsto (pu', v')$ with $u',v' \in \mathbb{Z}_p^{\times}$. In the following, we always choose the first identification.\\\\
The above identification makes the following diagram commutative:
\[\begin{tikzcd}
H_p^{2} \arrow{r} \arrow[swap]{d} & H_4[x^{\pm}] \arrow{d} \\
H_{p}^{1,1} \arrow{r} & H_{1,2,1}[x^{\pm}]
\end{tikzcd}
\]
We now want to obtain a Lemma regarding the decomposition of an element in $H_{1,2,1}$ into right cosets. Let $n \geq 1$. For a given square matrix $R \in M_{n}(\mathbb{Z}_p)$, we define $\Gamma_{n}^{R} := \Gamma_{n} \cap R^{-1}\Gamma_{n}R$. Also, if $A,D$ are square matrices of sizes $n_1, n_2$ respectively, we define
\begin{equation*}
     V(A,D) := \{ AY \,\,|\,\, Y \in M_{n_1,n_2}(\mathbb{Z}_p) \pmod{D^{\times}} \},
\end{equation*}
that is $AY_1 \equiv AY_2 \pmod{D^{\times}}$ if and only if  $AY_1D^{-1}-AY_2D^{-1} \in M_{n_1,n_2}(\mathbb{Z}_p)$.\\\\
A straightforward generalisation of \cite[Lemma 2]{gritsenko4} gives the following Lemma.
\begin{lemma}\label{cor1}
Let $a,b \in \mathbb{Q}_p^{\times}$ and $A \in \textup{GL}_{2}(\mathbb{Q}_p)$. We then have 
\begin{equation*}
\Gamma_{1,2,1}\begin{pmatrix}a & * & * \\ 0 & A & * \\ 0& 0 & b\end{pmatrix}\Gamma_{1,2,1} = \sum \Gamma_{1,2,1}\begin{pmatrix}a & B & D \\ 0 & A & C \\ 0& 0 & b\end{pmatrix}\begin{pmatrix}1 & 0 & 0 \\ 0 & N & 0 \\ 0& 0 & 1\end{pmatrix},
\end{equation*}
where $N \in \Gamma_{2}^{A} \backslash \Gamma_{2}$ and $B \in V(a, A), \textup{ }D \in V(a,b), \textup{ }C \in V(A,b)$.
\end{lemma}
We now denote by $T(a,b,c,d)$ the element of the Hecke ring $H(\textup{GL}_{4}(\mathbb{Q}_p), \textup{GL}_{4}(\mathbb{Z}_p))$ defined by
\begin{equation*}
    T(a,b,c,d) := \textup{GL}_{4}(\mathbb{Z}_p)\textup{diag}(a,b,c,d)\textup{GL}_{4}(\mathbb{Z}_p).
\end{equation*}
We then have the standard elements of $H_{4}$:
\begin{equation}\label{standard_hecke_operators_gl4}
    T_1:= T(1,1,1,p), \textup{ } T_2:= T(1,1,p,p), \textup{ }T_3:= T(1,p,p,p), \textup{ }\Delta:= T(p,p,p,p).
\end{equation}
The decomposition of these elements into right cosets can be found in \cite[Lemma $3.2.18$]{andrianov}.\\

We now note that the conditions of Lemma \ref{lemma:embeddings} hold for the Hecke rings $H_4, H_{1,2,1}$, as explained in \cite[page 2870]{gritsenko3}. Therefore, there is an embedding $H_4 \xhookrightarrow{}H_{1,2,1}$. Our aim is to compute the images of the standard Hecke operators of equation \eqref{standard_hecke_operators_gl4} under this embedding.  In a similar fashion to $H_4$, we write
\begin{equation*}
    T_0(a,b,c,d) := \Gamma_{1,2,1}\textup{diag}(a,b,c,d)\Gamma_{1,2,1},
\end{equation*}
for an element of $H_{1,2,1}$. Let us now introduce some useful elements of $H_{1,2,1}$.
\begin{itemize}
    \item $\Lambda_{+}^{1,3} := T_0(1, p, p, p)$.
    \item $\Lambda_{+}^{3,1} := T_0(1,1,1,p)$.
    \item $\Lambda_{-}^{1,3} := T_0(p,1,1,1)$.
    \item $\Lambda_{-}^{3,1} := T_0(p,p,p,1)$.
    \item $T_{-}(p) := T_0(p,1,p,1)$.
    \item $T_{+}(p) := T_0(1,1,p,p)$.
    \item $\Delta := T_0(p,p,p,p)$.
\end{itemize}
The right coset decompositions of these elements can be now computed by Lemma \ref{cor1}. We again note that we use the same symbol $\Delta$ in both $H_4$ and $H_{1,2,1}$. We can do that, as the embedding of Lemma \ref{lemma:embeddings} does not change this specific element.
\begin{proposition}\label{embedding}
Let $\epsilon$ denote the embedding of the Hecke ring $H_{4}$ into $H_{1,2,1}$, as described in Lemma \ref{lemma:embeddings}. We then have the following images of the elements $T_i$:
\begin{itemize}
    \item $\epsilon(T_1) = \Lambda_{-}^{1,3}+T_0(1,1,p,1)+\Lambda_{+}^{3,1}$.
    \item $\epsilon(T_2) = T_{-}(p)+T_{+}(p)+T_0(1,p,p,1)+T_0(p,1,1,p)$.
    \item $\epsilon(T_3) = \Lambda_{-}^{3,1}+T_0(p,1,p,p)+\Lambda_{+}^{1,3}$.
\end{itemize}
\end{proposition}
\begin{proof}
These follow directly by the right coset decompositions and the definition of the embedding $\epsilon$. We note here a typo in Gritsenko's paper \cite[page 2879]{gritsenko3} in the $\Lambda_{+}$ component (it appears we need to swap $\Lambda_{+}^{1,3}$ and $\Lambda_{+}^{3,1}$).
\end{proof}

We are now in a position to give the factorisation of the standard Hecke polynomial $Q_4$, as this is defined in \cite[Example $2$]{gritsenko3}.
\begin{theorem}\label{theor: fact}
Let $Q_4(t) := 1 - T_1t+pT_2t^2-p^3T_3t^3+p^6\Delta t^4 \in H^{4}[t]$. Then, in $H_{1,2,1}[t]$, we have the factorisation
\begin{equation*}
    Q_{4}(t) = (1-\Lambda_{-}^{1,3}t)(1-A_1t+pA_2t^2)(1-\Lambda_{+}^{3,1}t),
\end{equation*}
where $A_1 := T_0(1,1,p,1)$ and $A_2 := T_0(1,p,p,1)$.
\end{theorem}
\begin{proof}
The factorisation follows by the images of the elements in Proposition \ref{embedding} as well as a series of relations, which follow from the above right coset decompositions:
\begin{itemize}
    \item $\Lambda_{-}^{1,3}A_1 = pT_{-}(p)$.
    \item $A_1\Lambda_{+}^{3,1} = pT_{+}(p)$.
    \item $\Lambda_{-}^{1,3}\Lambda_{+}^{3,1} = pT_0(p,1,1,p)$.
    \item $A_2\Lambda_{+}^{3,1} = p^2\Lambda_{+}^{1,3}$.
    \item $\Lambda_{-}^{1,3}A_2 = p^2\Lambda_{-}^{3,1}$.
    \item $\Lambda_{-}^{1,3}A_1\Lambda_{+}^{3,1} = p^3T_0(p,1,p,p)$.
    \item $\Lambda_{-}^{1,3}A_2\Lambda_{+}^{3,1} = p^5\Delta$.
\end{itemize}
\end{proof}
\subsection{Hecke Operators and weak rationality theorems}\label{operators_split}
We will now translate the results above back to the Hecke rings $H_p^2$ and $H_p^{1,1}$ of the unitary group. We have the following correspondence between the standard elements of $H_4$ (see equation \eqref{standard_hecke_operators_gl4}) and of $H_{p}^2$:
\begin{itemize}
    \item $T_1 \longleftrightarrow T_{\overline{\pi}} := \Gamma_2\textup{diag}(1,\overline{\pi},p,\overline{\pi})\Gamma_2$.
    \item $T_2 \longleftrightarrow T_p := \Gamma_2\textup{diag}(1,1,p,p)\Gamma_2$.
    \item $T_3 \longleftrightarrow T_{\pi} := \Gamma_2\textup{diag}(1,\pi,p,\pi)\Gamma_2$.
\end{itemize}
Also, for the correspondence between the Hecke operators of $H_{p}^{1,1}$ and $H_{1,2,1}$, we have:
\begin{itemize}
    \item $\Lambda_{+}^{1,3} \longleftrightarrow \Lambda_{+}(\pi):= \Gamma_{1,1}\textup{diag}(\pi, 1, \pi, p)\Gamma_{1,1}$.
    \item $\Lambda_{+}^{3,1} \longleftrightarrow \Lambda_{+}(\overline{\pi}) := \Gamma_{1,1}\textup{diag}(\overline{\pi}, 1, \overline{\pi}, p)\Gamma_{1,1}$.
    \item $\Lambda_{-}^{3,1} \longleftrightarrow \Lambda_{-}(\pi) := \Gamma_{1,1}\textup{diag}(\pi, p, \pi, 1)\Gamma_{1,1}$.
    \item $\Lambda_{-}^{1,3} \longleftrightarrow \Lambda_{-}(\overline{\pi}) := \Gamma_{1,1}\textup{diag}(\overline{\pi}, p, \overline{\pi}, 1)\Gamma_{1,1}$.
    \item $T_0(1,1,p,1) \longleftrightarrow T(\overline{\pi}):= \Gamma_{1,1}\textup{diag}(1, \overline{\pi},p,\overline{\pi})\Gamma_{1,1}$.
    \item $T_0(p,1,p,p) \longleftrightarrow T(\pi):= \Gamma_{1,1}\textup{diag}(1, \pi,p,\pi)\Gamma_{1,1}$.
    \item $T_0(1,p,p,1) \longleftrightarrow T(\pi, \overline{\pi}):= \Gamma_{1,1}\textup{diag}(\pi, \overline{\pi},\pi,\overline{\pi})\Gamma_{1,1}$.
    \item $T_0(p,1,1,p) \longleftrightarrow T(\overline{\pi}, \pi):= \Gamma_{1,1}\textup{diag}(\overline{\pi}, \pi,\overline{\pi},\pi)\Gamma_{1,1}$.
    \item $T_{-}(p) \longleftrightarrow T_{-}(p) := \Gamma_{1,1}\textup{diag}(1,p,p,1)\Gamma_{1,1}$.
    \item $T_{+}(p) \longleftrightarrow T_{+}(p) := \Gamma_{1,1}\textup{diag}(1,1,p,p)\Gamma_{1,1}$.
\end{itemize}
We denote by $\Delta_{\pi} := \Gamma_{1,1}\textup{diag}(\pi, \pi, \pi, \pi)\Gamma_{1,1}$ and similarly for $\Delta_{\overline{\pi}}$ and $\Delta_p$. We again use the same notation for these as elements of $H_{p}^{2}$ as well. Finally, we also have the operator $\nabla_{p}$ as in Subsection \ref{hecke inert}.\\

In order to make clear how the isomorphism described in the beginning of Subsection \ref{gl4} works, let us describe it in the case $\Lambda_{+}(\pi)$. We have by sending $\pi \longmapsto p$
\begin{equation*}
    \Gamma_{1,1}\textup{diag}(\pi, 1, \pi, p)\Gamma_{1,1} \longmapsto \Gamma_{1,2,1}\textup{diag}(p, 1, p, p)\Gamma_{1,2,1} \longmapsto \Gamma_{1,2,1}\textup{diag}(1, p, p, p)\Gamma_{1,2,1},
\end{equation*}
where the second arrow is simply the swap of the first two diagonal elements induced by the matrix $C$, as described in Subsection \ref{gl4}. Also, since $\mu(\textup{diag}(\pi, 1, \pi, p)) = p$, $\Lambda_{+}(\pi)$ gets mapped to $(\Lambda_{+}^{1,3}, 1)_{\tilde{U}_p^{(2)}}$, but in general we will not keep account of the second coordinate. The only case in which this plays a difference is in the identification of $\Delta_{\pi}$ and $\Delta_p$, which both get mapped to $\textup{diag}(p,p,p,p)$ but their factors of similitude are $1, 2$ respectively. The reason why factors of $\Delta_{\overline{\pi}}$ appear in the relations below is to compensate for the second coordinate, as $\textup{diag}(\overline{\pi},\overline{\pi}, \overline{\pi}, \overline{\pi}) \longmapsto \textup{diag}(1,1,1,1)$.
\\\\
The table below shows some relations between the above Hecke operators. These can be obtained by translating back to $H_{1,2,1}$ and using the right coset decompositions. The way to read the table is that we first read an operator $X$ in the first row, then an operator $Y$ in the first column and the result is $XY$. We write "comm" to mean that the operators commute.
\begin{center}
\begin{table}[H]
\scalebox{0.8}{
\begin{tabular}{ |c| c| c| c|c|c|c|c|c|c|c|}
\hline
  & $\Lambda_{+}(\pi)$ & $\Lambda_{+}(\overline{\pi})$ & $\Lambda_{-}(\pi)$ & $\Lambda_{-}(\overline{\pi})$ & $T(\overline{\pi})$ & $T(\pi)$ &$T(\pi,\overline{\pi})$&$T(\overline{\pi},\pi)$ &$T_{-}(p)$&$T_{+}(p)$\\\hline
 $\Lambda_{+}(\pi)$ &  & \textup{comm} & $p\Delta_{\pi}T(\pi, \overline{\pi})$&$p^3\Delta_p$&\textup{comm}&$p\Delta_{\pi}T_{+}(p)$ &\textup{comm}&$p^2\Delta_{\pi}\Lambda_{+}(\overline{\pi})$ & $p^2\Delta_{\pi}T(\overline{\pi})$&\textup{comm}\\\hline  
 $\Lambda_{+}(\overline{\pi})$ & \textup{comm} & & $p^3\Delta_p$&$p\Delta_{\overline{\pi}}T(\overline{\pi}, \pi)$&$p\Delta_{\overline{\pi}}T_{+}(p)$ & \textup{comm}&$p^2\Delta_{\overline{\pi}}\Lambda_{+}(\pi)$ & $\textup{comm}$&$p^2\Delta_{\overline{\pi}}T(\pi)$&\textup{comm}\\  \hline
 $\Lambda_{-}(\pi)$&&&&&&&&&&\\\hline
 $\Lambda_{-}(\overline{\pi})$&&&&&&&&&& \\\hline
 $T(\overline{\pi})$&&&\textup{comm}&$p\Delta_{\overline{\pi}}T_{-}(p)$&&&\textup{comm}&&& \\\hline
 $T(\pi)$&&&$p\Delta_{\pi}T_{-}(p)$&\textup{comm}&&&&\textup{comm}&& \\\hline
 $T(\pi,\overline{\pi})$&&&\textup{comm}&$p^2\Delta_{\overline{\pi}}\Lambda_{-}(\pi)$&&&&&&\\\hline
 $T(\overline{\pi},\pi)$ &&& $p^2\Delta_{\pi}\Lambda_{-}(\overline{\pi})$&\textup{comm}&&&&&&\\\hline
 $T_{-}(p)$ &&&&&&&&&&\\\hline
 $T_{+}(p)$&&&$p^2\Delta_{\pi}T(\overline{\pi})$&$p^2\Delta_{\overline{\pi}}T(\pi)$&&&&&&\\\hline
\end{tabular}
}
\caption{Relations of Hecke Operators for split primes.}
\label{table:1}
\end{table}
\end{center}
\begin{proposition}\label{s polynomials}
Let
\begin{equation*}
    D_{\pi}^{(2)}(t) := 1 - T_{\overline{\pi}}t + p\Delta_{\overline{\pi}}T_pt^2 - p^3\Delta_{\overline{\pi}}^2T_{\pi}t^3+p^6\Delta_{\overline{\pi}}^3\Delta_{\pi} t^4 \in H_p^{2}[t],
\end{equation*}
and 
\begin{equation*}
    D_{\overline{\pi}}^{(2)}(t) := 1 - T_{\pi}t + p\Delta_{\pi} T_pt^2 - p^3\Delta_{\pi}^2 T_{\overline{\pi}}t^3 + p^6 \Delta_{\pi}^{3}\Delta_{\overline{\pi}}t^4 \in H_p^{2}[t].
\end{equation*}
Let also
\begin{equation*}
    S_{\pi}(t) := 1 - T(\overline{\pi})t + p\Delta_{\overline{\pi}}T(\pi, \overline{\pi})t^2 \in H_p^{1,1}[t],
\end{equation*}
and 
\begin{equation*}
    S_{\overline{\pi}}(t) := 1 - T(\pi)t + p\Delta_{\pi} T(\overline{\pi}, \pi)t^2 \in H_p^{1,1}[t].
\end{equation*}
We then have the following factorisations
\begin{equation*}
    D_{\pi}^{(2)}(t) = (1 - \Lambda_{-}(\overline{\pi})t)S_{\pi}(t)(1 - \Lambda_{+}(\overline{\pi})t),
\end{equation*}
\begin{equation*}
    D_{\overline{\pi}}^{(2)}(t) = (1 - \Lambda_{-}(\pi)t)S_{\overline{\pi}}(t)(1-\Lambda_{+}(\pi)t).
\end{equation*}
\end{proposition}
\begin{proof}
    This follows from Theorem \ref{theor: fact} after pulling back to the parabolic Hecke ring $H_{p}^{1,1}$ of the unitary group.
\end{proof}
\begin{remark}
    We remark here that $Z_{\pi}^{(2)}(t) = D_{\pi}^{(2)}(\Delta_{\overline{\pi}}^{-1}t)$ and $Z_{\overline{\pi}}^{(2)}(t) = D_{\overline{\pi}}^{(2)}(\Delta_{\pi}^{-1}t)$, where $Z_{\pi}^{(2)}, Z_{\overline{\pi}}^{(2)}$ are the standard polynomials defined in Section \ref{hecke algebras}. This can be seen by computing the images under the Satake mapping of the above coefficients, as can be found in \cite[Lemma 3.7]{gritsenko}.
\end{remark}
We will need a Lemma regarding the decomposition of a Hecke operator in $H_{p}^{1,1}$ into right cosets. Let $\iota : \Gamma_1 \longrightarrow \Gamma_{1,1}$ denote the embedding
\begin{equation*}
    \m{A &B\\C&D} \longmapsto \m{A&0&B&0\\0&1&0&0\\C&0&D&0\\0&0&0&1}.
\end{equation*}
Also, for a set of representatives $x \in \mathcal{O}_K/m^{k}$ with $m \in \mathcal{O}_K$ and $k \in \mathbb{Z}$, we understand that we only take $0$ as the only representative if $k \leq 0$. We then have:
\begin{lemma}\label{lem: decomp}
Let 
\[
M = \textup{diag}(\pi^{a_1} \overline{\pi}^{b_1}, \textup{ }\pi^{a_2} \overline{\pi}^{b_2},\textup{ } \pi^{a_3} \overline{\pi}^{b_3},\textup{ } \pi^{a_4} \overline{\pi}^{b_4}) \in S_p^{2},
\]
with $a_i,b_i$ nonnegative integers. Then
\[
\Gamma_{1,1} M \Gamma_{1,1} = \sum_{\substack{q, l, r\\ \gamma \in V}} \Gamma_{1,1} M  \begin{pmatrix} 1 &  0& 0 &l \\ -\overline{q} & 1 & \overline{l} & r-l\overline{q} \\ 0 & 0 &  1 &q \\ 0 & 0 & 0 & 1 \end{pmatrix} \imath(\gamma),
\]

where $l,q,r$ run over elements in $\mathcal{O}_K$ that satisfy $r \in \mathbb{Z}$ and they give representatives of 
\begin{equation*}
    l \in \mathcal{O}_K / \pi^{a_4-a_1} \overline{\pi}^{b_4-b_1},\textup{ } q \in \mathcal{O}_K / \pi^{a_4-a_3} \overline{\pi}^{b_4-b_3}, \textup{ }r \in \mathbb{Z} / p^{a_4-a_2}.
\end{equation*}

Finally, $\gamma$ runs over a set $V$ such that
\[
\Gamma_1 \textup{diag}(\pi^{a_1}\overline{\pi}^{b_1} , \pi^{a_3} \overline{\pi}^{b_3}) \Gamma_1 = \sum_{\gamma \in V} \Gamma_1 \textup{diag}(\pi^{a_1}\overline{\pi}^{b_1} , \pi^{a_3} \overline{\pi}^{b_3}) \gamma
\]
is a decomposition into distinct right cosets relative to $\Gamma_1$.
\end{lemma}

\begin{proof}
We write
\begin{equation*}
    \mathcal{H}_{1,1} := \left\{ \begin{pmatrix} 1 &  0& 0 &l \\ -\overline{q} & 1 & \overline{l} & r \\ 0 & 0 &  1 &q \\ 0 & 0 & 0 & 1 \end{pmatrix} \in \Gamma_{1,1} \mid  l,q,r \in \mathcal{O}_K \right\}
\end{equation*}
for the (integral) Heisenberg part of the Klingen parabolic. We then claim that
\begin{equation}\label{heisenberg}
\mathcal{H}_{1,1} M \mathcal{H}_{1,1} = \sum_{l,q,r} \mathcal{H}_{1,1} M \begin{pmatrix} 1 &  0& 0 &l \\ -\overline{q} & 1 & \overline{l} & r-l\overline{q} \\ 0 & 0 &  1 &q \\ 0 & 0 & 0 & 1 \end{pmatrix},
\end{equation}
where $q,l,r$ are as in the statement of the Lemma. 

To see this, we first set $h(l,q,r):= \begin{pmatrix} 1 &  0& 0 &l \\ -\overline{q} & 1 & \overline{l} & r \\ 0 & 0 &  1 &q \\ 0 & 0 & 0 & 1 \end{pmatrix}$
and $M = \textup{diag}( \alpha_1,\alpha_2,\alpha_3, \alpha_4)$ and calculate
\[
h(l_1,q_1,r_1) M h(l_2,q_2,r_2) = 
\begin{pmatrix} \alpha_1 & 0 & 0 &\alpha_1 l_2 + l_1 \alpha_4 \\
-\overline{q}_1\alpha_1-\alpha_2 \overline{q}_2 & \alpha_2 & \alpha_2 \overline{l}_2 + \overline{l}_1 \alpha_3 & * \\
0 & 0 & \alpha_3 & \alpha_3 q_2 + q_1 \alpha_4 \\
0 & 0 & 0 &\alpha_4
\end{pmatrix},
\]
where $* = -\overline{q}_1 \alpha_1 l_2 + \alpha_2 r_2 + \overline{l}_1 \alpha_3 q_2 + r_1 \alpha_4$. We first look at the upper right entry. We have
\[
\alpha_1 l_2 + l_1 \alpha_4 = \pi^{a_1}\overline{\pi}^{b_1} l_2 + l_1 \pi^{a_4}\overline{\pi}^{b_4}.
\]

If $a_1 \geq a_4$ and $b_1 \geq b_4$, we may write
\[
\alpha_1 l_2 + l_1 \alpha_4  =\left(\pi^{a_1-a_4}\overline{\pi}^{b_1-b_4} l_2 + l_1 \right) \pi^{a_4}\overline{\pi}^{b_4},
\]
and so we do not need right cosets in this case. In the other cases, we write
\begin{equation*}
    l_2 = x + y \pi^{a_4-a_1}\overline{\pi}^{b_4-b_1},
\end{equation*}
with the understanding that we set $\pi^i=1$ and $\overline{\pi}^j =1$ if $i,j \leq 0$. We then have 
\[
\alpha_1 l_2 + l_1 \alpha_4 = \pi^{a_1}\overline{\pi}^{b_1} x + (l_1+y) \pi^{a_4}\overline{\pi}^{b_4}.
\]
For example, when $a_4 > a_1$ and $b_1 \geq b_4$, we write $l_2 = x + y \pi^{a_4-a_1}$ and obtain
\[
\alpha_1 l_2 + l_1 \alpha_4 = \pi^{a_1} \overline{\pi}^{b_1}x + y \pi^{a_4} \overline{\pi}^{b_1} + l_1 \pi^{a_4} \overline{\pi}^{b_4} = \pi^{a_1} \overline{\pi}^{b_1}x + (y \overline{\pi}^{b_1-b_4} + l_1) \pi^{a_4}\overline{\pi}^{b_4}.
\]

Similarly, looking at the entry 
\[
\alpha_2 \overline{l}_2 + \overline{l}_1 \alpha_3 = \pi^{a_2}\overline{\pi}^{b_2} \overline{l}_2 + \overline{l}_1 \pi^{a_3}\overline{\pi}^{b_3}
\]
with $l_2 = x + y \pi^{a_4-a_1}\overline{\pi}^{b_4-b_1}$ as above, we obtain,
\[
\alpha_2 \overline{l}_2 + \overline{l}_1 \alpha_3= \pi^{a_2}\overline{\pi}^{b_2}\overline{x} + \overline{y} \pi^{a_2}\overline{\pi}^{b_2}\overline{\pi}^{a_4-a_1}\pi^{b_4-b_1}+ \overline{l}_1 \pi^{a_3}\overline{\pi}^{b_3} = \pi^{a_2}\overline{\pi}^{b_2}\overline{x} + (\overline{y} + \overline{l_1}) \pi^{a_3}\overline{\pi}^{b_3},
\]
where we have used the fact that $a_2+b_4 = b_1 + a_3$ and $a_4 + b_2 = a_1 + b_3$, since $M \in S_p^2$.\\

In particular, for these entries it is enough to consider the entry $l_2$ modulo $\pi^{a_4-a_1}\overline{\pi}^{b_4-b_1}$ (with our convention). Similarly, by looking at the entries $-\overline{q}_1\alpha_1-\alpha_2 \overline{q}_2$ and $\alpha_3 q_2 + q_1 \alpha_4$, we obtain the corresponding result for $q_2$.\\

We are now left with the $*$ entry.
Using the fact that we can write $r_1 = r'_1 - l_1\overline{q}_1$ and $r_2 = r'_2 - l_2 \overline{q}_2$ for some $r_i' \in \mathbb{Z}$, we have that the only part of the $*$ entry which is not determined by our choices of $q_1,q_2,l_1,l_2$ is $\alpha_2 r_2'+ r_1' \alpha_4$. But
\[
\alpha_2 r'_2 + r'_1 \alpha_4 = \pi^{a_2}\overline{\pi}^{b_2} r'_2 + r'_1 \pi^{a_4}\overline{\pi}^{b_4}.
\]
Arguing as above and using the fact that $a_4-a_2 = b_4-b_2$, we see that the element $r'_2$ needs to be selected from integers modulo $\pi^{a_4-a_2} \overline{\pi}^{b_4-b_2} = p^{a_4-a_2}$. \\

This establishes our claim of equation \eqref{heisenberg}. The rest of the proof is identical to the symplectic case, as done by Gritsenko in \cite[Lemma 3.1]{gritsenko1}.
\end{proof}
We now define the elements
\begin{equation*}
    T_{\pm}(p^{\delta}) := j_{\pm}\left(T\left(p^{\delta}\right)\right), \textup{ }\Lambda_{\pm}(\pi^{\delta}) := j_{\pm}\left(\Gamma_1\textup{diag}\left(\pi^{\delta}, \pi^{\delta}\right)\Gamma_1\right), \textup{ }\delta \geq 1,
\end{equation*}
as in the case of an inert prime (see equations \eqref{+- embeddings}, \eqref{t(p)}) and similarly for $\Lambda_{\pm}(\overline{\pi}^{\delta})$. In particular, using Lemma \ref{lem: decomp}, or translating back to the Hecke algebra of $\textup{GL}_4$, we obtain
\begin{equation}\label{mult lambda}
    \Lambda_{\pm}(\pi^{\delta}) = \Lambda_{\pm}(\pi^{\delta-1})\Lambda_{\pm}(\pi), \textup{ }\delta \geq 1.
\end{equation}
This implies $\Lambda_{\pm}(\pi^{\delta}) = \Lambda_{\pm}(\pi)^{\delta}, \textup{ }\delta \geq 1$. The same holds for $\Lambda_{\pm}(\overline{\pi})$.\\

We are now finally ready to obtain the rationality Theorems, as in the case of an inert prime. Assume $F \in S_{2}^{k}$ has a Fourier-Jacobi expansion as in equation \eqref{fourier-expansions} and $Q_{p}^{(2)}$ denotes the $p$-factor of Gritsenko's $L$-function, as in Definition \ref{gritsenko l-function}.
\begin{proposition} \label{prop:f1}
Let $F \in S_{2}^{k}$ be a Hecke eigenform for $H(\Gamma_2, S^2)$ and $m \geq 1$. Then
\begin{align*}
    Q_{p,F}^{(2)}(X)\sum_{\delta \geq 0}\phi_{mp^{\delta}}\textup{ }|\textup{ }T_{+}(p^{\delta})X^{\delta} = \left(\phi_m - \phi_{m/p} \textup{ }|\textup{ }T_{-}(p)X + p\phi_{m/p^{2}}\textup{ }|\textup{ }\Lambda_{-}(p)X^2\right)\mid B(X),
\end{align*}
where we define $B$ to be the middle polynomial of degree $4$ in the factorisation of $Q_{p}^{(2)}(t)$ given in \cite[Proposition 3.2, (3)]{gritsenko}. In particular, we have
\begin{equation*}
    B(t) = 1 - B_1t + B_2t^2 - B_3t^3 + B_4t^4 \in H_{p}^{1,1}[t],
\end{equation*}
where
\begin{align}
\begin{split}\label{coeffs of K}
    B_1 &= T(\overline{\pi}, \pi) + T(\pi, \overline{\pi}), \textup{ }B_2 = p\left(\Lambda_{+}(\pi)\Lambda_{-}(\overline{\pi})+\Lambda_{+}(\overline{\pi})\Lambda_{-}(\pi)\right)-p\nabla_p+(p^2-p^4)\Delta_p,\\
    &B_3 =p^{3}\left(\Delta_{\pi}\Lambda_{+}(\overline{\pi})\Lambda_{-}(\overline{\pi})+\Delta_{\overline{\pi}}\Lambda_{+}(\pi)\Lambda_{-}(\pi\right))-p^4\Delta_pB_1, \textup{ }B_4 = p^5\Delta_p\nabla_p -p^6\Delta_p^2.
\end{split}
\end{align}
Moreover, from \cite[Proposition 4.2]{gritsenko}, we have
\begin{equation*}
\phi_{m}\mid B(X) = (1-p^{k-3}X)^2 (1-p^{2k-4}X^2)\phi_m,
\end{equation*}
if $(m,p)=1$.
\end{proposition}
\begin{proof}
    The proof follows by \cite[Proposition 4.1]{gritsenko} and the fact that
    \begin{equation*}
        Q_{p,-}^{(1)}(t) = 1 - T_{-}(p)t + p\Lambda_{-}(p)t^2,
    \end{equation*}
    as is defined in \cite[Proposition 3.2, (1)]{gritsenko}.
\end{proof}
\begin{proposition}\label{prop:f2}
Let $F \in S_{2}^{k}$ be a Hecke eigenform for $H(\Gamma_2, S^{2})$ and $m \geq 1$. Then,
\begin{equation*}
D_{\pi,F}^{(2)}(X)\sum_{\delta \geq 0}\phi_{mp^{\delta}}\textup{ }|\textup{ }\Lambda_{+}(\overline{\pi}^{\delta})X^{\delta} = \left(\phi_m - \phi_{m/p}\textup{ }|\textup{ }\Lambda_{-}(\overline{\pi})X\right)\mid S_{\pi}(X).
\end{equation*}
Similarly,
\begin{equation*}
D_{\overline{\pi},F}^{(2)}(X)\sum_{\delta \geq 0}\phi_{mp^{\delta}}\textup{ }|\textup{ }\Lambda_{+}(\pi^{\delta})X^{\delta} = \left(\phi_m - \phi_{m/p}\textup{ }|\textup{ }\Lambda_{-}(\pi)X\right)\mid S_{\overline{\pi}}(X),
\end{equation*}
with notation obtained by exchanging $\pi$ and $\overline{\pi}$. Here, $S_{\pi}, S_{\overline{\pi}}$ are the polynomials appearing in Proposition \ref{s polynomials}.
\end{proposition}
\begin{proof} 
We have, in the same way as in \cite[Corollary, p. 264]{gritsenko1}, the proof of which can be found in \cite[Proposition 5.2]{gritsenko1}, that
\[
\phi_m \textup{ }||_k \textup{ }\Lambda_+(\overline{\pi}) = \phi_{mp}\textup{ }|_k\textup{ } \Lambda_+(\overline{\pi}).
\]
Then, inductively, using equation \eqref{mult lambda}, we obtain
\begin{equation*}
    \phi_m \textup{ }||_k \textup{ }\Lambda_+(\overline{\pi}^{\delta}) = \phi_{mp}\textup{ }|_k\textup{ } \Lambda_+(\overline{\pi}^{\delta}), \textup{ }\forall \delta \geq 1.
\end{equation*}
Now, since $F$ is an eigenfunction for $H(\Gamma_{2}, S^{2})$, we have $D_{\pi, F}^{(2)}(t)\phi_m = \phi_m \mid\mid D_{\pi}^{(2)}(t)$. Then, we can write
\begin{align*}
    D_{\pi,F}^{(2)}(X)\sum_{\delta \geq 0}\phi_{mp^{\delta}}\textup{ }|\textup{ }\Lambda_{+}(\overline{\pi}^{\delta})X^{\delta} &= D_{\pi,F}^{(2)}(X)\sum_{\delta \geq 0}(\phi_{m}\mid\mid\Lambda_{+}(\overline{\pi}^{\delta}))X^{\delta} =(\phi_{m}\mid\mid D_{\pi}^{(2)}(X))\mid\mid\sum_{\delta \geq 0}\Lambda_{+}(\overline{\pi}^{\delta})X^{\delta} \\
    &=  (\phi_{m}\mid\mid D_{\pi}^{(2)}(X))\mid\mid\sum_{\delta \geq 0}\Lambda_{+}(\overline{\pi})^{\delta}X^{\delta} = \phi_m \mid\mid (1 - \Lambda_{-}(\overline{\pi})X)S_{\pi}(X) \\
    &= (\phi_m - \phi_{m/p} \mid \Lambda_{-}(\overline{\pi})X) \mid S_{\pi}(X),
\end{align*}
as claimed. In the equalities above, we used that $\Lambda_{+}(\overline{\pi}^{\delta}) = \Lambda_{+}(\overline{\pi})^{\delta}$ from equation \eqref{mult lambda} and Proposition \ref{s polynomials}.
\end{proof}
To end this Section, we will prove a couple of Lemmas, which will be useful later, when we are dealing with the Dirichlet series of interest.
\begin{lemma}\label{adjoint}
Denote by $\Lambda_{-}(\pi)^{\textup{adj}}$ the adjoint of the operator $\Lambda_{-}(\pi)$ with respect to the inner product of Fourier-Jacobi forms. Then
\begin{equation*}
\Lambda_{-}(\pi)^{\textup{adj}} = p^{k-3}\Lambda_+(\overline{\pi}).
\end{equation*}
\end{lemma}
\begin{proof}
Let $\phi_l, \psi_{lp}$ be two Fourier-Jacobi forms of weight $k$ and of index $l, lp$ respectively. We observe from Lemma \ref{lem: decomp} that $\Lambda_{-}(\pi) = \Gamma_{1,1} \textup{diag}(\pi, p, \pi, 1)$. We also note that the Jacobi form $\phi_{l} \mid_{k} \textup{diag}(\pi, p, \pi, 1)$ is of index $lp$ for the group 
\begin{equation*}
    \Gamma_{-} := \Gamma_1 \times \mathcal{O}_{K} \times \mathcal{O}_{K} = \left\{\m{a&0&b&\mu\\**&1&*&*\\c&0&d&\lambda \\0&0&0&1} \in \Gamma_{1,1} \mid \m{a&b\\c&d} \in \Gamma_{1}, \textup{ }\mu,\lambda \in \mathcal{O}_K\right\}.
\end{equation*}
In particular, we may write

\[
\langle \phi_{l} \mid_{k} \Lambda_{-}(\pi), \psi_{lp} \rangle = \frac{p^{2k-4} \pi^{-k}}{\left[\Gamma_{1,1} : \Gamma_{-}\right]} \int_{\Gamma_{-}\setminus \mathbb{H}_{1}^{J}} \phi_{l}(\tau,\pi z_1, \overline{\pi} z_2) \overline{\psi_{lp}(\tau,z_1,z_2)} \exp\left(-\pi lp \frac{|z_1-\overline{z_2}|^2}{y} \right)y^{k-4} \hbox{d}\mu,
\]
as in Definition \ref{inner_product_jacobi}.
We now perform the change of variables $z_1 \longmapsto \pi^{-1} z_1$ and $z_2 \longmapsto \overline{\pi}^{-1} z_2$. This is equivalent to the action of the matrix $\textup{diag}(\overline{\pi}, 1, \overline{\pi}, p)$ on $\mathbb{H}_1^{J}:= \mathbb{H}_1 \times \mathbb{C} \times \mathbb{C}$. Now
\begin{equation*}
    \left(\psi_{lp}\mid_k \textup{diag}(\overline{\pi}, 1, \overline{\pi}, p) \right)(\tau, z_1, z_2) = p^{k-4}\overline{\pi}^{-k}\psi_{lp}(\tau, \pi^{-1}z_1, \overline{\pi}^{-1}z_2)
\end{equation*}
is a Jacobi form of weight $k$ and index $l$ with respect to the group
\begin{equation*}
    \Gamma_{+} := \Gamma_1 \times \overline{\pi}\mathcal{O}_K \times \overline{\pi}\mathcal{O}_K = \left\{\m{a&0&b&\mu\\**&1&*&*\\c&0&d&\lambda \\0&0&0&1}\in \Gamma_{1,1} \mid \m{a&b\\c&d} \in \Gamma_{1}, \textup{ } \mu, \lambda \equiv 0\pmod{\overline{\pi}}\right\}.
\end{equation*}
This group is obtained by considering the group $\textup{diag}(\overline{\pi}, 1, \overline{\pi}, p)^{-1}\Gamma_{-}\textup{diag}(\overline{\pi}, 1, \overline{\pi}, p)$. We therefore have
\[
\langle \phi_{l} \mid_{k} \Lambda_{-}(\pi), \psi_{lp} \rangle = \frac{p^{2k-6} \pi^{-k}}{\left[\Gamma_{1,1}:\Gamma_{-}\right]} \int_{\Gamma_{+}\setminus \mathbb{H}_{1}^{J}} \phi_l(\tau, z_1,  z_2) \overline{\psi_{lp}(\tau,\pi^{-1} z_1,\overline{\pi}^{-1} z_2)} \exp\left(-\pi l \frac{|z_1-\overline{z_2}|^2}{y} \right)y^{k-4} \hbox{d}\mu
\] 

On the other hand, we have by Lemma \ref{lem: decomp}, that 
\begin{equation*}
    \Lambda_{+}(\overline{\pi}) = \sum_{a,b \in \mathcal{O}_K / \pi, \textup{ }c \in \mathbb{Z}/p}\Gamma_{1,1}\m{\overline{\pi}&0&0&0\\0&1&0&0\\0&0&\overline{\pi}&0\\0&0&0&p}h(a,b,c),
\end{equation*}
where $h(a,b,c):= \begin{pmatrix} 1 &  0& 0 &a \\ -\overline{b} & 1 & \overline{a} & c-a\overline{b} \\ 0 & 0 &  1 &b \\ 0 & 0 & 0 & 1 \end{pmatrix}$. By now using the fact that 
\begin{equation*}
    \langle \phi_{l} \mid_{k} h(a,b,c)^{-1}, \psi_{lp} \rangle = \langle \phi_{l} , \psi_{lp} \mid_{k}h(a,b,c)\rangle,
\end{equation*}
we obtain
\begin{equation*}
\langle\phi_l, \psi_{lp}\mid_k \Lambda_{+}(\overline{\pi})\rangle = p^3\frac{p^{k-4}\pi^{-k}}{[\Gamma_{1,1}:\Gamma_{+}]}\int_{\Gamma_{+} \backslash \mathbb{H}_1^{J}} \phi_{l}(\tau, z_1, z_2)\overline{\psi_{lp}(\tau, \pi^{-1}z_1, \overline{\pi}^{-1}z_2)}e\left(-\pi l \frac{|z_1-\overline{z_2}|^2}{y}\right)y^{k-4}\hbox{d}\mu,
\end{equation*}
as $h(l,q,r)^{-1} \in \Gamma_{1,1}$ and hence they act trivially on $\psi_{lp}$. Therefore,
\begin{equation*}
    \langle \phi_l\mid_{k} \Lambda_{-}(\pi), \psi_{lp}\rangle =p^{k-5}\frac{[\Gamma_{1,1}:\Gamma_{+}]}{[\Gamma_{1,1}:\Gamma_{-}]}\langle\phi_l, \psi_{lp}\mid \Lambda_{+}(\overline{\pi})\rangle = p^{k-3}\langle\phi_l, \psi_{lp}\mid_k \Lambda_{+}(\overline{\pi})\rangle,
\end{equation*}
as the ratio of indices is $p^2$. The result now follows.
\end{proof}
Finally, knowing the action of the operators $T(\pi, \overline{\pi})$ and $T(\overline{\pi}, \pi)$ on $J_{k,1}^{2}$ will prove helpful in the following, so we give the following Lemma:
\begin{lemma}\label{action}
 Let $\phi \in J_{k,1}^{2}$. We then have
 \begin{equation*}
     \phi \mid_k T(\overline{\pi}, \pi)  = p^{k-3}\phi.
 \end{equation*}
 \end{lemma}
 \begin{proof}
 Using the decomposition in Lemma \ref{lem: decomp}, we can write
\begin{equation*}
    T(\overline{\pi},\pi) = \sum_{\gamma \in \mathcal{O}_K/ \pi, \textup{ }\beta \in \mathcal{O}_K/\overline{\pi}} \Gamma_{1,1}\m{\overline{\pi}&0&0&\overline{\pi}\overline{\beta}\\-\overline{\gamma}&\pi&\pi\beta&-\overline{\gamma}\overline{\beta}\\0&0&\overline{\pi}&\gamma\\0&0&0&\pi}.
\end{equation*}
The result now follows by \cite[Lemma 3.2]{gritsenko_maass}.
\end{proof}
\begin{remark}
The same is true for the operator $T(\pi, \overline{\pi}) = \Gamma_{1,1}\textup{diag}(\pi, \overline{\pi},\pi,\overline{\pi})\Gamma_{1,1}$.
\end{remark}
\subsection{Calculation of the Dirichlet series - First Part}
Assume $F,G,h$ satisfy the same assumptions as in the beginning of Subsection \ref{dirichlet series inert}. We recall from equation \eqref{dirichlet_series_defn} that
\begin{equation*}
    D_{F,G,h}(s)=4\beta_{k}\sum_{l,\epsilon, m}\langle\tilde{\phi}_{1}\mid T_{-}(m)U_{l}, \tilde{\psi}_{mN(l)}\rangle_{\mathcal{A}}a_{mN(\epsilon)}N(l)^{-s}N(\epsilon)^{-(k+s-1)}m^{-(2k+s-4)},
\end{equation*}
with $l,\epsilon \in \mathbb{Z}[i]$ coprime with their real parts positive and imaginary parts non-negative and $m \in \mathbb{N}$.
In the case of a split prime $p = \pi \overline{\pi}$, we define the $p$-part by
\begin{equation}\label{p-part, split}
    D_{F,G,h}^{(p)}(s) := \sum_{l_1,l_2, \epsilon_1, \epsilon_2, m \geq 0}\langle \widetilde{\phi}_1\mid T_{-}(p^{m})U_{\pi^{l_1}}U_{\overline{\pi}^{l_2}}, \widetilde{\psi}_{p^{m+l_1+l_2}}\rangle_{\mathcal{A}}a_{p^{m+\epsilon_1+\epsilon_2}}p^{-s(l_1+l_2)}p^{-(k+s-1)(\epsilon_1+\epsilon_2)}p^{-(2k+s-4)m},
 \end{equation}   
together with the conditions $\min(\epsilon_1,l_1) = 0$ and $\min(\epsilon_2,l_2) = 0$.\\\\
Consider now the Hecke operator 
\begin{equation*}
    \Lambda_{-}(\pi) = \Gamma_{1,1}\textup{diag}(\pi,p,\pi,1)\Gamma_{1,1} = \Gamma_{1,1}\textup{diag}(\pi,p,\pi,1),    
\end{equation*}
by Lemma \ref{lem: decomp}.
Then, if $\phi$ is a Fourier-Jacobi form of any index $m$, we get
\begin{equation}\label{relation Lambda, U}
    \phi\mid_{k}\Lambda_{-}(\pi) = p^{2k-4} \pi^{-k} \tilde{\phi}\left(\begin{pmatrix}\tau & \pi z_1\\\overline{\pi} z_2&p\tau'\end{pmatrix}\right)e^{-2\pi \frac{m}{p}\tau'} = p^{2k-4} \pi^{-k}\phi(\tau, \pi z_1, \overline{\pi}z_2) = p^{2k-4} \pi^{-k} \phi\mid_{k}U_{\pi}.
\end{equation}
Hence, we can rewrite the series as:
\begin{multline*}
D_{F,G,h}^{(p)}(s) = \sum_{\substack{l_1,l_2,\\ \epsilon_1, \epsilon_2,m \geq 0\\\textup{min}(l_i, \epsilon_i)=0}}\langle \tilde{\phi}_1\mid T_{-}(p^{m})\Lambda_{-}(\pi^{l_1})\Lambda_{-}(\overline{\pi}^{l_2}), \tilde{\psi}_{p^{m+l_1+l_2}}\rangle_{\mathcal{A}}a_{p^{m+\epsilon_1+\epsilon_2}} p^{(4-2k)l_1} p^{(4-2k)l_2} \pi^{l_1 k} \overline{\pi}^{l_2k}\times\\\times p^{-s(l_1+l_2)}p^{-(k+s-1)(\epsilon_1+\epsilon_2)}p^{-(2k+s-4)m}.
\end{multline*}
By then using an inclusion-exclusion argument, we have that the above series can be written as
\begin{equation*}
     D_{F,G,h}^{(p)}(s) = D_{(\epsilon_1,\epsilon_2)}(s)+D_{(l_1,l_2)}(s) + D_{(\epsilon_1,l_2)}(s) +
 D_{(\epsilon_2, l_1)}(s) \textup{ }-
 \end{equation*}
 \begin{equation*}
 -\textup{ }D_{(\epsilon_1,\epsilon_2,l_1)}(s) -D_{(\epsilon_2,l_1,l_2)}(s)- D_{(\epsilon_1,l_1,l_2)}(s)-D_{(\epsilon_1,\epsilon_2,l_2)}(s)+D_{(\epsilon_1, \epsilon_2, l_1,l_2)}(s),
 \end{equation*}
 where we use the same notation as in Subsection \ref{dirichlet series inert}, meaning that the corresponding index means the variables are $0$. We can then deal with the "easy" parts first, i.e. when the operators $\Lambda_{-}$ do not appear. We again consider $s \in \mathbb{R}$ big enough, as in Remark \ref{identity_theorem}.
 \begin{proposition}
 We have
 \begin{equation*}
     D_{(l_1,l_2)}(s)-D_{(l_1,l_2,\epsilon_1)}(s)-D_{(l_1,l_2,\epsilon_2)}(s)+D_{(l_1,l_2,\epsilon_1, \epsilon_2)}(s)=
 \end{equation*}
 \begin{equation*}
 =\frac{\langle\tilde{\phi}_1, \tilde{\psi}_1\rangle_{\mathcal{A}}}{\alpha_1-\alpha_2}\left[\frac{\alpha_1^3(1-p^{2k-4}X_1^2)X^2}{Q_{p,G}^{(2)}(X_1)} - \frac{\alpha_2^3(1-p^{2k-4}X_2^2)X^2}{Q_{p,G}^{(2)}(X_2)}\right],
 \end{equation*}
 where $X = p^{-(k+s-1)}$ and $X_i = \alpha_ip^{-(2k+s-4)}$, as in Subsection \ref{dirichlet series inert}.
 \end{proposition}
 \begin{proof}
 We have
 \begin{align*}
 D_{(l_1,l_2)}(s)&= \sum_{\epsilon_1,\epsilon_2,m\geq0}\langle\tilde{\phi}_1\mid T_{-}(p^{m}),\tilde{\psi}_{p^{m}}\rangle_{\mathcal{A}} a_{p^{m+\epsilon_1+\epsilon_2}}p^{-(k+s-1)(\epsilon_1+\epsilon_2)}p^{-m(2k+s-4)},\\
 D_{(l_1,l_2,\epsilon_1)}(s) &= \sum_{\epsilon_2,m\geq 0}\langle\tilde{\phi}_1\mid T_{-}(p^{m}),\tilde{\psi}_{p^{m}}\rangle_{\mathcal{A}}a_{p^{m+\epsilon_2}}p^{-(k+s-1)\epsilon_2}p^{-m(2k+s-4)},\\
 D_{(l_1,l_2,\epsilon_2)}(s) &= \sum_{\epsilon_1,m\geq 0}\langle\tilde{\phi}_1\mid T_{-}(p^{m}),\tilde{\psi}_{p^{m}}\rangle_{\mathcal{A}}a_{p^{m+\epsilon_1}}p^{-(k+s-1)\epsilon_1}p^{-m(2k+s-4)},\\
 D_{(l_1,l_2,\epsilon_1,\epsilon_2)}(s) &= \sum_{m\geq 0}\langle\tilde{\phi}_1\mid T_{-}(p^{m}),\tilde{\psi}_{p^{m}}\rangle_{\mathcal{A}}a_{p^{m}}p^{-m(2k+s-4)}.
 \end{align*}
 Using now the fact that $a_{p^m} = (\alpha_1^{m+1}-\alpha_2^{m+1})/(\alpha_1-\alpha_2)$ and the fact that the adjoint of $T_{-}(p^{m})$ is $T_{+}(p^m)$, when they are acting on $P$-forms (see also in \cite[Proposition $5.1$]{gritsenko} and Subsection \ref{dirichlet series inert}), we get
 \begin{align*}
    D_{(l_1,l_2)}(s)(\alpha_1-\alpha_2)&= \alpha_1\sum_{\epsilon_1,\epsilon_2,m\geq 0}\langle\tilde{\phi}_1,\tilde{\psi}_{p^{m}}\mid T_{+}(p^m)\rangle_{\mathcal{A}} (\alpha_1 p^{-(k+s-1)})^{\epsilon_1+\epsilon_2}(\alpha_1 p^{-(2k+s-4)})^m\\
    &-\alpha_2 \sum_{\epsilon_1,\epsilon_2,m\geq 0}\langle\tilde{\phi}_1,\tilde{\psi}_{p^{m}}\mid T_{+}(p^m)\rangle_{\mathcal{A}} (\alpha_2 p^{-(k+s-1)})^{\epsilon_1+\epsilon_2}(\alpha_2 p^{-(2k+s-4)})^m,
 \end{align*}
 and similarly for the others. Now, by Proposition \ref{prop:f1}, we obtain, as in Proposition \ref{easy_part_dirichlet_inert}:
 \begin{equation*}
     \sum_{m=0}^{\infty}\langle\tilde{\phi}_1,\tilde{\psi}_{p^{m}}\mid T_{+}(p^m)\rangle_{\mathcal{A}}X_1^m = (1-p^{k-3}X_1)^2(1-p^{2k-4}X_1^2)\langle\tilde{\phi}_1, \tilde{\psi}_1\rangle_{\mathcal{A}}Q_{p,G}^{(2)}(X_1)^{-1}.
 \end{equation*}
 Also, $\sum_{\epsilon_2=0}^{\infty}(\alpha_1 p^{-(k+s-1)})^{\epsilon_2} = \frac{1}{1-\alpha_1p^{-(k+s-1)}}$ and similarly for $\epsilon_1$ and
 \begin{equation*}\sum_{\epsilon_1, \epsilon_2=0}^{\infty}(\alpha_1 p^{-(k+s-1)})^{(\epsilon_1+\epsilon_2)} = \left(\frac{1}{1-\alpha_1p^{-(k+s-1)}}\right)^2.
 \end{equation*}
 Hence, we obtain
 \begin{equation*}
     D_{(l_1,l_2)}(s)-D_{(l_1,l_2,\epsilon_1)}(s)-D_{(l_1,l_2,\epsilon_2)}(s)+D_{(l_1,l_2,\epsilon_1, \epsilon_2)}(s)=
 \end{equation*}
 \begin{align*}
     =\frac{\langle\tilde{\phi}_1,\tilde{\psi}_1\rangle_{\mathcal{A}}}{\alpha_1-\alpha_2}&\left[\frac{\alpha_1(1-p^{k-3}X_1)^2(1-p^{2k-4}X_1^2)}{Q_{p,G}^{(2)}(X_1)}\left(1-\frac{2}{1-\alpha_1p^{-(k+s-1)}}+\left(\frac{1}{1-\alpha_1p^{-(k+s-1)}}\right)^2\right)- \right.\\ 
     &\left. -\frac{\alpha_2(1-p^{k-3}X_2)^2(1-p^{2k-4}X_2^2)}{Q_{p,G}^{(2)}(X_2)}\left(1-\frac{2}{1-\alpha_2p^{-(k+s-1)}}+\left(\frac{1}{1-\alpha_2p^{-(k+s-1)}}\right)^2\right)\right]=
\end{align*}
\begin{equation*}
     =\frac{\langle\tilde{\phi}_1, \tilde{\psi}_1\rangle_{\mathcal{A}}}{\alpha_1-\alpha_2}\left[\frac{\alpha_1^3(1-p^{2k-4}X_1^2)X^2}{Q_{p,G}^{(2)}(X_1)} - \frac{\alpha_2^3(1-p^{2k-4}X_2^2)X^2}{Q_{p,G}^{(2)}(X_2)}\right].
 \end{equation*}
 \end{proof}
\subsection{Calculation of the Dirichlet Series - Second Part}
In the following, we define
\begin{equation*}
    \textup{ }Y_1 := \pi^{k}p^{-(2k+s-4)}, \textup{ }Y_2 := \overline{\pi}^{k}p^{-(2k+s-4)}, \textup{ }X := p^{-(k+s-1)}, \textup{ } X_i := \alpha_i p^{-(2k+s-4)}, \textup{ }i=1,2.
\end{equation*}
Let us now consider the series
\begin{equation*}
    D_{(\epsilon_1, \epsilon_2, l_2)}(s) = \sum_{l_1, m \geq 0}\langle \tilde{\phi}_1\mid T_{-}(p^{m})U_{\pi^{l_1}}, \tilde{\psi}_{p^{m+l_1}}\rangle_{\mathcal{A}}a_{p^{m}}p^{-sl_1}p^{-(2k+s-4)m}.
\end{equation*}
Using the fact that $a_{p^m} = (\alpha_1^{m+1}-\alpha_2^{m+1})/(\alpha_1-\alpha_2)$ and the relation between $U_\pi$ and $\Lambda_{-}(\pi)$ of equation \eqref{relation Lambda, U}, we obtain that
\begin{equation*}
    (\alpha_1-\alpha_2)D_{(\epsilon_1, \epsilon_2, l_2)}(s) = \alpha_1S_1(s) - \alpha_2S_2(s),
\end{equation*}
where 
\begin{equation*}
    S_i(s) := \sum_{l, m \geq 0}\langle \tilde{\phi}_1 \mid T_{-}(p^{m})\Lambda_{-}(\pi^{l}), \tilde{\psi}_{p^{m+l}}\rangle_{\mathcal{A}}p^{-(2k+s-4)l}\pi^{lk}(\alpha_ip^{-(2k+s-4))})^m.
\end{equation*}
Using now the fact that the adjoint (with respect to the inner product of $P$-forms) of $T_{-}(p)$ is $T_{+}(p)$ and of $\Lambda_{-}(\pi)$ is $\Lambda_{+}(\overline{\pi})$ (Lemma \ref{adjoint}) and that $T_{-}(p)$ and $\Lambda_{-}(\pi)$ commute, we get
\begin{equation*}
    S_i(s) = \sum_{l, m \geq 0}\langle \tilde{\phi}_1, \tilde{\psi}_{p^{m+l}}\mid T_{+}(p^{m})\Lambda_{+}(\overline{\pi}^{l})\rangle_{\mathcal{A}}X_i^{m}Y_1^{l} = \sum_{l, m \geq 0}\langle \tilde{\phi}_1, \tilde{\psi}_{p^{m+l}}\mid T_{+}(p^{m})\Lambda_{+}(\overline{\pi}^{l})Y_2^{l}\rangle_{\mathcal{A}}X_i^{m},
\end{equation*}
because we have a Hermitian inner product (and therefore we have to conjugate in the second component of the inner product). We remind the reader that we work with $s \in \mathbb{R}$ big enough.
\begin{lemma}
For $i=1, 2$, we have
\begin{equation*}
    \sum_{l, m \geq 0} \tilde{\psi}_{p^{m+l}}\mid T_{+}(p^{m})\Lambda_{+}(\overline{\pi}^{l})X_i^{m}Y_2^{l} = \frac{1}{Q_{p,G}^{(2)}(X_i)}\sum_{l\geq 0}\left[\tilde{\psi}_{p^{l}} - \tilde{\psi}_{p^{l-1}}|T_{-}(p)X_i + p\tilde{\psi}_{p^{l-2}}|\Lambda_{-}(p)X_i^2\right]\mid B(X_i)\Lambda_{+}(\overline{\pi}^{l})Y_2^{l},
\end{equation*}
with $B$ the polynomial of Proposition \ref{prop:f1}.
\end{lemma}
\begin{proof}
The proof follows immediately by Proposition \ref{prop:f1}.
\end{proof}
Let us now compute each of the sums occuring above.
\begin{proposition}\label{comp1}
For $i=1,2$, we have
\begin{equation*}
    \frac{1}{1-\alpha_i X}\sum_{l\geq 0}\tilde{\psi}_{p^l}\mid B(X_i)\Lambda_{+}(\overline{\pi}^l)Y_2^{l} = 
\end{equation*}
\begin{equation*}
    \frac{\tilde{\psi}_1\mid S_{\pi}(Y_2)}{{D_{\pi, G}^{(2)}(Y_2)}}-p^2 \frac{\left[\tilde{\psi}_p-\tilde{\psi}_1\mid\Lambda_{-}(\overline{\pi})Y_2\right]\mid S_{\pi}(Y_2)\Lambda_{+}(\pi)\Delta_{\overline{\pi}}Y_2X_i}{D_{\pi, G}^{(2)}(Y_2)} + \left[(1-\alpha_iX)(1-p^{2k-4}X_i^2)-1\right]\tilde{\psi}_1.
\end{equation*}
\end{proposition}
\begin{proof}
Using the commutativity relations from Table \ref{table:1}, we obtain (using also equation \eqref{coeffs of K})
\begin{equation}\label{k_2lambda+}
    B_2\Lambda_{+}(\overline{\pi}) = p^2\Delta_{\overline{\pi}}\Lambda_{+}(\pi)T(\overline{\pi},\pi), \textup{ }B_3\Lambda_{+}(\overline{\pi}) = B_4\Lambda_{+}(\overline{\pi}) = 0
\end{equation}
Hence, from the rationality theorem given in Proposition \ref{prop:f2}, we obtain
\begin{itemize}
    \item $\sum_{l\geq 0} \tilde{\psi}_{p^{l}} \mid \Lambda_{+}(\overline{\pi}^{l})Y_2^{l} = \frac{\tilde{\psi}_1\mid S_{\pi}(Y_2)}{D_{\pi, G}^{(2)}(Y_2)}$.
    \item $\sum_{l\geq 0} \tilde{\psi}_{p^{l}}\mid T(\overline{\pi}, \pi)\Lambda_{+}(\overline{\pi}^{l})Y_2^{l} = \sum_{l\geq 0} \tilde{\psi}_{p^{l}}\mid \Lambda_{+}(\overline{\pi}^{l})T(\overline{\pi}, \pi)Y_2^{l}  = \frac{\tilde{\psi}_1 \mid S_{\pi}(Y_2)T(\overline{\pi},\pi)}{D_{\pi, G}^{(2)}(Y_2)}$.
    \item $\begin{aligned}[t]\sum_{l \geq 0} \tilde{\psi}_{p^{l}} \mid T(\pi, \overline{\pi})\Lambda_{+}(\overline{\pi}^{l})Y_2^{l} &= \tilde{\psi}_1\mid T(\pi, \overline{\pi}) + p^2 \sum_{l \geq 1}\tilde{\psi}_{p^{l}}\mid \Delta_{\overline{\pi}}\Lambda_{+}(\pi)\Lambda_{+}(\overline{\pi}^{l-1})Y_2^{l}\\ &= \tilde{\psi}_1\mid T(\pi, \overline{\pi}) + p^2\frac{\left[\tilde{\psi}_p-\tilde{\psi}_{1}\mid \Lambda_{-}(\overline{\pi})Y_2\right]\mid S_{\pi}(Y)\Lambda_{+}(\pi)\Delta_{\overline{\pi}}Y_2}{D_{\pi, G}^{(2)}(Y_2)}\end{aligned}$.
    \item $\begin{aligned}[t]\sum_{l \geq 0} \tilde{\psi}_{p^{l}}\mid B_2 \Lambda_{+}(\overline{\pi}^{l})Y_2^{l} &= \tilde{\psi}_1 \mid B_2 + p^2 \sum_{l \geq 1} \tilde{\psi}_{p^{l}} \mid \Lambda_{+}(\overline{\pi}^{l-1})\Lambda_{+}(\pi)T(\overline{\pi}, \pi)\Delta_{\overline{\pi}}Y_2^{l}\\ &=\tilde{\psi}_{1} \mid B_2 + p^2\frac{\left[\tilde{\psi}_p - \tilde{\psi}_1 \mid \Lambda_{-}(\overline{\pi})Y_2\right] \mid S_{\pi}(Y_2)\Lambda_{+}(\pi)\Delta_{\overline{\pi}}T(\overline{\pi}, \pi)Y_2}{D_{\pi, G}^{(2)}(Y_2)}\end{aligned}$.
    \item $\sum_{l \geq 0} \tilde{\psi}_{p^{l}} \mid B_3 \Lambda_{+}(\overline{\pi}^{l})Y_2^{l} = \tilde{\psi}_1 \mid B_3$.
    \item $\sum_{l \geq 0} \tilde{\psi}_{p^{l}} \mid B_4 \Lambda_{+}(\overline{\pi}^{l})Y_2^{l} = \tilde{\psi}_1 \mid B_4$.
\end{itemize}
By putting all these together and then using Lemma \ref{action}, together with the fact that $a_iX = p^{k-3}X_i$, we obtain the result.
\end{proof}
Let us now consider the third sum.
\begin{proposition}\label{lem: Lambda-}
For $i=1,2$, we have
\begin{align*}
    \frac{1}{1-\alpha_iX}\sum_{l \geq 0}\tilde{\psi}_{p^{l-2}}\mid \Lambda_{-}(p)X_i^2 \mid B(X_i)\Lambda_{+}(\overline{\pi}^{l})Y_2^{l} &= \frac{\tilde{\psi}_1 \mid S_{\pi}(Y_2)U_{\pi}(X_i)X_i^2Y_2^2}{D_{\pi, G}^{(2)}(Y_2)} \\
    &= \frac{p^{2k-4}(p^{k-3}-p^{2k-4}X_i)\tilde{\psi}_{1}\mid S_{\pi}(Y_2)\Delta_{\overline{\pi}}Y_2^2X_i^2}{D_{\pi, G}^{(2)}(Y_2)},
\end{align*}
\begin{flushleft}
where we define $U_{\pi}(t) := p^4\Delta_{\overline{\pi}}\Delta_{p}(T(\overline{\pi}, \pi) - p^4\Delta_{p}t) \in H_{p}^{1,1}[t]$.
\end{flushleft}
\end{proposition}
\begin{proof}
We will first simplify $\Lambda_{-}(p)B(X_i)$. But $\Lambda_{-}(p)=\Lambda_{-}(\pi)\Lambda_{-}(\overline{\pi})=\Lambda_{-}(\overline{\pi})\Lambda_{-}(\pi)$, so from the relations of Table \ref{table:1} and equation \eqref{coeffs of K}, we have:
\begin{itemize}
    \item $\Lambda_{-}(p)B_1 = \Lambda_{-}(p)(T(\pi, \overline{\pi})+T(\overline{\pi}, \pi)) = p^2\left(\Delta_{\overline{\pi}}\Lambda_{-}(\pi)^2+\Delta_{\pi}\Lambda_{-}(\overline{\pi})^2\right)$.
    \item $\Lambda_{-}(p)B_2 = p^4\Delta_{p}\Lambda_{-}(p)$.
    \item $\Lambda_{-}(p)B_3 = \Lambda_{-}(p)B_4 = 0$.
\end{itemize}
Also, again from Table \ref{table:1}, we have that
\begin{itemize}
\item $\Lambda_{-}(\pi)\Lambda_{+}(\overline{\pi}) = p^3\Delta_p$.
\item $\Lambda_{-}(\overline{\pi})\Lambda_{+}(\overline{\pi}) = p\Delta_{\overline{\pi}}T(\overline{\pi},\pi)$.
\item $\Lambda_{-}(p)\Lambda_{+}(\overline{\pi}^2) = \Lambda_{-}(\overline{\pi})\Lambda_{-}(\pi)\Lambda_{+}(\overline{\pi})\Lambda_{+}(\overline{\pi}) = p^4\Delta_p\Delta_{\overline{\pi}}T(\overline{\pi},\pi)$.
\end{itemize}
Now, $T(\overline{\pi},\pi)$ and $\Delta_p$ commute with $\Lambda_{+}(\overline{\pi})$ and so for $l \geq 2$, we can write 
\begin{align*}
    \Lambda_{-}(p)B(X_i)\Lambda_{+}(\overline{\pi}^l) &= \Lambda_{+}(\overline{\pi}^{l-2})(p^4\Delta_p\Delta_{\overline{\pi}}T(\overline{\pi},\pi)-(p^4\Delta_{\overline{\pi}}\Delta_pT(\overline{\pi},\pi)^2 + p^8\Delta_{\overline{\pi}}\Delta_p^2)X_i+p^8\Delta_p^2\Delta_{\overline{\pi}}T(\overline{\pi},\pi)X_i^2)\\
    &= \Lambda_{+}(\overline{\pi}^{l-2})(1-T(\overline{\pi}, \pi)X_i)U_{\pi}(X_i).
\end{align*}
Hence,
\begin{align*}
   \sum_{l \geq 0}\tilde{\psi}_{p^{l-2}}\mid (\Lambda_{-}(p)X_i^2) \mid B(X_i)\Lambda_{+}(\overline{\pi}^{l}) Y_2^l&=\sum_{l \geq 2} \tilde{\psi}_{p^{l-2}} \mid \Lambda_{+}(\overline{\pi}^{l-2})Y_2^{l-2}(1-T(\overline{\pi},\pi)X_i)U_{\pi}(X_i)X_i^2Y_2^2= \\
   &=(1-\alpha_iX)\frac{\tilde{\psi}_1 \mid S_{\pi}(Y_2)U_{\pi}(X_i)X_i^2Y_2^2}{D_{\pi, G}^{(2)}(Y_2)},
\end{align*}
by Proposition \ref{prop:f2} and Lemma \ref{action}. Hence, the result follows.
\end{proof}
Finally, for the middle term, we have:
\begin{proposition}\label{prop: t-}
We have, for $i=1,2$
\begin{equation*}
-\frac{1}{1-\alpha_i X}\sum_{l\geq 0}\tilde{\psi}_{p^{l-1}}\mid T_{-}(p)X_i \mid B(X_i)\Lambda_{+}(\overline{\pi}^{l})Y_2^{l} = 
\end{equation*}
\begin{equation*}
    =-p^2\frac{\tilde{\psi}_1\mid S_{\pi}(Y_2)T(\pi)\Delta_{\overline{\pi}}X_iY_2}{D_{\pi,G}^{(2)}(Y_2)}+p^5\frac{\left[\tilde{\psi}_p-\tilde{\psi}_1 \mid \Lambda_{-}(\overline{\pi})Y_2\right]\mid S_{\pi}(Y_2)\Delta_{p}\Delta_{\overline{\pi}}T_{+}(p)Y_2^2X_i^2}{D_{\pi, G}^{(2)}(Y_2)} + p^{2k-4}\tilde{\psi}_{1}\mid T(\overline{\pi})Y_2X_i^2.
\end{equation*}
\end{proposition}
\begin{proof}
Firstly, we have no terms for $l=0$, so we consider $l \geq 1$. The idea is to pass $\Lambda_{+}(\overline{\pi}^m)$ to the left for some $m$, so that it acts on the Fourier-Jacobi coefficients, and then we will be able to apply the rationality Proposition \ref{prop:f2}. From equation \eqref{k_2lambda+}, we have $B_3\Lambda_{+}(\overline{\pi})=B_4\Lambda_{+}(\overline{\pi})=0$.
Now, using Table \ref{table:1}, we have $T_{-}(p)\Lambda_{+}(\overline{\pi}) = p^2\Delta_{\overline{\pi}}T(\pi)$ and that $T(\pi)$ commutes with $\Lambda_{+}(\overline{\pi})$. Therefore, we can compute:
\begin{equation*}
    \sum_{l\geq 0}\tilde{\psi}_{p^{l-1}}\mid T_{-}(p)X_i\mid\Lambda_{+}(\overline{\pi}^{l})Y_2^l = p^2\sum_{l\geq 1}\tilde{\psi}_{p^{l-1}}\mid\Lambda_{+}(\overline{\pi}^{l-1})Y_2^{l-1}T(\pi)\Delta_{\overline{\pi}}X_i Y_2 = p^2\frac{\tilde{\psi}_1 \mid S_{\pi}(Y_2)T(\pi)\Delta_{\overline{\pi}}Y_2X_i}{D_{\pi, G}^{(2)}(Y_2)}.
\end{equation*}
Let us now deal with $T_{-}(p)B_1\Lambda_{+}(\overline{\pi}^{l})$. We remind the reader that $B_1 = T(\pi, \overline{\pi})+T(\overline{\pi},\pi)$. We will deal with each part separately. By Table \ref{table:1}, we have $(l \geq 1)$
\begin{equation*}
    T_{-}(p)T(\overline{\pi}, \pi)\Lambda_{+}(\overline{\pi}^{l}) = T_{-}(p)\Lambda_{+}(\overline{\pi}^{l})T(\overline{\pi}, \pi) = p^2\Delta_{\overline{\pi}}\Lambda_{+}(\overline{\pi}^{l-1})T(\pi)T(\overline{\pi}, \pi).
\end{equation*}
For the other part, if $l \geq 2$,
\begin{equation*}
    T_{-}(p)T(\pi,\overline{\pi})\Lambda_{+}(\overline{\pi}^{l})= p^2\Delta_{\overline{\pi}}T_{-}(p)\Lambda_{+}(\pi)\Lambda_{+}(\overline{\pi}^{l-1}) = p^4\Delta_{p}T(\overline{\pi})\Lambda_{+}(\overline{\pi}^{l-1}) = p^5\Delta_{\overline{\pi}}\Delta_{p}\Lambda_{+}(\overline{\pi}^{l-2})T_{+}(p).
\end{equation*}
For $l=1$, we have
\begin{equation*}
    T_{-}(p)T(\pi,\overline{\pi})\Lambda_{+}(\overline{\pi})= p^2T_{-}(p)\Lambda_{+}(\overline{\pi}) = p^4\Delta_{p}T(\overline{\pi}).
\end{equation*}
Finally, we will deal with the term $T_{-}(p)B_2\Lambda_{+}(\overline{\pi}^{l})$. 
Using equation \eqref{k_2lambda+} and relations of Table \ref{table:1}, we have for $l \geq 2$,
\begin{align*}
    T_{-}(p)B_2\Lambda_{+}(\overline{\pi}^l) = p^2\Delta_{\overline{\pi}}T_{-}(p)\Lambda_{+}(\pi)T(\overline{\pi},\pi)\Lambda_{+}(\overline{\pi}^{l-1})&=p^{4}\Delta_pT(\overline{\pi})\Lambda_{+}(\overline{\pi}^{l-1})T(\overline{\pi},\pi)\\
    &= p^5\Delta_{\overline{\pi}}\Delta_p\Lambda_{+}(\overline{\pi}^{l-2})T_{+}(p)T(\overline{\pi},\pi).
\end{align*}
Finally, for $l=1$, we get
\begin{equation*}
    T_{-}(p)B_2\Lambda_{+}(\overline{\pi})=  p^4\Delta_p T(\overline{\pi})T(\overline{\pi},\pi).
\end{equation*}
Applying now Proposition \ref{prop:f2} and using Lemma \ref{action} as well, we obtain the stated result.
\end{proof}
\subsection{Calculation of the Dirichlet Series - Third Part}
We will now deal with the Dirichlet series
\begin{multline*}
    D_{(\epsilon_1, \epsilon_2)}(s) = \sum_{l_1,l_2,m \geq 0}\langle \tilde{\phi}_1\mid T_{-}(p^{m})\Lambda_{-}(\pi^{l_2})\Lambda_{-}(\overline{\pi}^{l_1}), \tilde{\psi}_{p^{m+l_1+l_2}}\rangle_{\mathcal{A}}a_{p^{m}} p^{(4-2k)l_2} p^{(4-2k)l_1}\times \\ \times \pi^{l_2 k} \overline{\pi}^{l_1k} p^{-s(l_1+l_2)}p^{-(2k+s-4)m} :=
\end{multline*}
\begin{equation*}
    := (\alpha_1 V_1(s) - \alpha_2 V_2(s))/(\alpha_1-\alpha_2),
\end{equation*}
where
\begin{equation*}
    V_i(s) := \sum_{l_1,l_2,m \geq 0}\langle \tilde{\phi}_1, \tilde{\psi}_{p^{m+l_1+l_2}}\mid T_{+}(p^{m})\Lambda_{+}(\pi^{l_1})\Lambda_{+}(\overline{\pi}^{l_2})\rangle_{\mathcal{A}}X_i^mY_2^{l_1}Y_1^{l_2}.
\end{equation*}
Here, we again set $X_i = \alpha_i p^{-(2k+s-4)}$, $Y_1 = \pi^{k}p^{-(2k+s-4)}$, $Y_2 = \overline{\pi}^{k}p^{-(2k+s-4)}$ and we keep in mind that the operators $T_{+}(p), \Lambda_{+}(\pi), \Lambda_{+}(\overline{\pi})$ all commute with each other. This follows from the fact that $j_{+}$ of equation \eqref{+- embeddings} is a ring homomorphism and $H(\Gamma_{1}, S_{p}^{1})$ is commutative.
\begin{lemma}
For $i=1,2$, we have 
\begin{equation*}
    \sum_{l_1,l_2, m \geq 0} \tilde{\psi}_{p^{m+l_1+l_2}}\mid T_{+}(p^{m})\Lambda_{+}(\pi^{l_1})\Lambda_{+}(\overline{\pi}^{l_2})X_i^mY_1^{l_1}Y_2^{l_2} =
\end{equation*}
\begin{equation*}
    =Q_{p,G}^{(2)}(X_i)^{-1}\sum_{l_1,l_2\geq 0}\left[\tilde{\psi}_{p^{l_1+l_2}} - \tilde{\psi}_{p^{l_1+l_2-1}}\mid T_{-}(p)X_i + p\tilde{\psi}_{p^{l_1+l_2-2}}\mid \Lambda_{-}(p)X_i^2\right]\mid B(X_i)\Lambda_{+}(\overline{\pi}^{l_2})\Lambda_{+}(\pi^{l_1})Y_1^{l_1}Y_2^{l_2}.
\end{equation*}
\end{lemma}
\begin{proof}
The proof follows immediately by Proposition \ref{prop:f1}.
\end{proof}
We will now deal with each sum occurring above.
\begin{proposition}\label{lem:sum1}
For $i=1, 2$, we have
\begin{equation*}
    \sum_{l_1,l_2\geq 0}\tilde{\psi}_{p^{l_1+l_2}}\mid B(X_i)\Lambda_{+}(\overline{\pi}^{l_2})\Lambda_{+}(\pi^{l_1})Y_1^{l_1}Y_2^{l_2} = 
\end{equation*}
\begin{equation*}
    =(1-p^{2k-5}Y_1Y_2)(1 - p^2Y_2Y_1^{-1}\lambda_{\overline{\pi}}X_i)(1-p^2Y_1Y_2^{-1}\lambda_{\pi}X_i)\frac{\tilde{\psi}_1 \mid S_{\overline{\pi}}(Y_1)S_{\pi}(Y_2)}{D_{\overline{\pi}, G}^{(2)}(Y_1)D_{\pi, G}^{(2)}(Y_2)}
\end{equation*}
\begin{equation*}
    - \left[(p^{k-3}-p^2Y_1Y_2^{-1}\lambda_{\pi})X_i+p^{2k-4}X_i^2\right]\frac{\tilde{\psi}_1 \mid S_{\overline{\pi}}(Y_1)}{D_{\overline{\pi},G}^{(2)}(Y_1)} - \left[(p^{k-3}-p^2Y_2Y_1^{-1}\lambda_{\overline{\pi}})X_i+p^{2k-4}X_i^2\right]\frac{\tilde{\psi}_1 \mid S_{\pi}(Y_2)}{D_{\pi,G}^{(2)}(Y_2)}+
\end{equation*}
\begin{equation*}
    +p^2\frac{\left[\tilde{\psi}_p-\tilde{\psi}_1\mid\Lambda_{-}(\overline{\pi})Y_2\right]\mid S_{\pi}(Y_2)\Lambda_{+}(\pi)\Delta_{\overline{\pi}}T(\overline{\pi},\pi)Y_2X_i^2}{D_{\pi, G}^{(2)}(Y_2)} +
\end{equation*}
\begin{equation*}
    +p^2\frac{\left[\tilde{\psi}_p-\tilde{\psi}_1\mid\Lambda_{-}(\pi)Y_1\right]\mid S_{\overline{\pi}}(Y_1)\Lambda_{+}(\overline{\pi})\Delta_{\pi}T(\pi,\overline{\pi})Y_1X_i^2}{D_{\overline{\pi}, G}^{(2)}(Y_1)}+
\end{equation*}
\begin{equation*}
    + \left[(1-\alpha_iX)^2(1-p^{2k-4}X_i^2)+2\alpha_iX-(1-p^{2k-4}X_i^2)\right]\tilde{\psi}_1,
\end{equation*}
where $\lambda_{\pi}, \lambda_{\overline{\pi}}$ are the eigenvalues of $\Delta_{\pi}, \Delta_{\overline{\pi}}$ respectively.
\end{proposition}
\begin{proof}
Firstly, using Proposition \ref{prop:f2} and commutativity relations of Table \ref{table:1}, we have
\begin{align*}
    \sum_{l_1,l_2\geq 0} \tilde{\psi}_{p^{l_1+l_2}}\mid \Lambda_{+}(\overline{\pi}^{l_2})\Lambda_{+}(\pi^{l_1})Y_1^{l_1}Y_2^{l_2} &= \sum_{l_1\geq 0}\frac{\left[\tilde{\psi}_{p^{l_1}}-\tilde{\psi}_{p^{l_1-1}}\mid \Lambda_{-}(\overline{\pi})Y_2\right]\mid S_{\pi}(Y_2)\Lambda_{+}(\pi^{l_1})Y_1^{l_1}}{D_{\pi,G}^{(2)}(Y_2)}\\
    &= \frac{(1-p^{2k-5}Y_1Y_2)\tilde{\psi}_1\mid S_{\overline{\pi}}(Y_1)S_{\pi}(Y_2)}{D_{\overline{\pi},G}^{(2)}(Y_1)D_{\pi, G}^{(2)}(Y_2)}.
\end{align*}
Now, from equation \eqref{coeffs of K}, $B_1=T(\pi,\overline{\pi})+T(\overline{\pi},\pi)$. Hence, from Proposition \ref{prop:f2} and Table \ref{table:1}:
\begin{equation*}
    \sum_{l_1,l_2\geq 0} \tilde{\psi}_{p^{l_1+l_2}}\mid T(\pi, \overline{\pi})\Lambda_{+}(\overline{\pi}^{l_2})\Lambda_{+}(\pi^{l_1})Y_2^{l_2}Y_1^{l_1} = \sum_{l_1,l_2\geq 0} \tilde{\psi}_{p^{l_1+l_2}}\mid \Lambda_{+}(\pi^{l_1}) T(\pi, \overline{\pi})\Lambda_{+}(\overline{\pi}^{l_2})Y_2^{l_2}Y_1^{l_1} = 
\end{equation*}
\begin{equation*}
    = \sum_{l_1 \geq 0} \tilde{\psi}_{p^{l_1}} \mid \Lambda_{+}(\pi^{l_1})T(\pi, \overline{\pi})Y_1^{l_1} + p^2\sum_{l_1 \geq 0, l_2 \geq 1} \tilde{\psi}_{p^{l_1+l_2}}\mid \Lambda_{+}(\pi^{l_1+1})\Lambda_{+}(\overline{\pi}^{l_2-1})\Delta_{\overline{\pi}}Y_1^{l_1}Y_2^{l_2}
\end{equation*}
\begin{equation*}
    =\frac{\tilde{\psi}_{1}\mid S_{\overline{\pi}}(Y_1)T(\pi, \overline{\pi})}{D_{\overline{\pi}, G}^{(2)}(Y_1)}+p^2Y_2Y_1^{-1}\left((1-p^{2k-5}Y_1Y_2)\frac{\tilde{\psi}_1\mid S_{\overline{\pi}}(Y_1)S_{\pi}(Y_2)\Delta_{\overline{\pi}}}{D_{\overline{\pi},G}^{(2)}(Y_1)D_{\pi, G}^{(2)}(Y_2)} - \frac{\tilde{\psi}_1\mid S_{\pi}(Y_2)\Delta_{\overline{\pi}}}{D_{\pi, G}^{(2)}(Y_2)}\right),
\end{equation*}
and we get an analogous result for
\begin{equation*}
    \sum_{l_1,l_2\geq 0} \tilde{\psi}_{p^{l_1+l_2}}\mid T(\overline{\pi}, \pi)\Lambda_{+}(\overline{\pi}^{l_2})\Lambda_{+}(\pi^{l_1})Y_2^{l_2}Y_1^{l_1}.
\end{equation*}
Next, using Table \ref{table:1}, we observe that for $l_1, l_2 \geq 1$ we have 
\begin{align}\label{k_2}
    B_2\Lambda_{+}(\overline{\pi}^{l_2})\Lambda_{+}(\pi^{l_1}) &= p\Lambda_{+}(\pi)\Lambda_{-}(\overline{\pi})\Lambda_{+}(\overline{\pi}^{l_2})\Lambda_{+}(\pi^{l_1}) = p\Lambda_{+}(\pi)\Lambda_{-}(\overline{\pi})\Lambda_{+}(\pi)\Lambda_{+}(\pi^{l_1-1})\Lambda_{+}(\overline{\pi}^{l_2}) \nonumber\\
    &=p^4\Lambda_{+}(\pi)\Delta_{p}\Lambda_{+}(\pi^{l_1-1})\Lambda_{+}(\overline{\pi}^{l_2}) = p^4\Lambda_{+}(\pi^{l_1})\Lambda_{+}(\overline{\pi}^{l_2})\Delta_{p}.
\end{align}
Hence,
\begin{equation*}
    \sum_{l_1,l_2 \geq 0}\tilde{\psi}_{p^{l_1+l_2}}\mid B_2\Lambda_{+}(\pi^{l_1})\Lambda_{+}(\overline{\pi}^{l_2})Y_1^{l_1}Y_2^{l_2} =
\end{equation*}
\begin{equation*}
    \sum_{l_1 \geq 0} \tilde{\psi}_{p^{l_1}}\mid B_2 \Lambda_{+}(\pi^{l_1})Y_1^{l_1} + \sum_{l_2 \geq 0} \tilde{\psi}_{p^{l_2}}\mid B_2 \Lambda_{+}(\pi^{l_2})Y_2^{l_2} + \sum_{l_1,l_2\geq 1}\tilde{\psi}_{p^{l_1+l_2}}\mid B_2\Lambda_{+}(\overline{\pi}^{l_2})\Lambda_{+}(\pi^{l_1})Y_1^{l_1}Y_2^{l_2} -\textup{ } \tilde{\psi}_1\mid B_2 =  
\end{equation*}
\begin{equation*}
    = \tilde{\psi}_1\mid B_2 + p^2\frac{\left[\tilde{\psi}_p-\tilde{\psi}_1\mid\Lambda_{-}(\overline{\pi})\right]\mid S_{\pi}(Y_2)\Lambda_{+}(\pi)\Delta_{\overline{\pi}}T(\overline{\pi},\pi)Y_2}{D_{\pi, G}^{(2)}(Y_2)} + 
\end{equation*}
\begin{equation*}
+p^2\frac{\left[\tilde{\psi}_p-\tilde{\psi}_1\mid\Lambda_{-}(\pi)\right]\mid S_{\overline{\pi}}(Y_1)\Lambda_{+}(\overline{\pi})\Delta_{\pi}T(\pi,\overline{\pi})Y_1}{D_{\overline{\pi}, G}^{(2)}(Y_1)}+
\end{equation*}
\begin{equation*}
    +p^{2k-4}\left(\frac{(1-p^{2k-5}Y_1Y_2)\tilde{\psi}_1\mid S_{\overline{\pi}}(Y_1)S_{\pi}(Y_2)}{D_{\overline{\pi},G}^{(2)}(Y_1)D_{\pi, G}^{(2)}(Y_2)} - \frac{\tilde{\psi}_1\mid S_{\pi}(Y_2)}{D_{\pi, G}^{(2)}(Y_2)}-\frac{\tilde{\psi}_1\mid S_{\overline{\pi}}(Y_1)}{D_{\overline{\pi}, G}^{(2)}(Y_1)}+ \tilde{\psi}_1\right),
\end{equation*}
as the sum
\begin{equation*}
    \sum_{l_1,l_2\geq 1} \tilde{\psi}_{p^{l_1+l_2}}\mid \Lambda_{+}(\pi^{l_1})\Lambda_{+}(\overline{\pi}^{l_2})p^4\Delta_{p}Y_1^{l_1}Y_2^{l_2}
\end{equation*}
can be computed to be
\begin{equation*}
    p^{2k-4}\left(\frac{(1-p^{2k-5}Y_1Y_2)\tilde{\psi}_1\mid S_{\overline{\pi}}(Y_1)S_{\pi}(Y_2)}{D_{\overline{\pi},G}^{(2)}(Y_1)D_{\pi, G}^{(2)}(Y_2)} - \frac{\tilde{\psi}_1\mid S_{\pi}(Y_2)}{D_{\pi, G}^{(2)}(Y_2)}-\frac{\tilde{\psi}_1\mid S_{\overline{\pi}}(Y_1)}{D_{\overline{\pi}, G}^{(2)}(Y_1)}+ \tilde{\psi}_1\right).
\end{equation*}
Finally,
\begin{equation*}
    \sum_{l_1,l_2\geq 0} \tilde{\psi}_{p^{l_1+l_2}}\mid B_3\Lambda_{+}(\overline{\pi}^{l_2})\Lambda_{+}(\pi^{l_1})Y_1^{l_1}Y_2^{l_2} = \tilde{\psi}_{1}\mid B_3,
\end{equation*}
and
\begin{equation*}
    \sum_{l_1,l_2\geq 0} \tilde{\psi}_{p^{l_1+l_2}}\mid B_4\Lambda_{+}(\overline{\pi}^{l_2})\Lambda_{+}(\pi^{l_1})Y_1^{l_1}Y_2^{l_2} = \tilde{\psi}_{1}\mid B_4,
\end{equation*}
as $B_3\Lambda_{+}(\overline{\pi}) = B_4\Lambda_{+}(\overline{\pi}) = B_3\Lambda_{+}(\pi) = B_4\Lambda_{+}(\pi) = 0$ from equation \eqref{k_2lambda+}.
\end{proof}
\begin{proposition}
For $i=1,2$, we have
\begin{equation*}
    \sum_{l_1,l_2 \geq 0}\tilde{\psi}_{p^{l_1+l_2-2}} \mid \Lambda_{-}(p)B(X_i)\Lambda_{+}(\pi^{l_1})\Lambda_{+}(\overline{\pi}^{l_2})Y_1^{l_1}Y_2^{l_2}X_i^2 = 
\end{equation*}
\begin{equation*}
    = (1-\alpha_iX)\left[\frac{\tilde{\psi}_1 \mid S_{\pi}(Y_2)U_{\pi}(X_i)X_i^2Y_2^2}{D_{\pi, G}^{(2)}(Y_2)} + \frac{\tilde{\psi}_1 \mid S_{\overline{\pi}}(Y_1)U_{\overline{\pi}}(X_i)X_i^2Y_1^2}{D_{\overline{\pi}, G}^{(2)}(Y_1)}\right]+
\end{equation*}
\begin{equation*}
    +p^{4k-10}Y_1Y_2(1-p^{2k-5}Y_1Y_2)(1-p^2Y_2Y_1^{-1}\lambda_{\overline{\pi}}X_i)(1-p^2Y_1Y_2^{-1}\lambda_{\pi}X_i)X_i^2\frac{\tilde{\psi}_1 \mid S_{\overline{\pi}}(Y_1)S_{\pi}(Y_2)}{D_{\overline{\pi}, G}^{(2)}(Y_1)D_{\pi, G}^{(2)}(Y_2)}
\end{equation*}
\begin{equation*}
    -p^{4k-10}X_i^3Y_1Y_2(p^{k-3}- p^2Y_1Y_2^{-1}\lambda_{\pi})\frac{\tilde{\psi}_{1}\mid S_{\overline{\pi}}(Y_1)}{D_{\overline{\pi},G}^{(2)}(Y_1)}-p^{4k-10}X_i^3Y_1Y_2(p^{k-3}- p^2Y_2Y_1^{-1}\lambda_{\overline{\pi}})\frac{\tilde{\psi}_{1}\mid S_{\pi}(Y_2)}{D_{\pi,G}^{(2)}(Y_2)},
\end{equation*}
where again $U_{\pi}(t) := p^4\Delta_{\overline{\pi}}\Delta_p(T(\overline{\pi}, \pi)-p^4\Delta_p t) \in H_{p}^{1,1}[t]$, as in Proposition \ref{lem: Lambda-}.
\end{proposition}
\begin{proof}
For the proof, we rewrite the sum as follows:
\begin{equation*}
    \sum_{l_1, l_2 \geq 0} = \sum_{l_1 = 0, \textup{ }l_2 \geq 2} + \sum_{l_2 =0, \textup{ }l_1 \geq 2} + \sum_{l_1,l_2 \geq 1}.
\end{equation*}
We know how to compute the first two sums by Proposition \ref{lem: Lambda-}, so will now deal with the last one. We rewrite this as
\begin{equation*}
    \sum_{l_1,l_2 \geq 1}\tilde{\psi}_{p^{l_1+l_2-2}} \mid \Lambda_{-}(p)B(X_i)\Lambda_{+}(p)\Lambda_{+}(\pi^{l_1-1})\Lambda_{+}(\overline{\pi}^{l_2-1})Y_1^{l_1}Y_2^{l_2}X_i^2.
\end{equation*}
But
\begin{equation*}
    \Lambda_{-}(p)B(X_i)\Lambda_{+}(p) = p^6\Delta_p^{2}(1 - B_1 X_i + p^4\Delta_p X_i^2),
\end{equation*}
as we can obtain by Table \ref{table:1} or the relations written in \cite[p. 2881-2882]{gritsenko3}. Now, using Proposition \ref{prop:f2}, we get
\begin{equation*}
    \sum_{l_1,l_2 \geq 1}\tilde{\psi}_{p^{l_1+l_2-2}} \mid \Lambda_{+}(\pi^{l_1-1})\Lambda_{+}(\overline{\pi}^{l_2-1})Y_1^{l_1}Y_2^{l_2}X_i^2 = Y_1Y_2 \sum_{l_1,l_2 \geq 0}\tilde{\psi}_{p^{l_1+l_2}} \mid \Lambda_{+}(\pi^{l_1})\Lambda_{+}(\overline{\pi}^{l_2})Y_1^{l_1}Y_2^{l_2}X_i^2 = 
\end{equation*}
\begin{equation*}
    = (1-p^{2k-5}Y_1Y_2)Y_1Y_2X_i^2\frac{\tilde{\psi}_1\mid S_{\overline{\pi}}(Y_1)S_{\pi}(Y_2)}{D_{\overline{\pi},G}^{(2)}(Y_1)D_{\pi, G}^{(2)}(Y_2)}.
\end{equation*}
Also, $B_1 = T(\pi, \overline{\pi})+T(\overline{\pi}, \pi)$ and we have
\begin{equation*}
    \sum_{l_1,l_2 \geq 1} \tilde{\psi}_{p^{l_1+l_2-2}}\mid T(\pi, \overline{\pi})\Lambda_{+}(\pi^{l_1-1})\Lambda_{+}(\overline{\pi}^{l_2-1})Y_1^{l_1}Y_2^{l_2} = \sum_{l_1,l_2 \geq 1} \tilde{\psi}_{p^{l_1+l_2-2}}\mid \Lambda_{+}(\pi^{l_1-1})T(\pi, \overline{\pi})\Lambda_{+}(\overline{\pi}^{l_2-1})Y_1^{l_1}Y_2^{l_2} = 
\end{equation*}
\begin{equation*}
    =Y_1Y_2\sum_{l_1,l_2\geq 0} \tilde{\psi}_{p^{l_1+l_2}} \mid \Lambda_{+}(\pi^{l_1})T(\pi, \overline{\pi})\Lambda_{+}(\overline{\pi}^{l_2})Y_1^{l_1}Y_2^{l_2}=
\end{equation*}
\begin{equation*}
    = Y_1Y_2 \sum_{l_1 \geq 0} \tilde{\psi}_{p^{l_1}} \mid \Lambda_{+}(\pi^{l_1})T(\pi, \overline{\pi})Y_1^{l_1} + p^2Y_1Y_2 \sum_{l_1 \geq 0, l_2 \geq 1} \tilde{\psi}_{p^{l_1+l_2}}\mid \Lambda_{+}(\pi^{l_1+1})\Lambda_{+}(\overline{\pi}^{l_2-1})\Delta_{\overline{\pi}}Y_1^{l_1}Y_2^{l_2}
\end{equation*}
\begin{equation*}
    =Y_1Y_2\frac{\tilde{\psi}_{1}\mid S_{\overline{\pi}}(Y_1)T(\pi, \overline{\pi})}{D_{\overline{\pi}, G}^{(2)}(Y_1)}+p^2Y_2^2\sum_{l_1\geq 0}\frac{\left[\tilde{\psi}_{p^{l_1+1}}-\tilde{\psi}_{p^{l_1}}\mid \Lambda_{-}(\overline{\pi})Y_2\right]\mid S_{\pi}(Y_2)\Lambda_{+}(\pi^{l_1+1})\Delta_{\overline{\pi}}Y_1^{l_1+1}}{D_{\pi, G}^{(2)}(Y_2)}
\end{equation*}
\begin{equation*}
    =Y_1Y_2\frac{\tilde{\psi}_{1}\mid S_{\overline{\pi}}(Y_1)T(\pi, \overline{\pi})}{D_{\overline{\pi}, G}^{(2)}(Y_1)}+p^2Y_2^2\left((1-p^{2k-5}Y_1Y_2)\frac{\tilde{\psi}_1\mid S_{\overline{\pi}}(Y_1)S_{\pi}(Y_2)\Delta_{\overline{\pi}}}{D_{\overline{\pi},G}^{(2)}(Y_1)D_{\pi, G}^{(2)}(Y_2)} - \frac{\tilde{\psi}_1\mid S_{\pi}(Y_2)\Delta_{\overline{\pi}}}{D_{\pi, G}^{(2)}(Y_2)}\right),
\end{equation*}
We obtain a similar expression for $T(\overline{\pi}, \pi)$ and then the result follows.
\end{proof}
\begin{proposition}
For $i=1, 2$, we have
\begin{equation*}
    \sum_{l_1,l_2 \geq 0}\tilde{\psi}_{p^{l_1+l_2-1}} \mid T_{-}(p)B(X_i)\Lambda_{+}(\pi^{l_1})\Lambda_{+}(\overline{\pi}^{l_2})Y_1^{l_1}Y_2^{l_2} = 
\end{equation*}
\begin{equation*}
(1-\alpha_iX)\textup{ }\times
\end{equation*}
\begin{equation*}
    \times\left[p^2\frac{\tilde{\psi}_1\mid S_{\pi}(Y_2)T(\pi)\Delta_{\overline{\pi}}X_iY_2}{D_{\pi,G}^{(2)}(Y_2)}-p^5\frac{\left[\tilde{\psi}_p-\tilde{\psi}_1 \mid \Lambda_{-}(\overline{\pi})Y_2\right]\mid S_{\pi}(Y_2)\Delta_{p}\Delta_{\overline{\pi}}T_{+}(p)Y_2^2X_i^2}{D_{\pi, G}^{(2)}(Y_2)}- p^{2k-4}\tilde{\psi}_{1}\mid T(\overline{\pi})Y_2X_i^2\right.+
\end{equation*}
\begin{equation*}
    \left.+p^2\frac{\tilde{\psi}_1\mid S_{\overline{\pi}}(Y_1)T(\overline{\pi})\Delta_{\pi}X_iY_1}{D_{\overline{\pi},G}^{(2)}(Y_1)}-p^5\frac{\left[\tilde{\psi}_p-\tilde{\psi}_1 \mid \Lambda_{-}(\pi)Y_1\right]\mid S_{\overline{\pi}}(Y_1)\Delta_{p}\Delta_{\pi}T_{+}(p)Y_1^2X_i^2}{D_{\overline{\pi}, G}^{(2)}(Y_1)} - p^{2k-4}\tilde{\psi}_{1}\mid T(\pi)Y_1X_i^2\right]+
\end{equation*}
\begin{equation*}
     +\frac{1}{2}p^{2k-5}(1+p^{2k-4}X_i^2)X_i \textup{ }\times
\end{equation*}
\begin{equation*}
\times\left[(1-p^{2k-5}Y_1Y_2)\frac{\left[\tilde{\psi}_p-\tilde{\psi}_1 \mid \Lambda_{-}(\pi)Y_1\right]\mid S_{\overline{\pi}}(Y_1)S_{\pi}(Y_2)T_{+}(p)Y_1Y_2}{D_{\overline{\pi},G}^{(2)}(Y_1)D_{\pi, G}^{(2)}(Y_2)} - \frac{\tilde{\psi}_1\mid \Lambda_{-}(\overline{\pi})S_{\pi}(Y_2)T_{+}(p)Y_1Y_2^2}{D_{\pi, G}^{(2)}(Y_2)}+\right.
\end{equation*}
\begin{equation*}
+\left.(1-p^{2k-5}Y_1Y_2)\frac{\left[\tilde{\psi}_p-\tilde{\psi}_1 \mid \Lambda_{-}(\overline{\pi})Y_2\right]\mid S_{\pi}(Y_2)S_{\overline{\pi}}(Y_1)T_{+}(p)Y_1Y_2}{D_{\pi,G}^{(2)}(Y_2)D_{\overline{\pi}, G}^{(2)}(Y_1)} - \frac{\tilde{\psi}_1\mid \Lambda_{-}(\pi)S_{\overline{\pi}}(Y_1)T_{+}(p)Y_2Y_1^2}{D_{\overline{\pi}, G}^{(2)}(Y_1)}\right]-
\end{equation*}
\begin{equation*}
    -X_i^2\left[p^5(1-p^{2k-5}Y_1Y_2)\frac{\left[\tilde{\psi}_p-\tilde{\psi}_1 \mid \Lambda_{-}(\pi)Y_1\right] \mid S_{\overline{\pi}}(Y_1)S_{\pi}(Y_2)\Delta_{\overline{\pi}}\Delta_{p}T_{+}(p)Y_2^2}{D_{\overline{\pi}, G}^{(2)}(Y_1)D_{\pi, G}^{(2)}(Y_2)} - p^5\frac{\tilde{\psi}_p \mid S_{\pi}(Y_2)T_{+}(p)\Delta_{\overline{\pi}}\Delta_pY_2^2}{D_{\pi, G}^{(2)}(Y_2)}\right.
\end{equation*}
\begin{equation*}
    \left.+\textup{ } p^{2k-4}Y_2\left(\frac{\tilde{\psi}_1 \mid S_{\overline{\pi}}(Y_1)T(\overline{\pi})}{D_{\overline{\pi}, G}^{2}(Y_1)} - \tilde{\psi}_1 \mid T(\overline{\pi})\right)\right.+
\end{equation*}
\begin{equation*}
    +p^5(1-p^{2k-5}Y_1Y_2)\frac{\left[\tilde{\psi}_p-\tilde{\psi}_1 \mid \Lambda_{-}(\overline{\pi})Y_2\right] \mid S_{\pi}(Y_2)S_{\overline{\pi}}(Y_1)\Delta_{\pi}\Delta_{p}T_{+}(p)Y_1^2}{D_{\pi, G}^{(2)}(Y_2)D_{\overline{\pi}, G}^{(2)}(Y_1)} - p^5\frac{\tilde{\psi}_p \mid S_{\overline{\pi}}(Y_1)T_{+}(p)\Delta_{\pi}\Delta_pY_1^2}{D_{\overline{\pi}, G}^{(2)}(Y_1)}
\end{equation*}
\begin{equation*}
    \left.+ \textup{ }p^{2k-4}Y_1\left(\frac{\tilde{\psi}_1 \mid S_{\pi}(Y_2)T(\pi)}{D_{\pi, G}^{2}(Y_2)} - \tilde{\psi}_1 \mid T(\pi)\right)\right].
\end{equation*}
\end{proposition}
\begin{proof}
If $l_1=0$ or $l_2=0$, then we know how to compute this by Proposition \ref{prop: t-}. So, assume $l_1, l_2 \geq 1$. Now,
\begin{equation*}
    T_{-}(p)\Lambda_{+}(\overline{\pi}^{l_2})\Lambda_{+}(\pi^{l_1}) = p^3\Lambda_{+}(\overline{\pi}^{l_2-1})\Lambda_{+}(\pi^{l_1-1})T_{+}(p)\Delta_p,
\end{equation*}
using that $T_{-}(p)\Lambda_{+}(p) = p^3\Delta_p T_{+}(p)$. Hence, we have 
\begin{equation*}
    \sum_{l_1,l_2\geq 1} \tilde{\psi}_{p^{l_1+l_2-1}}\mid T_{-}(p)\Lambda_{+}(\pi^{l_1})\Lambda_{+}(\overline{\pi}^{l_2})Y_1^{l_1}Y_2^{l_2} =p^3\sum_{l_1,l_2\geq 1} \tilde{\psi}_{p^{l_1+l_2-1}} \mid \Lambda_{+}(\pi^{l_1-1})\Lambda_{+}(\overline{\pi}^{l_2-1})T_{+}(p)\Delta_pY_1^{l_1}Y_2^{l_2} = 
\end{equation*}
\begin{equation*}
    = p^3\sum_{l_1 \geq 1} \frac{\left[\tilde{\psi}_{p^{l_1}}-\tilde{\psi}_{p^{l_1-1}}\mid \Lambda_{-}(\overline{\pi})Y_2\right] \mid S_{\pi}(Y_2)\Lambda_{+}(\pi^{l_1-1})T_{+}(p)\Delta_pY_1^{l_1}Y_2}{D_{\pi, G}^{(2)}(Y_2)}
\end{equation*}
\begin{equation*}
    = p^3\frac{\left[\tilde{\psi}_p-\tilde{\psi}_1 \mid \Lambda_{-}(\pi)Y_1\right]\mid S_{\overline{\pi}}(Y_1)S_{\pi}(Y_2)T_{+}(p)\Delta_pY_1Y_2}{D_{\overline{\pi},G}^{(2)}(Y_1)D_{\pi, G}^{(2)}(Y_2)} - p^3\frac{\tilde{\psi}_1\mid \Lambda_{-}(\overline{\pi})S_{\pi}(Y_2)T_{+}(p)\Delta_pY_1Y_2^2}{D_{\pi, G}^{(2)}(Y_2)}
\end{equation*}
\begin{equation*}
    - p^6\frac{\left[\tilde{\psi}_p-\tilde{\psi}_1 \mid \Lambda_{-}(\pi)Y_1\right]\mid S_{\overline{\pi}}(Y_1)S_{\pi}(Y_2)T_{+}(p)\Delta_p^2Y_1^2Y_2^2}{D_{\overline{\pi},G}^{(2)}(Y_1)D_{\pi, G}^{(2)}(Y_2)}=
\end{equation*}
\begin{equation*}
    = p^3(1-p^{2k-5}Y_1Y_2)\frac{\left[\tilde{\psi}_p-\tilde{\psi}_1 \mid \Lambda_{-}(\pi)Y_1\right]\mid S_{\overline{\pi}}(Y_1)S_{\pi}(Y_2)T_{+}(p)\Delta_pY_1Y_2}{D_{\overline{\pi},G}^{(2)}(Y_1)D_{\pi, G}^{(2)}(Y_2)} - p^3\frac{\tilde{\psi}_1\mid \Lambda_{-}(\overline{\pi})S_{\pi}(Y_2)T_{+}(p)\Delta_pY_1Y_2^2}{D_{\pi, G}^{(2)}(Y_2)}.
\end{equation*}\\
We note here that the last expression is not (visibly) symmetric when we interchange $\pi \longleftrightarrow \overline{\pi}$. In order to make it symmetric, we compute it by first calculating the series involving the operator $\Lambda_{-}(\pi)$ first and hence we can write
\begin{equation*}
    \sum_{l_i\geq 1} \tilde{\psi}_{p^{l_1+l_2-1}}\mid T_{-}(p)\Lambda_{+}(\pi^{l_1})\Lambda_{+}(\overline{\pi}^{l_2})Y_1^{l_1}Y_2^{l_2} =
\end{equation*}
\begin{align*}
     =\frac{1}{2}p^{2k-5}&\left[(1-p^{2k-5}Y_1Y_2)\frac{\left[\tilde{\psi}_p-\tilde{\psi}_1 \mid \Lambda_{-}(\pi)Y_1\right]\mid S_{\overline{\pi}}(Y_1)S_{\pi}(Y_2)T_{+}(p)Y_1Y_2}{D_{\overline{\pi},G}^{(2)}(Y_1)D_{\pi, G}^{(2)}(Y_2)} - \frac{\tilde{\psi}_1\mid \Lambda_{-}(\overline{\pi})S_{\pi}(Y_2)T_{+}(p)Y_1Y_2^2}{D_{\pi, G}^{(2)}(Y_2)}+\right. \\
     &+\left.(1-p^{2k-5}Y_1Y_2)\frac{\left[\tilde{\psi}_p-\tilde{\psi}_1 \mid \Lambda_{-}(\overline{\pi})Y_2\right]\mid S_{\pi}(Y_2)S_{\overline{\pi}}(Y_1)T_{+}(p)Y_1Y_2}{D_{\pi,G}^{(2)}(Y_2)D_{\overline{\pi}, G}^{(2)}(Y_1)} - \frac{\tilde{\psi}_1\mid \Lambda_{-}(\pi)S_{\overline{\pi}}(Y_1)T_{+}(p)Y_2Y_1^2}{D_{\overline{\pi}, G}^{(2)}(Y_1)}\right].
\end{align*}
Moreover, as in equation \eqref{k_2}, we have that
\begin{equation*}
    B_2\Lambda_{+}(\overline{\pi}^{l_2})\Lambda_{+}(\pi^{l_1})= p^4\Delta_p \Lambda_{+}(\pi^{l_1})\Lambda_{+}(\overline{\pi}^{l_2}),
\end{equation*}
and so
\begin{equation*}
 T_{-}(p)B_2\Lambda_{+}(\overline{\pi}^{l_2})\Lambda_{+}(\pi^{l_1}) =p^7\Lambda_{+}(\overline{\pi}^{l_2-1})\Lambda_{+}(\pi^{l_1-1})T_{+}(p)\Delta_p^{2}.
\end{equation*}
Finally, for the last one, we note $B_1 = T(\pi, \overline{\pi})+T(\overline{\pi}, \pi)$. Now
\begin{equation*}
    \sum_{l_1,l_2 \geq 1} \tilde{\psi}_{p^{l_1+l_2-1}} \mid T_{-}(p)T(\pi, \overline{\pi})\Lambda_{+}(\overline{\pi}^{l_2})\Lambda_{+}(\pi^{l_1})Y_1^{l_1}Y_2^{l_2} = 
\end{equation*}
\begin{equation*}
    =p^2\sum_{l_1,l_2 \geq 1} \tilde{\psi}_{p^{l_1+l_2-1}}\mid T_{-}(p)\Lambda_{+}(\pi^{l_1+1})\Lambda_{+}(\overline{\pi}^{l_2-1})\Delta_{\overline{\pi}}Y_1^{l_1}Y_2^{l_2}=
\end{equation*}
\begin{equation*}
    =p^2\sum_{l_1\geq 1, l_2=1} + \textup{ }p^2\sum_{l_1\geq 1, l_2 \geq 2}.
\end{equation*}
For the first sum, we have
\begin{equation*}
    p^2\sum_{l_1\geq 1} \tilde{\psi}_{p^{l_1}}\mid T_{-}(p)\Lambda_{+}(\pi^{l_1+1})\Delta_{\overline{\pi}}Y_1^{l_1}Y_2 = p^4Y_2\sum_{l_1 \geq 1} \tilde{\psi}_{p^{l_1}} \mid \Lambda_{+}(\pi^{l_1})T(\overline{\pi})\Delta_p Y_1^{l_1} = 
\end{equation*}
\begin{equation*}
     = p^4Y_2\sum_{l_1 \geq 0} \tilde{\psi}_{p^{l_1}} \mid \Lambda_{+}(\pi^{l_1})T(\overline{\pi})\Delta_p Y_1^{l_1} - p^4Y_2\tilde{\psi}_1 \mid T(\overline{\pi}) \Delta_p =
\end{equation*}
\begin{equation*}
    = p^{2k-4}Y_2\left[\frac{\tilde{\psi}_1 \mid S_{\overline{\pi}}(Y_1)T(\overline{\pi})}{D_{\overline{\pi}, G}^{2}(Y_1)} - \tilde{\psi}_1 \mid T(\overline{\pi})\right].
\end{equation*}
For the second
\begin{equation*}
    p^2 \sum_{l_1 \geq 1, l_2 \geq 2} \tilde{\psi}_{p^{l_1+l_2-1}}\mid T_{-}(p)\Lambda_{+}(\pi^{l_1+1})\Lambda_{+}(\overline{\pi}^{l_2-1})\Delta_{\overline{\pi}}Y_1^{l_1}Y_2^{l_2} =
\end{equation*}
\begin{equation*}
    = p^5 \sum_{l_1 \geq 1, l_2 \geq 2} \tilde{\psi}_{p^{l_1+l_2-1}} \mid \Lambda_{+}(\overline{\pi})^{l_2-2}\Lambda_{+}(\pi^{l_1}) \Delta_{\overline{\pi}}\Delta_{p} T_{+}(p) Y_1^{l_1}Y_2^{l_2} =
\end{equation*}
\begin{equation*}
    = p^5 \sum_{l_1 \geq 1} \frac{\left[\tilde{\psi}_{p^{l_1+1}}-\tilde{\psi}_{p^{l_1}}\mid \Lambda_{-}(\overline{\pi})Y_2\right]\mid S_{\pi}(Y_2) \Lambda_{+}(\pi^{l_1})\Delta_{\overline{\pi}}\Delta_{p} T_{+}(p)Y_2^{2}Y_1^{l_1}}{D_{\pi, G}^{(2)}(Y_2)}.
\end{equation*}
But
\begin{equation*}
    \sum_{l_1 \geq 1} \tilde{\psi}_{p^{l_1+1}} \mid S_{\pi}(Y_2) \Lambda_{+}(\pi^{l_1}) \Delta_{\overline{\pi}}\Delta_p T_{+}(p) Y_2^2Y_1^{l_1} =
\end{equation*}
\begin{equation*}
    = \frac{\left[\tilde{\psi}_p-\tilde{\psi}_1 \mid \Lambda_{-}(\pi)Y_1\right] \mid S_{\overline{\pi}}(Y_1)S_{\pi}(Y_2)\Delta_{\overline{\pi}}\Delta_{p}T_{+}(p)Y_2^2}{D_{\overline{\pi}, G}^{(2)}(Y_1)} - \tilde{\psi}_p \mid S_{\pi}(Y_2)T_{+}(p)\Delta_{\overline{\pi}}\Delta_pY_2^2,
\end{equation*}
and
\begin{equation*}
    \sum_{l_1 \geq 1} \tilde{\psi}_{p^{l_1}} \mid \Lambda_{-}(\overline{\pi})Y_2 S_{\pi}(Y_2) \Lambda_{+}(\pi^{l_1})\Delta_{\overline{\pi}}\Delta_p T_{+}(p) Y_2^2Y_1^{l_1} =
\end{equation*}
\begin{equation*}
    = p^3\sum_{l_1 \geq 1} \tilde{\psi}_{p^{l_1}} \mid \Lambda_{+}(\pi^{l_1-1})Y_1^{l_1} S_{\pi}(Y_2) T_{+}(p)\Delta_{\overline{\pi}}\Delta_p^2 Y_2^{3} = 
\end{equation*}
\begin{equation*}
    =p^3\frac{\left[\tilde{\psi}_{p}-\tilde{\psi}_1 \mid \Lambda_{-}(\pi)Y_1\right] \mid S_{\overline{\pi}}(Y_1)S_{\pi}(Y_2)\Delta_{\overline{\pi}}\Delta_p^{2}T_{+}(p)Y_2^3Y_1}{D_{\overline{\pi}, G}^{(2)}(Y_1)}.
\end{equation*}
Hence, in total
\begin{equation*}
    \sum_{l_1,l_2 \geq 1} \tilde{\psi}_{p^{l_1+l_2-1}} \mid T_{-}(p)T(\pi, \overline{\pi})\Lambda_{+}(\overline{\pi}^{l_2})\Lambda_{+}(\pi^{l_1})Y_1^{l_1}Y_2^{l_2} =
\end{equation*}
\begin{equation*}
    = p^5\frac{\left[\tilde{\psi}_p-\tilde{\psi}_1 \mid \Lambda_{-}(\pi)Y_1\right] \mid S_{\overline{\pi}}(Y_1)S_{\pi}(Y_2)\Delta_{\overline{\pi}}\Delta_{p}T_{+}(p)Y_2^2}{D_{\overline{\pi}, G}^{(2)}(Y_1)D_{\pi, G}^{(2)}(Y_2)} - p^5\frac{\tilde{\psi}_p \mid S_{\pi}(Y_2)T_{+}(p)\Delta_{\overline{\pi}}\Delta_pY_2^2}{D_{\pi, G}^{(2)}(Y_2)} 
\end{equation*}
\begin{equation*}
    - p^8\frac{\left[\tilde{\psi}_{p}-\tilde{\psi}_1 \mid \Lambda_{-}(\pi)Y_1\right] \mid S_{\overline{\pi}}(Y_1)S_{\pi}(Y_2)\Delta_{\overline{\pi}}\Delta_p^{2}T_{+}(p)Y_2^3Y_1}{D_{\overline{\pi}, G}^{(2)}(Y_1)D_{\pi, G}^{(2)}(Y_2)} + p^{2k-4}Y_2\left[\frac{\tilde{\psi}_1 \mid S_{\overline{\pi}}(Y_1)T(\overline{\pi})}{D_{\overline{\pi}, G}^{2}(Y_1)} - \tilde{\psi}_1 \mid T(\overline{\pi})\right],
\end{equation*}
and the corresponding expression for $T(\overline{\pi},\pi)$.
\end{proof}
\subsection{Final expression for the Dirichlet series}
We recall that
 \begin{equation*}
     D_{F,G,h}^{(p)}(s) = D_{(\epsilon_1,\epsilon_2)}(s)+D_{(l_1,l_2)}(s) + D_{(\epsilon_1,l_2)}(s) +
 D_{(\epsilon_2, l_1)}(s) -
 \end{equation*}
 \begin{equation*}
 -D_{(\epsilon_1,\epsilon_2,l_1)}(s) -D_{(\epsilon_2,l_1,l_2)}(s)- D_{(\epsilon_1,l_1,l_2)}(s)-D_{(\epsilon_1,\epsilon_2,l_2)}(s)+D_{(\epsilon_1, \epsilon_2, l_1,l_2)}(s).
 \end{equation*}
 Now, if
 \begin{equation*}
    (\alpha_1-\alpha_2)D_{(\epsilon_1, \epsilon_2, l_2)}(s) = \alpha_1S_1(s)-\alpha_2S_2(s),
 \end{equation*}
 we have
 \begin{equation*}
     (\alpha_1-\alpha_2)D_{(\epsilon_1,l_2)}(s) = \frac{\alpha_1}{1-\alpha_1X}S_1(s) - \frac{\alpha_2}{1-\alpha_2X}S_2(s),
 \end{equation*}
 so 
 \begin{equation*}
      (\alpha_1-\alpha_2)[D_{(\epsilon_1,l_2)}(s) - D_{(\epsilon_1, \epsilon_2, l_2)}(s)] = \frac{\alpha_1^2X}{1-\alpha_1X}S_1(s) - \frac{\alpha_2^2X}{1-\alpha_2X}S_2(s).
 \end{equation*}
 We also recall that we have
 \begin{equation*}
     (\alpha_1-\alpha_2)D_{(\epsilon_1, \epsilon_2)}(s) = \alpha_1V_1(s) - \alpha_2V_2(s).
\end{equation*}
We can now state:
\begin{theorem}\label{Main Theorem, split case}
Let $2 \neq p = \pi \overline{\pi}$ be a split prime in $\mathcal{O}_K$. Let $F,G \in S_{2}^{k}$ and $h \in S_{1}^{k}$ be Hecke eigenforms, all having totally real Fourier coefficients, $h$ normalised, and $F$ belonging in the Maass space. Let also $\phi_1, \psi_1$ be the first Fourier-Jacobi coefficients of $F,G$ respectively and $X_i = \alpha_i p^{-(2k+s-4)}$, $Y_1 = \pi^{k}p^{-(2k+s-4)}$, $Y_2 = \overline{\pi}^{k}p^{-(2k+s-4)}$. We then have for $\textup{Re}(s)$ large enough 
\begin{equation*}
    (\alpha_1-\alpha_2)D_{F,G,h}^{(p)}(s) = 
    \frac{1}{Q_{p,G}^{(2)}(X_1)}\langle\tilde{\phi}_1, P(\alpha_2, \overline{s};\textup{ } G)\rangle_{\mathcal{A}} - \frac{1}{Q_{p,G}^{(2)}(X_2)}\langle\tilde{\phi}_1, P(\alpha_1, \overline{s};\textup{ } G)\rangle_{\mathcal{A}},
\end{equation*}
where (keeping in mind the conjugation because of the inner product)
\begin{equation*}
P(\alpha_i, s;\textup{ } G):=\alpha_iX_ip^{k-2}(1-p^{k-2}X_i)\left[(1+p^{3k-8}X_iY_1Y_2)\frac{\tilde{\psi}_1\mid S_{\pi}(Y_2)}{D_{\pi, G}^{(2)}(Y_2)}-Y_1\frac{\tilde{\psi}_1\mid S_{\pi}(Y_2)T(\pi)}{D_{\pi,G}^{(2)}(Y_2)} +\right.
\end{equation*}
\begin{equation*}
    +\left.(1+p^{3k-8}X_iY_1Y_2)\frac{\tilde{\psi}_1\mid S_{\overline{\pi}}(Y_1)}{D_{\overline{\pi}, G}^{(2)}(Y_1)}-Y_2\frac{\tilde{\psi}_1\mid S_{\overline{\pi}}(Y_1)T(\overline{\pi})}{D_{\overline{\pi},G}^{(2)}(Y_1)}\right]-
\end{equation*}
\begin{equation*}
    -\frac{1}{2}\alpha_iX_ip^{2k-5}Y_1Y_2(1-p^{k-2}X_i)^2\left[(1-p^{2k-5}Y_1Y_2)\frac{\left[\tilde{\psi}_p-\tilde{\psi}_1 \mid \Lambda_{-}(\pi)Y_1\right]\mid S_{\overline{\pi}}(Y_1)S_{\pi}(Y_2)T_{+}(p)}{D_{\overline{\pi},G}^{(2)}(Y_1)D_{\pi, G}^{(2)}(Y_2)}-\right. 
\end{equation*}
\begin{equation*}
    \left.-\frac{\tilde{\psi}_1\mid \Lambda_{-}(\overline{\pi})S_{\pi}(Y_2)T_{+}(p)Y_2}{D_{\pi, G}^{(2)}(Y_2)}+
    (1-p^{2k-5}Y_1Y_2)\frac{\left[\tilde{\psi}_p-\tilde{\psi}_1 \mid \Lambda_{-}(\overline{\pi})Y_2\right]\mid S_{\pi}(Y_2)S_{\overline{\pi}}(Y_1)T_{+}(p)}{D_{\pi,G}^{(2)}(Y_2)D_{\overline{\pi}, G}^{(2)}(Y_1)} - \right.
 \end{equation*}
 \begin{equation*}
\left.-\frac{\tilde{\psi}_1\mid \Lambda_{-}(\pi)S_{\overline{\pi}}(Y_1)T_{+}(p)Y_1}{D_{\overline{\pi}, G}^{(2)}(Y_1)}\right]+
\end{equation*}
\begin{equation*}
    +\alpha_i(1-p^{2k-5}Y_1Y_2)(1+p^{4k-9}Y_1Y_2X_i^2)(1-p^{k-2}X_i)^2\frac{\tilde{\psi}_1 \mid S_{\overline{\pi}}(Y_1)S_{\pi}(Y_2)}{D_{\overline{\pi}, G}^{(2)}(Y_1)D_{\pi, G}^{(2)}(Y_2)},
\end{equation*}
with $S_{\pi}, S_{\overline{\pi}}$ the polynomials defined in Proposition \ref{s polynomials} and $\Lambda_{-}(\pi), \Lambda_{-}(\overline{\pi}), T(\pi), T(\overline{\pi}), T_{+}(p)$ are operators defined in Subsection \ref{operators_split}. Also, $Q_{p}^{(2)}$ and $D_{\pi}^{(2)}, D_{\overline{\pi}}^{(2)}$ denote the $p$-factors of Gritsenko's and standard's $L$-function respectively, as in Definitions \ref{gritsenko l-function} and \ref{standard l-function} and $D_{F,G,h}^{(p)}$ is the $p$-factor of the Dirichlet series, as in equation \eqref{p-part, split}.
\end{theorem}
\begin{proof}
We observe that both the left and right hand side of the claimed equation in the Theorem are holomorphic functions in $s$ for $\textup{Re}(s)$ large enough. Hence, it is enough to prove the equality for $s \in \mathbb{R}$ (see also the Remark \ref{identity_theorem} before Proposition \ref{easy_part_dirichlet_inert}). But then this follows by putting together the results of the last three Subsections.
\end{proof}
We finally have the following Proposition about the relation of $S_{\overline{\pi}}(Y_1)S_{\pi}(Y_2)$ with known $L$-functions.
\begin{proposition}
Assume $2\neq p = \pi \overline{\pi}$ is a split prime in $\mathcal{O}_K$. We have
\begin{equation*}
    S_{\overline{\pi},F}(Y_1)S_{\pi, F}(Y_2) = L_p(s+k-2, f)L_p\left(s+k-2, f, \left(\frac{-4}{p}\right)\right),
\end{equation*}
where $f \in S_{k-1}\left(\Gamma_{0}(4), \left(\frac{-4}{\cdot}\right)\right)$ is the modular form whose Maass lift is $F$, as in Proposition \ref{maass_lift}.
\begin{proof}
    Let us first consider $S_{\pi, F}(Y_2)$. By assuming that (here, $\mid_{k-1}$ is the usual $\textup{GL}_2$-action)
\begin{equation*}
    f \mid_{k-1} T(p) = a(p)f,
\end{equation*}
and using \cite[Lemma 3.3]{gritsenko_maass}, we obtain that
\begin{equation*}
    \tilde{\phi}_1 \mid_{k} T(\overline{\pi}) = p^{k-2}(\overline{\pi})^{-k}a(p)\tilde{\phi}_1.
\end{equation*}
Using now the fact that $Y_2 = \overline{\pi}^{k}p^{-(2k+s-4)}$ and that
\begin{equation*}
    S_{\pi}(Y_2) = 1 - T(\overline{\pi})Y_2 + p\Delta_{\overline{\pi}}T(\pi, \overline{\pi})Y_2^2,
\end{equation*}
we get
\begin{equation*}
    S_{\pi, F}(Y_2) = 1 - p^{-k-s+2}a(p)+p^{-k-2s+2} = L_{p}(s+k-2, f),
\end{equation*}
and similarly for $S_{\overline{\pi}}(Y_1)$. Given that $\left(\frac{-4}{p}\right) = 1$ in this case, the result follows. 
\end{proof}
\end{proposition}
\section{Euler Product}
We can now use the above calculations in order to deduce the following Theorem:
\begin{theorem}\label{Main Theorem, Euler Product}
Assume $F,G,h$ satisfy the same assumptions as in the beginning of Subsection \ref{dirichlet series inert}, with $\psi_1 \not \equiv 0$. We then have that the series $D_{F,G,h}(s)$ of Theorem \ref{integral_representation_theorem} has an Euler product of the form
\begin{equation*}
    D_{F,G,h}(s) = 4\beta_k \langle \tilde{\phi}_1, \tilde{\psi}_1\rangle_{\mathcal{A}}\prod_{p \textup{ prime}}\frac{D_{F,G,h}^{(p)}(s)}{\langle \tilde{\phi}_1, \tilde{\psi}_1 \rangle_\mathcal{A}},
\end{equation*}
where $D_{F,G,h}^{(p)}(s)$ has been defined in equations \eqref{p-part dirichlet} and \eqref{p-part, split} for $p\neq 2$ and for $p=2$, we define
\begin{equation*}
    D_{F,G,h}^{(2)}(s) := \sum_{l, \epsilon, m \geq 0}\langle \tilde{\phi}_1\mid T_{-}(2^m)U_{\pi^{l}}, \tilde{\psi}_{2^{m+l}}\rangle_{\mathcal{A}}a_{2^{m+\epsilon}}2^{-sl}2^{-(k+s-1)\epsilon}2^{-(2k+s-4)m},
\end{equation*}
with $\pi := (1+i)$, together with the condition $\textup{min}(l, \epsilon)=0$.
\end{theorem}

The proof of this Theorem is the subject matter of this 
Section. We first need to define some elements of the global Hecke ring $H^{1,1}$. Let $m \geq 1$ and $l \in \mathcal{O}_K$. We then define
\begin{equation*}
    T_{-}(m) := j_{-} (T(m)), \textup{ }\Lambda_{-}(l) := j_{-}\left(\Gamma_{1}\textup{diag}(l, l)\Gamma_1\right),
\end{equation*}
where $j_{-}$ is the embedding of equation \eqref{+- embeddings}. Here, $T(m)$ is the standard Hecke element in $H^{1}$, as in Definition \ref{maass_defn}. We then observe that 
\begin{equation}\label{multiplicative}
    T_{-}(m_1m_2) = T_{-}(m_1)T_{-}(m_2), \textup{ }\Lambda_{-}(l_1l_2) = \Lambda_{-}(l_1)\Lambda_{-}(l_2)
\end{equation}
when $m_1,m_2 \in \mathbb{N}$ and $l_1,l_2 \in \mathcal{O}_K$ are co-prime. This follows from the corresponding statements for $H^{1}$ and the fact that the $j_{-}$ embedding is a ring homomorphism. We also claim that these elements commute with our known Hecke elements when we allow co-prime arguments.
\begin{lemma}\label{different primes}
    Let $p \neq 2$ be any rational prime. Assume that $m \in \mathbb{N}$ and $l \in \mathcal{O}_K$ are co-prime to $p$. Then, the elements $T_{-}(m)$ and $\Lambda_{-}(l)$ commute with all the elements listed in Subsection \ref{hecke inert} (if $p$ is inert) and all the elements listed in Subsection \ref{operators_split} (if $p$ splits).
\end{lemma}
\begin{proof}
    The proof is done case by case. By the multiplicative property of equation \eqref{multiplicative}, it suffices to consider $m, l$ prime elements, co-prime to $p$. Assume first $p$ is inert. Let then $X$ be either $T_{-}$ or $\Lambda_{-}$ with the corresponding argument being prime co-prime to $p$. By \cite[Lemma 3.8]{gritsenko}, we have
    \begin{equation}\label{comm}
        \epsilon(T_{1,p})X = X\epsilon(T_{1,p}), \textup{ }\epsilon(T_p)X = X\epsilon(T_p). 
    \end{equation}
    The first equation now gives (from the proof of Proposition \ref{prop: rankin_prime})
    \begin{equation*}
        \left(T^{J}(p) + \Lambda_{-}(p)+\Lambda_{+}(p) + \nabla_{p} - \Delta_p\right)X = X\left(T^{J}(p) + \Lambda_{-}(p)+\Lambda_{+}(p) + \nabla_{p} - \Delta_p\right).
    \end{equation*}
    By then looking at the different signatures of the elements (see Definition \ref{signature}) and using \cite[Proposition 3.3]{gritsenko} or \cite[Section 3.3]{heim}, we obtain the relations
    \begin{align*}
        \Lambda_{-}(p)X &= X\Lambda_{-}(p).\\
        \Lambda_{+}(p)X &= X\Lambda_{+}(p).\\
        \left(T^{J}(p)+(\nabla_p-\Delta_p)\right)X &= X\left(T^{J}(p)+(\nabla_p-\Delta_p)\right).
    \end{align*}
    Now, $X$ commutes with $\Delta_p$ and we can show commutativity with $\nabla_p$ using coset decompositions. Then, commutativity with $T^{J}(p)$ follows from the third equation above. Finally, commutativity with $T_{+}(p), T_{-}(p)$ follows from the second equation in \ref{comm}, as $\epsilon(T_p) = T_{+}(p)+T_{-}(p)$.\\\\
    For the split case, we proceed similarly, using the embeddings of the standard elements $T_{\pi}, T_{\overline{\pi}}$ and $T_p$ (which follow from Proposition \ref{embedding}). The only relation we do not obtain immediately is the commutativity with each of $T(\pi, \overline{\pi})$ and $T(\overline{\pi}, \pi)$. Instead, we get the commutativity with their sum (from the $\epsilon$-embedding of $T_p$). But, from Table \ref{table:1}, we have $\Lambda_{-}(\pi)\Lambda_{+}(\pi) = p\Delta_{\pi}T(\pi, \overline{\pi})$ and then commutativity follows from the commutativity of $X$ with the $\Lambda_{\pm}(\pi)$ elements and the fact that $\Delta_{\pi}$ is a unit in $H^{1,1}$. 
\end{proof}
Let us now focus on the proof of Theorem \ref{Main Theorem, Euler Product}. We need to distinguish cases when $p$ is inert or splits in $\mathbb{Z}[i]$. We have the following two Propositions, the proof of which is essentially the same.
\begin{proposition}
Let $p$ be an inert prime. Let $m' \in \mathbb{N}$ and $l', \epsilon' \in \mathbb{Z}[i]$ all relative prime to $p$. Then, we claim
\begin{equation*}
    \sum_{\substack{l, \epsilon, m \geq 0\\\textup{min}(l,\epsilon)=0}}\langle \tilde{\phi}_1\mid T_{-}(m'p^{m})\Lambda_{-}(l'p^l), \tilde{\psi}_{m'N(l')p^{m+2l}}\rangle_{\mathcal{A}}a_{m'N(\epsilon')p^{m+2\epsilon}}p^{-(3k+2s-8)l}p^{-2(k+s-1)\epsilon}p^{-(2k+s-4)m} = 
\end{equation*}
\begin{equation*}
    = \langle\tilde{\phi}_1\mid T_{-}(m')\Lambda_{-}(l'), \tilde{\psi}_{m'N(l')}\rangle_{\mathcal{A}}a_{m'N(\epsilon')}\left(\frac{D_{F,G,h}^{(p)}(s)}{\langle \tilde{\phi}_1, \tilde{\psi}_1 \rangle_\mathcal{A}}\right).
\end{equation*}
\begin{proposition}
Let $2 \neq p=\pi\overline{\pi}$ be a prime that splits in $\mathbb{Z}[i]$. Let $m' \in \mathbb{N}$ and $l', \epsilon' \in \mathbb{Z}[i]$ all relative prime to $p$ (or equivalently coprime to both $\pi, \overline{\pi}$). Then, we claim
\begin{multline*}
\sum_{\substack{l_1,l_2,\\ \epsilon_1, \epsilon_2,m \geq 0\\ \textup{min}(l_i,\epsilon_i)=0}}\langle \tilde{\phi}_1 \mid T_{-}(m'p^{m})\Lambda_{-}(l'\pi^{l_1}\overline{\pi}^{l_2}), \tilde{\psi}_{m'N(l')p^{m+l_1+l_2}}\rangle_{\mathcal{A}}a_{m'N(\epsilon')p^{m+\epsilon_1+\epsilon_2}} p^{(4-2k)l_1} p^{(4-2k)l_2} \pi^{l_1 k} \overline{\pi}^{l_2k}\times\\\times p^{-s(l_1+l_2)}p^{-(k+s-1)(\epsilon_1+\epsilon_2)}p^{-(2k+s-4)m}
 \end{multline*}
\begin{equation*}
    =\langle\tilde{\phi}_1\mid T_{-}(m')\Lambda_{-}(l'), \tilde{\psi}_{m'N(l')}\rangle_{\mathcal{A}}a_{m'N(\epsilon')}\left(\frac{D_{F,G,h}^{(p)}(s)}{\langle \tilde{\phi}_1, \tilde{\psi}_1 \rangle_\mathcal{A}}\right).
\end{equation*}
\end{proposition}
\end{proposition}
\begin{proof}
The proof is analogous to the proof of the results in Sections \ref{inert primes} and \ref{split primes}. By the multiplicative property of equation \eqref{multiplicative}, we rewrite the sum in the inert case as 
\begin{multline*}
    a_{m'N(\epsilon')}\sum_{l, \epsilon, m \geq 0}\langle \tilde{\phi}_1\mid T_{-}(m')\Lambda_{-}(l'), \tilde{\psi}_{m'N(l')p^{m+2l}} \mid T_{+}(p^{m})\Lambda_{+}(p^l)\rangle_{\mathcal{A}}a_{p^{m+2\epsilon}}p^{-(3k+2s-8)l}\times\\\times p^{-2(k+s-1)\epsilon}p^{-(2k+s-4)m},
\end{multline*}
using the multiplicativity property of the Fourier coefficients of $h$ as well. Similarly, we rewrite the sum for the split case in an analogous way.\\

We can now apply the rationality Propositions, as in Sections \ref{inert primes} and \ref{split primes}. The difference is that every time we previously had the term $\tilde{\psi}_1$, we will instead now have $\tilde{\psi}_{m'N(l')}$. This follows from fact that $m'N(l')$ is co-prime to $p$ (so terms of the form $\tilde{\psi}_{m'N(l')/p}$ vanish). Similarly, we now have terms of the form $\widetilde{\psi}_{pm'N(l')} \mid T_{+}(p)$ instead of the terms $\widetilde{\psi}_p \mid T_{+}(p)$.\\

By the calculations leading to Theorems \ref{Main Theorem, inert case} and \ref{Main Theorem, split case}, we claim that the expressions involving the Fourier-Jacobi coefficients of $G$ (i.e. before taking the inner product with $\widetilde{\phi}_1 \mid T_{-}(m')\Lambda_{-}(l')$), can be written in the form $\widetilde{\psi}_{m'N(l')} \mid R(Y,Y_1,Y_2,X_1,X_2)$, where $R$ is a polynomial with coefficients involving the operators $T^{J}(p), \textup{ } T(\pi), \textup{ }T(\overline{\pi}), \textup{ }T(\pi, \overline{\pi}), \textup{ }T(\overline{\pi},\pi)$ and is independent of $m', l'$.
\begin{flushleft}
Let us first deal with the inert case. The only expressions that are not into the form claimed above, are these of the form $\tilde{\psi}_{pm'N(l')} \mid T_{+}(p)$ (see Proposition \ref{prop4}). But, we can write
\begin{equation*}
    \tilde{\psi}_{pm'N(l')} \mid T_{+}(p) = \lambda_p \tilde{\psi}_{m'N(l')},
\end{equation*} 
where $\lambda_p$ is the eigenvalue of the operator $T_p \in H^{2}_p$, when it acts on $G$, i.e. $G \mid_kT_p = \lambda_{p}G$. This is true because of the embedding
\begin{equation*}
    \epsilon(T_p) = T_{+}(p)+T_{-}(p),
\end{equation*}
as in the proof of Proposition \ref{prop: rankin_prime}. So, we get
\begin{equation*}
    \tilde{\psi}_{m'N(l')} \mid\mid T_p = \lambda_p \tilde{\psi}_{m'N(l')},
\end{equation*}
and
\begin{equation*}
    \tilde{\psi}_{m'N(l')} \mid\mid T_p = \tilde{\psi}_{pm'N(l')} \mid T_{+}(p) + \tilde{\psi}_{m'N(l')/p} \mid T_{-}(p) = \tilde{\psi}_{pm'N(l')} \mid T_{+}(p).
\end{equation*}
\end{flushleft}
The same can be said for the split case as well. From Theorem \ref{Main Theorem, split case}, we will have terms of the form 
\begin{equation*}
\tilde{\psi}_{pm'N(l')} \mid S_{\overline{\pi}}(Y_1)S_{\pi}(Y_2)T_{+}(p), \textup{ }\tilde{\psi}_{m'N(l')}\mid \Lambda_{-}(\pi)S_{\overline{\pi}}(Y_1)S_{\pi}(Y_2)T_{+}(p), \textup{ }\tilde{\psi}_{m'N(l')} \mid \Lambda_{-}(\overline{\pi})S_{\pi}(Y_2)T_{+}(p),
\end{equation*}
(and the corresponding expressions for $\overline{\pi}$). But, from Table \ref{table:1}, we have the relations:
\begin{itemize}
    \item $\begin{aligned}[t]S_{\overline{\pi}}(Y_1)S_{\pi}(Y_2)T_{+}(p) &= T_{+}(p)\left[1-T(\pi)Y_1+p^3\Delta_{p}Y_1Y_2\right]\left[1-T(\overline{\pi})Y_2\right] + \\&+\Lambda_{+}(\pi)\left[T(\overline{\pi},\pi)Y_1 - p^2\Delta_{\overline{\pi}}Y_2\right]\left[1-T(\overline{\pi})Y_2\right]-\\&-p^2\Delta_{\pi}Y_1\Lambda_{+}(\overline{\pi})\left[1-T(\pi)Y_1+p^3\Delta_{p}Y_1Y_2\right]\left[1-T(\overline{\pi})Y_2\right]+\Lambda_{+}(\overline{\pi})S_{\overline{\pi}}(Y_1)Y_2.\end{aligned}$
    \item $\begin{aligned}[t]\Lambda_{-}(\pi)S_{\overline{\pi}}(Y_1)S_{\pi}(Y_2)T_{+}(p) &= p^2\Delta_{\pi}T(\overline{\pi})\left[1-T(\pi)Y_1+p^3\Delta_{p}Y_1Y_2\right]\left[1-T(\overline{\pi})Y_2\right] + \\&+p\Delta_{\pi}T(\pi,\overline{\pi})\left[T(\overline{\pi},\pi)Y_1 - p^2\Delta_{\overline{\pi}}Y_2\right]\left[1-T(\overline{\pi})Y_2\right]-\\&-p^5\Delta_{\pi}\Delta_pY_1\left[1-T(\pi)Y_1+p^3\Delta_{p}Y_1Y_2\right]\left[1-T(\overline{\pi})Y_2\right]+p^3\Delta_pS_{\overline{\pi}}(Y_1)Y_2.\end{aligned}$ 
    \item $\begin{aligned}[t]\Lambda_{-}(\overline{\pi})S_{\pi}(Y_2)T_{+}(p) &= p^2\Delta_{\overline{\pi}}T(\pi) - \left[p^5\Delta_{\overline{\pi}}\Delta_p - p\Delta_{\overline{\pi}}T(\pi, \overline{\pi})T(\overline{\pi}, \pi) + p^2\Delta_{\overline{\pi}}T(\pi)T(\overline{\pi})\right]Y_2 +\\&+ p^5\Delta_{\overline{\pi}}\Delta_pT(\overline{\pi})Y_2^2.\end{aligned}$
\end{itemize}
Now, from Proposition \ref{embedding}, we have
\begin{equation*}
    \epsilon(T_p) = T_{-}(p)+T_{+}(p)+T(\pi,\overline{\pi})+T(\overline{\pi},\pi),
\end{equation*}
and so we get
\begin{align*}
    \tilde{\psi}_{m'N(l')} \mid \mid T_p &= \tilde{\psi}_{m'N(l')} \mid \mid \left(T_{+}(p)+T_{-}(p)+T(\pi,\overline{\pi})+T(\overline{\pi},\pi)\right)\\
    &= \tilde{\psi}_{pm'N(l')} \mid T_{+}(p) + 0 + \tilde{\psi}_{m'N(l')} \mid (T(\pi,\overline{\pi})+T(\overline{\pi},\pi))\\
    &= \tilde{\psi}_{pm'N(l')} \mid T_{+}(p) + \tilde{\psi}_{m'N(l')} \mid 
    (T(\pi,\overline{\pi})+T(\overline{\pi},\pi)).
\end{align*}
But $\tilde{\psi}_{m'N(l')} \mid \mid T_p = \lambda_p \tilde{\psi}_{m'N(l')}$, where $\lambda_p$ is the eigenvalue of $G$ corresponding to $T_p$, and so we obtain
\begin{equation*}
    \tilde{\psi}_{pm'N(l')} \mid T_{+}(p) = \lambda_p\tilde{\psi}_{m'N(l')} -\tilde{\psi}_{m'N(l')} \mid (T(\pi,\overline{\pi})+T(\overline{\pi},\pi)).
\end{equation*}
Moreover, again from Proposition \ref{embedding}, we have
\begin{equation*}
    \epsilon(T_{\overline{\pi}}) = \Lambda_{-}(\overline{\pi})+T(\overline{\pi})+\Lambda_{+}(\overline{\pi}).
\end{equation*}
Hence, by a similar argument as above, we get
\begin{equation*}
    \tilde{\psi}_{pm'N(l')}\mid \Lambda_{+}(\overline{\pi}) = \lambda_{T_{\overline{\pi}}}\tilde{\psi}_{m'N(l')} - \tilde{\psi}_{m'N(l')}\mid T(\overline{\pi}),
\end{equation*}
and similarly for $\Lambda_{+}(\pi)$.\\

In particular, our claim now follows for both the inert and split case and therefore, the expression involving the Fourier-Jacobi coefficients of $G$ can be written in the form $\widetilde{\psi}_{m'N(l')} \mid R$, where $R = R(Y,Y_1,Y_2,X_1,X_2)$ is a polynomial with coefficients involving the operators $T^{J}(p), \textup{ }T(\pi), \textup{ }T(\overline{\pi}), \textup{ }T(\pi, \overline{\pi}
), \textup{ }T(\overline{\pi},\pi)$. These are all self-adjoint operators (see \cite[Lemma 4.3]{gritsenko}). Moreover, since $F$ is in the Maass space, from \cite[Theorem, p. 2911]{gritsenko}, $\tilde{\phi}_1$ is an eigenform for these operators, as these all have signature $1$. By now writing $R_{F}$ for the polynomial obtained when we substitute the eigenvalues of $\tilde{\phi}_1$ with respect to the above operators and using the commutativity of Lemma \ref{different primes}, we can write
\begin{align*}
    \langle \tilde{\phi}_1 \mid T_{-}(m')\Lambda_{-}(l'), \tilde{\psi}_{m'N(l')} \mid R\rangle_{\mathcal{A}} 
    &= \langle \tilde{\phi}_1 \mid \overline{R} T_{-}(m')\Lambda_{-}(l'), \tilde{\psi}_{m'N(l')}\rangle_{\mathcal{A}} \\
    &=\overline{R_{F}}\langle \tilde{\phi}_1 \mid T_{-}(m')\Lambda_{-}(l'), \tilde{\psi}_{m'N(l')} \rangle_{\mathcal{A}}\\
    &=\frac{1}{\langle \tilde{\phi}_1, \tilde{\psi}_1 \rangle}\langle \tilde{\phi}_1 \mid T_{-}(m')\Lambda_{-}(l'), \tilde{\psi}_{m'N(l')} \rangle_{\mathcal{A}} \overline{R_{F}}\langle \tilde{\phi}_1, \tilde{\psi}_1 \rangle_{\mathcal{A}}\\
    &=\frac{1}{\langle \tilde{\phi}_1, \tilde{\psi}_1 \rangle}\langle \tilde{\phi}_1 \mid T_{-}(m')\Lambda_{-}(l'), \tilde{\psi}_{m'N(l')} \rangle_{\mathcal{A}} \langle \tilde{\phi}_1, \tilde{\psi}_1 \mid R \rangle_{\mathcal{A}},
\end{align*}
where $\overline{R_{F}}$ is the polynomial obtained by taking the complex conjugate. The result now follows by comparing with the initial expression for $D_{F, G,h}^{(p)}(s)$, as the rightmost term is what we have originally (i.e., for $m'=l'=1$).
\end{proof}
The proof of Theorem \ref{Main Theorem, Euler Product} now follows from the above two Propositions by working prime by prime and factoring from the initial Dirichlet series the corresponding expression for each prime.
\renewcommand{\abstractname}{Acknowledgements}
\begin{abstract}
The second author was supported by the HIMR/UKRI "Additional Funding Programme for Mathematical Sciences", grant number EP/V521917/1 as well as by a scholarship by Onassis Foundation. The authors would also like to thank the anonymous referee, whose remarks contributed significantly in the improvement of the initial manuscript.
\end{abstract}
\printbibliography
\end{document}